\documentclass[a4paper,10pt,oneside,reqno,fleqn]{amsart}
\usepackage[utf8]{inputenc}
\usepackage{amsthm}
\usepackage{amsmath}
\usepackage{bm}
\usepackage{bbm}
\usepackage{amsfonts}
\usepackage{amssymb}
\usepackage{mathtools}
\usepackage{appendix}
\usepackage{mathrsfs}
\usepackage{setspace}
\usepackage{todonotes}
\usepackage{thmtools} 
\usepackage[foot]{amsaddr}
\usepackage[pdfdisplaydoctitle,colorlinks,breaklinks,urlcolor=blue,linkcolor=blue,citecolor=blue]{hyperref} 
\usepackage[nameinlink,capitalise]{cleveref}

\hypersetup{pdftitle={Kraichnan for gSQG}}

\newcommand{\E}{\mathbb{E}}

\newcommand{\N}{\mathbb{N}}

\renewcommand{\P}{\mathbb{P}}

\newcommand{\R}{\mathbb{R}}

\newcommand{\T}{\mathbb{T}}

\newcommand{\cA}{\mathcal{A}}
\newcommand{\cB}{\mathcal{B}}

\newcommand{\cF}{\mathcal{F}}

\newcommand{\cP}{\mathcal{P}}

\newcommand{\cS}{\mathcal{S}}

\newcommand{\fkp}{\mathfrak{p}}

\DeclareMathOperator{\trace}{Tr}
\DeclareMathOperator{\id}{Id}

\DeclareMathOperator{\e}{e}
\let\div\relax
\DeclareMathOperator{\div}{div}

\newcommand{\dd}{\,\mathrm{d}}
\newcommand{\eps}{\varepsilon}
\renewcommand{\setminus}{\smallsetminus}
\newcommand{\loc}{\mathrm{loc}}

\newtheorem{theorem}{Theorem}[section]
\newtheorem{definition}[theorem]{Definition}

\newtheorem{corollary}[theorem]{Corollary}
\newtheorem{lemma}[theorem]{Lemma}
\newtheorem{proposition}[theorem]{Proposition}

\theoremstyle{remark}
\newtheorem{remark}[theorem]{Remark}

\numberwithin{equation}{section}

\renewcommand{\leq}{\leqslant}
\renewcommand{\le}{\leqslant}
\renewcommand{\geq}{\geqslant}

\title[Regularization by rough Kraichnan for gSQG]{Regularization by rough Kraichnan noise for the generalised SQG equations}
\author[M. Bagnara]{Marco Bagnara}
\address[M. Bagnara]{Scuola Normale Superiore, Piazza dei Cavalieri 7, 56126 Pisa, Italy.}
\email{marco.bagnara@sns.it}
\author[L. Galeati]{Lucio Galeati}
\address[L. Galeati]{EPFL, Bâtiment MA, Station 8, 1015 Lausanne, Switzerland.}
\email{lucio.galeati@epfl.ch}
\author[M. Maurelli]{Mario Maurelli}
\address[M. Maurelli]{Dipartimento di Matematica, Università di Pisa, Largo Bruno Pontecorvo 5, 56127 Pisa, Italy.}
\email{mario.maurelli@unipi.it}
\date\today

\begin{document}
\begin{abstract}
    We consider the generalised Surface Quasi-Geostrophic (gSQG) equations in $\R^2$ with parameter $\beta\in (0,1)$, an active scalar model interpolating between SQG ($\beta=1$) and the 2D Euler equations ($\beta=0$) in vorticity form.
    Existence of weak $(L^1\cap L^p)$-valued solutions in the deterministic setting is known, but their uniqueness is open.
    We show that the addition of a rough Stratonovich transport noise of Kraichnan type regularizes the PDE, providing strong existence and pathwise uniqueness of solutions for initial data $\theta_0\in L^1\cap L^p$, for suitable values $p\in[2,\infty]$ related to the regularity degree $\alpha$ of the noise and the singularity degree $\beta$ of the velocity field; in particular, we can cover any $\beta\in (0,1)$ for suitable $\alpha$ and $p$ and we can reach a suitable (``critical'') threshold. The result also holds in the presence of external forcing $f\in L^1_t (L^1\cap L^p)$ and solutions are shown to depend continuously on the data of the problem; furthermore, they are well approximated by vanishing viscosity and regular approximations. With similar techniques, we also show well-posedness for two-dimensional linear transport equation with random drift, with the same noise.
\end{abstract}

\maketitle



\section{Introduction}\label{sec:introduction} 

\subsection{Preliminaries and main result}

Consider the generalised Surface Quasi-Geostrophic equations (gSQG for short), a family of $2$D active scalar PDEs on $\R^2$ indexed by a parameter $\beta\in (0,2)$, of the form
\begin{equation}\label{eq:gSQG}\begin{cases}
     \partial_t \theta + u\cdot\nabla \theta = f,\\
     u = -\nabla^\perp \Lambda^{-2+\beta}\theta.
\end{cases}\end{equation}
where $\Lambda=(-\Delta)^{1/2}=|\nabla|$.
The gSQG equations were introduced in \cite{CCCGW2012,ChCoWu2011}, as a family of PDEs interpolating between the $2$D Euler in vorticity form, corresponding to $\beta=0$, and SQG, corresponding to $\beta=1$.
Recalling the Riesz kernel representation of $\Lambda^{-2+\beta}$, 
we may express the velocity $u$ in function of $\theta$ by the nonlocal relation
\begin{equation}\label{eq:generalised_Biot_Savart}
	u = -\nabla^\perp \Lambda^{-2+\beta}\theta= K_\beta\ast \theta, \quad
	K_\beta(x) = c_\beta \frac{x^\perp}{|x|^{2+\beta}}
\end{equation}
so that \eqref{eq:gSQG} can be regarded as a closed equation in the variable $\theta$.
To this day, we lack a satisfactory solution theory for \eqref{eq:gSQG}; let us shortly mention some key results in different regularity classes (see Section \ref{subsec:literature} for more details):
\begin{enumerate}
    \item[i)] For sufficiently regular initial conditions $\theta_0$, local existence of regular solutions holds \cite{CCCGW2012}; however, instantaneous norm inflation may happen in positive regularity classes $H^s$ or $C^\gamma$ \cite{JeoKim2024,cordoba2022,choi2024well}, to the point where even non-existence results can be established. See also the recent work \cite{cordoba2025strong}.
    \item[ii)] For $L^1\cap L^p$-valued initial conditions $\theta_0$, thanks to the transport structure of the PDE and the divergence-free property of $u$, global existence of $L^p$-valued solutions can be established by classical compactness arguments, for instance for $p \in [2,\infty]$, cf. \cite{Marchand2008,CCCGW2012,Nguyen2018}. Their uniqueness is an open problem (except for $\beta=0$ and $p=\infty$, see below).
    \item[iii)] Convex integrations schemes yield non-uniqueness of very weak solutions $\theta$ in negative H\"older spaces, cf. \cite[Section 3.3]{ma2022some} and more recently \cite{zhao2024onsager}.
    In this regard, let us also mention the work \cite{brue2024flexibility}, proving non-uniqueness of $L^p$-valued solutions to the unforced Euler equations, for $p>1$ small enough. 
\end{enumerate}

In particular, the situation for $\beta\in (0,1]$ is drastically different from the Euler case, where we have well-posedness for initial conditions $\theta_0\in L^1\cap L^\infty$ and propagation of higher regularity thanks to the classical results of Wolibner and Yudovich.

Motivated by the theoretical study of turbulent fluids, we would like to understand whether the introduction of a transport noise term may affect the resulting solution theory for gSQG. In particular, we introduce the SPDE
\begin{equation}\label{eq:stochastic_gSQG_intro}
    \dd \theta + (K_\beta\ast \theta)\cdot\nabla \theta \dd t + \circ \dd W\cdot\nabla \theta = f \dd t.
\end{equation}
Here $W=W(t,x)$ is a divergence-free Gaussian velocity field, corresponding to the Kraichnan model of turbulence \cite{Kraichnan1968,kraichnan1994anomalous}, with covariance function (formally) given by
\begin{align*}
    \E[ W(t,x)\otimes W(s,y)] = (t\wedge s)\, Q(x-y),\quad \widehat{Q}(n) \sim (1+|n|^2)^{-1-\alpha} \left(I-\frac{n\otimes n}{|n|^2}\right);
\end{align*}
in particular, $W$ is Brownian in time, coloured but rough in space (more precisely, $W$ is $(\alpha-\eps)$-H\"older continuous in space for any $\eps>0$); $\circ \dd W$ in \eqref{eq:stochastic_gSQG_intro} denotes Stratonovich integration, which is the correct physical choice in view of the Wong--Zakai principle.
 
From the Lagrangian viewpoint, $W$ is an external turbulent fluid in which all the particles are immersed; although its introduction is somewhat artificial, the theoretical interest in SPDEs of fluid dynamics in the style of \eqref{eq:stochastic_gSQG_intro} comes from the idea that $W$ is just a proxy for the small, turbulent scales of the velocity $u$ itself, in a regime of fully developed turbulence.
In particular, such an exhogeneous noise is expected to reproduce and highlight features of intrinsically stochastic turbulent solutions $\theta$ to the original PDE \eqref{eq:gSQG}. The latter might generically display self-stabilizing and self-regularizing properties, thanks to anomalous dissipation of kinetic energy happening at small scales.
In this sense, another main reason for studing \eqref{eq:stochastic_gSQG_intro} comes from the investigation of regularization by noise phenomena in fluids \cite{flandoli} and their connection to turbulence.

With this goal in mind, it is convenient to consider an additional forcing term $f$ on the r.h.s.\ \eqref{eq:stochastic_gSQG_intro}. Indeed, already in the Euler case $\beta=0$, it was recently shown by Vishik (see \cite{vishik2018a,vishik2018b} and the monograph \cite{ABCDGJK2024}) that a carefully chosen forcing $f$, acting as an unstable background, may produce non-uniqueness of $L^p$-valued solutions.
After the completion of this work, Vishik's ideas have been readapted in \cite{castro2025unstable} to cover the whole range $\beta\in [0,1]$, including the SQG case $\beta=1$; therein, the authors show non-uniqueness of weak $W^{s,p}$-valued solutions to the forced gSQG equations, in the supercritical regime $s<\beta+2/p$. Here, we would like to understand whether a sufficiently turbulent background noise is additionally able to prevent this kind of non-uniqueness scenario.

Our results, building on the techniques first developed by one of the authors in \cite{coghi2023existence} for the $2$D Euler equations, answer positively to the above questions. 
Here is a statement summarizing our main findings.

\begin{theorem}\label{thm:main}
    Let $\alpha$, $\beta$ and $p$ be parameters satisfying
    \begin{equation}\label{eq:parameters_uniqueness_intro}
		0<\frac{\beta}{2} < \alpha< \frac{1}{2}, \quad
		\frac{\beta}{2}+\alpha \leq 1-\frac{1}{p}, \quad p\in [2,\infty).
    \end{equation}
    Then, for any initial condition $\theta_0\in L^1\cap L^p$ and any $f\in L^1_\loc([0,+\infty);L^1\cap L^p)$, there exists a global probabilistically strong, analytically weak solution $\theta$ to \eqref{eq:stochastic_gSQG_intro}, with the property that
    \begin{equation}\label{eq:uniqueness_class}
        \theta\in L^\infty\big(\Omega; L^\infty([0,T];L^1\cap L^p)\big) \quad \forall\, T<\infty.
    \end{equation}
    Furthermore, pathwise uniqueness and uniqueness in law holds in the class of solutions $\theta$ satisfying \eqref{eq:uniqueness_class}. 
 
    Moreover, solutions depend continuously on the data of the problem: if $(\theta^n_0, f^n)$ is a sequence of data such that $\theta^n_0\to \theta_0$ in $L^1\cap L^p$, $\{f^n\}_n$ is bounded in $L^1_\loc([0,+\infty);L^1\cap L^p)$ and $f^n\to f$ in $L^1_\loc([0,+\infty);\dot H^{\beta/2-1})$, then the corresponding solutions $\theta^n$ converge in probability to $\theta$. More precisely, for any $T<\infty$, any $m\in\N$ and any $\delta\in (0,1/2)$ it holds
    \begin{equation}\label{eq:intro_convergence}
        \lim_{n\to\infty} \E\Big[ \sup_{t\in [0,T]} \| \theta^n_t-\theta_t\|_{\dot{H}^{-\delta}}^m\Big] = 0.
    \end{equation}
\end{theorem}

\begin{remark}
    The convergence statement also holds if $f^n\to f$ in $L^1_{loc}([0,+\infty);L^1\cap L^2)$, by the embedding $L^1\cap L^2\hookrightarrow \dot H^{\beta/2-1}$ (see Lemma \ref{lem:embedding_in_neg_Sobolev}).
\end{remark}

\begin{remark}
Let us emphasize some aspects of Theorem \ref{thm:main}:
\begin{itemize}
    \item Our result covers the whole range $\beta\in (0,1)$.
    Indeed, for such $\beta$, upon choosing $\alpha=\beta/2+\delta$ with $\delta>0$ small enough, we deduce strong well-posedness of the SPDE for $\theta\in L^p$ as soon as
    \begin{align*}
        p \geq 2, \quad p>\frac{1}{1-\beta};
    \end{align*}
    for any $\beta\in (0,1)$, we can always find $p\in [2,\infty)$ large enough such that the above holds. For instance, for such choice of $\alpha$, we deduce well-posedness for $\theta_0\in L^1\cap L^2$ as soon as $\beta < 1/2$.
    \item Formally taking $\beta=0$, our condition \eqref{eq:parameters_uniqueness_intro} recovers the main constraint $\alpha<1-1/p$ from \cite{coghi2023existence}. However, differently from there, we can allow the somewhat critical equality $\alpha+\beta/2=1-1/p$; this is achieved by more refined arguments, based on a clever decomposition of our solutions, weak-strong uniqueness techniques and Yamada-Watanabe arguments.
    \item So far we are unable to treat the SQG case $\beta=1$. Indeed, even replacing all $<$ with $\leq$ in \eqref{eq:parameters_uniqueness}, in order for it to be satisfied we would be forced to take $\alpha=1/2$ and $p=\infty$. This is currently beyond our methods and an interesting problem for the future.
    \item As a consequence of well-posedness, we can deduce that  solutions of \eqref{eq:stochastic_gSQG_intro} are the unique limit of vanishing viscosity (Proposition \ref{prop:vanishing_viscosity}) and smoothing approximations (Proposition \ref{prop:smooth_approximations}).
\end{itemize}
\end{remark}

While preparing this paper, the work \cite{jiao2024well} appeared on arXiv. In particular, \cite[Theorem 1.4]{jiao2024well} establishes uniqueness by rough incompressible Kraichnan noise for gSQG, for certain values of the parameters $\alpha$ and $\beta$; existence by noise in certain $L^p$ spaces for $p<2$ is also shown.
However, their uniqueness result only seems to cover values $\beta\in (0,1/2)$ nor the full range \eqref{eq:parameters_uniqueness_intro}, cf. \cite[Remark 1.6]{jiao2024well}\footnote{Note the different convention for $\beta$: in \cite{jiao2024well}, they take $\tilde \beta=2-\beta$ with respect to our convention.}; moreover, here we can deal with the ``critical'' case $\alpha+\beta/2=1-1/p$.

\subsection{Main ideas of the proof}

For simplicity, we focus here on the case with zero forcing, $f=0$. Firstly we describe the properties of the solutions and of the noise that we are going to use, then we sketch the main computations and arguments to deal with both the subcritical and critical case.

We recall the existence of two important formal invariants for \eqref{eq:gSQG}: since $\nabla\cdot u=0$, we have the Casimir invariants
\begin{equation}\label{eq:casimir}
	\frac{\dd}{\dd t} \int_{\R^2} \varphi(\theta_t(x)) \dd x = 0
\end{equation}
for all sufficiently nice $\varphi:\R\to\R$, implying in particular formal preservation of all $L^p$ norms.
This property (formally) transfers to the SPDE \eqref{eq:stochastic_gSQG_intro}, thanks to the noise being Stratonovich and divergence free.

The second invariant, only valid in the deterministic case $W\equiv 0$, is the energy or Hamiltonian:
\begin{equation}\label{eq:energy_invariance}
	\frac{1}{2} \frac{\dd}{\dd t} \| \theta_t\|_{\dot H^{-1+\beta/2}}^2 = 0.
\end{equation}
To get \eqref{eq:energy_invariance}, one can use \eqref{eq:gSQG} and integration by parts as follows:
\begin{align*}
	\frac{1}{2} \frac{\dd}{\dd t} \| \theta_t\|_{\dot H^{-1+\beta/2}}^2
	& = \langle \partial_t \theta, \Lambda^{-2+\beta} \theta\rangle
	= \langle u\cdot\nabla \Lambda^{-2+\beta} \theta, \theta \rangle\\
	& = -\langle \nabla^\perp \Lambda^{-2+\beta} \theta\cdot\nabla \Lambda^{-2+\beta} \theta, \theta \rangle
	= 0.
\end{align*}
Notice that the above computation reveals the general cancellation property
\begin{equation}\label{eq:cancellation_nonlinearity}
	\langle (K_\beta \ast \theta)\cdot\nabla \psi, \theta\rangle_{\dot H^{\beta/2-1}}
    = \langle (K_\beta \ast \theta)\cdot\nabla \psi, \Lambda^{-2+\beta} \theta \rangle = 0,
\end{equation}
valid for all sufficiently regular $\theta$ and $\psi$ such that the above pairings are well defined.

In the stochastic case, even though \eqref{eq:energy_invariance} does not hold anymore, we can still expect some useful cancellations from \eqref{eq:cancellation_nonlinearity}, whenever we look at the evolution of $\dot H^{-s}$-norms; they will be relevant in the uniqueness argument sketched below.

In light of \eqref{eq:casimir}, fairly standard a priori estimates and compactness/tightness arguments yield weak existence of solutions $\theta\in L^\infty_\omega L^\infty_t L^p_x$ to \eqref{eq:stochastic_gSQG_intro}, such that $\| \theta_t\|_{L^p}\leq \| \theta_0\|_{L^p}$ for all $t\geq 0$.
Therefore let us only focus here on the most important part of Theorem \ref{thm:main}, namely uniqueness. This is where we expect the non-trivial structure of the noise $W$ to play an important role.

A celebrated and striking property of the Kraichnan noise is the intrinsic stochasticity of the underlying Lagrangian particles $X_t^x$, see e.g. \cite{BeGaKu1998,LeJRai2002}: particles starting at the same position $x$ split with positive probability, at any $t>0$, and their evolution is correctly described by kernels of probability measures.
At the Eulerian level, if $\rho$ solves the linear transport equation $\dd \rho  + \circ \dd W\cdot\nabla \rho = 0$, particle splitting translates into anomalous dissipation of $t\mapsto \| \rho_t\|_{L^2}$, which decreases over time.
Unfortunately, the last property is only known to hold for linear transport equations, and so far has not been successfully transferred to the nonlinear setting.
A fundamental intuition of the work \cite{coghi2023existence} (see also \cite{coti2024mixing,GGM2024}) is that the presence of transport noise strongly affects the evolution of negative Sobolev norms $\dot H^{-s}$ of solutions: if $\xi$ is an $L^p$-valued It\^o process satisfying
\begin{equation}\label{eq:ito_process_intro}
    \dd \xi_t + h_t \dd t + \circ \dd W_t\cdot\nabla \xi_t =0, 
\end{equation}
then for a large class of random processes $h$ one can show that the $\dot H^{-s}$-norm of $\xi$ satisfies the inequality
\begin{equation}\label{eq:balance_negative_norms}
	\frac{\dd }{\dd t} \E[\| \xi_t\|_{\dot H^{-s}}^2 ]
	\leq - K\, \E[\| \xi_t\|_{\dot H^{-s+1-\alpha}}^2 ] + 2\E[\langle \xi_t,h_t\rangle_{\dot H^{-s}}] + C \E[\| \xi_t\|_{\dot H^{-s}}^2]
\end{equation}
for suitable constants $K,C>0$, as long as $s>1-\alpha$. Heuristically, on average the noise $W$ gives rise to a coercive term $\| \cdot \|_{\dot H^{-s+1-\alpha}}^2$, corresponding to an increase in regularity of a factor $1-\alpha$ compared to the initial norm $\| \cdot \|_{\dot H^{-s}}^2$ in consideration. We see that, as $\alpha$ gets smaller, while $W$ becomes spatially rougher, the energy balance gets better. For a rigorous formalization of \eqref{eq:balance_negative_norms}, we refer to Theorem \ref{thm:GGM}, which is taken from \cite{GGM2024}.

Armed with \eqref{eq:balance_negative_norms}, we can present a bit loosely the regularity counting yielding condition \eqref{eq:parameters_uniqueness_intro}. To this end, given two solutions $\theta^i$, denoting by $u^i$ the corresponding velocity fields, we set $\xi=\theta^1-\theta^2$; it solves
\begin{equation}\label{eq:difference_solutions}
	\dd \xi +u^1\cdot\nabla \xi \dd t + (K_\beta\ast \xi)\cdot\nabla \theta^2 \dd t+ \circ \dd W\cdot\nabla \xi =0
\end{equation}
which is exactly of the form \eqref{eq:ito_process_intro}. 
In order to handle the nonlinear term in estimates of the form \eqref{eq:balance_negative_norms}, it is convenient to try to produce the best possible cancellations; in view of \eqref{eq:cancellation_nonlinearity}, we pick $s= 1-\beta/2$, which requires the constraint $\beta/2<\alpha$.
In this way, the term $K_\beta\ast \xi\cdot\nabla \theta^2$ gives no contribution and from \eqref{eq:balance_negative_norms} (upon computing the $\langle \cdot,\cdot\rangle_{\dot H^{\beta/2-1}}$-pairing and doing integration by parts) we arrive at
\begin{equation}\label{eq:basic_ideas_eq2}
	\frac{\dd }{\dd t} \E[\| \xi\|_{\dot H^{\beta/2-1}}^2 ]
	\le - K\, \E[\| \xi\|_{\dot H^{\beta/2-\alpha}}^2 ] + \E[|\langle u^1\cdot\nabla \Lambda^{\beta-2}\xi, \xi \rangle|] + C\,  \E[\| \xi\|_{\dot H^{\beta/2-1}}^2 ].
\end{equation}
In order to close the estimate by a Gr\"onwall-type argument, using duality pairings, we expect uniqueness to be within reach as soon as we establish a (deterministic) functional inequality of the form
\begin{equation}\label{eq:basic_ideas_eq3}
	\| u^1_t\cdot\nabla \Lambda^{\beta-2}\xi_t\|_{\dot H^{\alpha-\beta/2}}
    \lesssim \| \theta^1_t\|_{L^p} \| \xi\|_{\dot H^{\beta/2-\alpha}}
    \lesssim \| \theta^1_0\|_{L^p} \| \xi\|_{\dot H^{\beta/2-\alpha}}
\end{equation}
possibly with a small multiplicative constant in front, so that the above right-hand side is absorbed into \eqref{eq:basic_ideas_eq2}.
Inequality \eqref{eq:basic_ideas_eq3} provides another natural constrain on the parameters $\alpha$, $\beta$. Indeed, in order to have $u^1\cdot\nabla \Lambda^{\beta-2}\xi\in \dot{H}^{\alpha-\beta/2}$, it is natural to require $\nabla \Lambda^{\beta-2}\xi\in \dot{H}^{\alpha-\beta/2}$; for $\xi\in \dot{H}^{\beta/2-\alpha}$, in view of the smoothing properties of $\Lambda$, in general this requires $\dot{H}^{1-\alpha-\beta/2}\subset \dot{H}^{\alpha-\beta/2}$ and thus (combined with our previous constraint)
\begin{equation*}
	\frac{\beta}{2}< \alpha\leq \frac{1}{2}.
\end{equation*}
Next, we can do a quick regularity counting: if $\theta^1\in L^p$, then $u^1=\nabla^\perp \Lambda^{\beta-2}\theta^1$ belongs to the homogeneous Bessel space $\dot L^{1-\beta,p}$ (i.e. $\Lambda^{1-\beta}u^1\in L^p$);
similarly, if $\xi\in \dot{H}^{\beta/2-\alpha}$, then $\nabla \Lambda^{\beta-2}\xi \in \dot H^{1-\alpha-\beta/2}$.
The desired estimate \eqref{eq:basic_ideas_eq3} can then be deduced by a combination of the fractional Leibniz rule and functional embeddings, resulting in our Lemma \ref{lem:product_sobolev_norm} on products in fractional Bessel spaces, provided that
\begin{equation}\label{eq:critical}
    \alpha+\frac{\beta}{2} = 1-\frac{1}{p}.
\end{equation}
Constraint \eqref{eq:critical} is somewhat critical, in the sense that a direct application of Lemma \ref{lem:product_sobolev_norm} does not allow us to produce a small constant in \eqref{eq:basic_ideas_eq3}.
The situation would be different in the case of strict inequality (which we could regard as a subcritical regime) in \eqref{eq:critical}, in which case one can gain a ``$\delta$ of space'' in regularity estimates, by means of
\begin{equation}\label{eq:basic_ideas_delta_room}
	|\langle u^1\cdot\nabla \Lambda^{\beta-2}\xi, \xi \rangle|
    \lesssim \| \theta^1_0\|_{L^p} \| \xi\|_{\dot H^{\beta/2-\alpha-\delta}}^2,
\end{equation}
valid for some $\delta=\delta(p)>0$. In this case, by interpolation estimates, the r.h.s. of \eqref{eq:basic_ideas_delta_room} can always be controlled by a small multiple of $\| \xi\|_{\dot H^{\beta/2-\alpha}}^2$, at the price of producing an additional term $C \| \xi\|_{\dot H^{\beta/2-1}}^2$, with $C$ possibly very large. The latter term is however harmless, as it can always be controlled by Gr\"onwall's lemma.

Our main intuition in order to reach equality \eqref{eq:critical} is to combine the nice estimates available in the subcritical regime, with the incompressible transport nature of the SPDE. In particular, given any $\theta_0\in L^p$, we can always decompose it as $\theta_0=\theta_0^<+\theta_0^>$, where $\theta_0^>$ can be made arbitrarily small in $L^p$, while $\theta_0^<$ enjoys higher integrability, in the sense that $\theta_0^<\in L^p\cap L^\infty$.
Correspondingly, if we think of the SPDE \eqref{eq:stochastic_gSQG_intro} as being ``linear in $\theta$ for fixed $u$'', the solution $\theta$ should decompose as $\theta=\theta^<+\theta^>$, where $\theta^\lessgtr$ respectively solve
\begin{align*}
    \dd \theta^\lessgtr + u\cdot\nabla \theta^\lessgtr \dd t + \circ \dd W\cdot\nabla \theta^\lessgtr = 0,\quad \theta^\lessgtr\vert_{t=0}=\theta^\lessgtr_0.
\end{align*}
As $u$ and $W$ are divergence free, we then expect both the smallness property of $\theta^>_0$ in $L^p$ and the higher integrability of $\theta^<_0$ to be propagated at positive times, and thus transfer to the decomposition $\theta=\theta^<+\theta^>$ of our solution. 
We can now combine both estimates \eqref{eq:basic_ideas_eq3} and \eqref{eq:basic_ideas_delta_room}, applied respectively to terms associated to $\theta^>$ and $\theta^<$, to produce the desired small constant in front of \eqref{eq:basic_ideas_eq3}, for any choice of initial data $\theta_0\in L^p$.

Rigorously, the construction of a solution $\theta$ satisfying the aforementioned decomposition is a bit more technical, requiring to argue via a priori estimates and compactness arguments; once we have $\theta$, any other weak $L^p$-valued solution $\tilde\theta$ can be compared to $\theta$, by employing a PDE weak-strong uniqueness type argument and a Yamada-Watanabe type coupling. Overall, this allows to conclude that any other weak solution must coincide with $\theta$, and so that pathwise uniqueness and uniqueness in law hold.

\subsection{Transport equation with random drift}

In the above argument to show uniqueness in Theorem \ref{thm:main}, we have not really used the exact form of the nonlinearity, but rather:
\begin{itemize}
\item the regularity constraint $u^1\in \dot{L}^{1-\beta,p}$;
\item the uniform a priori $L^p$-bounds for solutions $\theta$ (in formula \eqref{eq:basic_ideas_eq3}), coming from the divergence free property of $u$;
\item the cancellation of the ``nonlinear contribution'' coming from
$(u^1-u^2)\cdot \nabla \theta^2=K_\beta\ast \xi\cdot\nabla \theta^2$, thanks to identity \eqref{eq:cancellation_nonlinearity}.
\end{itemize}
Hence, for a given \emph{random} drift $b$, if consider the linear stochastic transport equation
\begin{equation}\label{eq:new_transport_intro}
    \dd \zeta + b\cdot\nabla\zeta \dd t + \circ \dd W\cdot\nabla \zeta=0,
\end{equation}
we may expect a similar argument to yield wellposedness
in the class of $L^1\cap L^2$-valued solutions $\tilde\theta$, under the assumptions that $b$ is adapted, divergence free and satisfies the regularity condition $b\in L^\infty_{\omega,t} \dot L^{1-\beta,p}_x$.
In the next result, we set $\gamma:=1-\beta$, so to rephrase everything in terms of the fractional $L^{\gamma,p}_x$-regularity of $b$; given $q\in [1,\infty]$ and a function space $E$, we adopt the notation
\begin{align*}
    L^\infty\big( \Omega; L^q_\loc([0,+\infty);E)=\bigcap_{n\in\N} L^\infty\big( \Omega; L^q([0,n]);E).
\end{align*}

\begin{theorem}\label{thm:main_linear}
    Let $\alpha$, $\gamma$ and $p$ be parameters satisfying
    \begin{equation}\label{eq:parameters_linear_uniqueness_intro}
		\alpha\in \Big(0,\frac{1}{2}\Big), \quad \gamma\in (0,1), \quad p\in (1,\infty), \quad \gamma p<2, \quad \alpha < \frac{\gamma +1}{2}-\frac{1}{p}.
    \end{equation}
    Let $b$ and $f$ be random, progressively measurable coefficients such that
    \begin{align*}
        \nabla\cdot b\equiv 0, \quad
        b\in L^\infty\big( \Omega; L^\infty_\loc([0,+\infty); \dot L^{\gamma,p})\big), \quad
        f\in L^\infty\big( \Omega; L^1_\loc([0,+\infty);L^1\cap L^2)\big).
    \end{align*}
    Then, for any initial condition $\zeta_0\in L^1\cap L^2$, there exists a global probabilistically strong, analytically weak solution $\zeta$ to \eqref{eq:new_transport_intro}, with the property that
    \begin{equation}\label{eq:uniqueness_class_linear}
        \zeta\in L^\infty\big(\Omega; L^\infty_\loc([0,+\infty);L^1\cap L^2)\big).
    \end{equation}
    Furthermore, pathwise uniqueness holds in the class of solutions $\zeta$ satisfying \eqref{eq:uniqueness_class_linear}. 
 
    Moreover, solutions depend continuously on the data of the problem: given two solutions $\zeta^i$, $i=1,2$, associated to different initial data $\xi^i_0$ and forcing $f^i$, but with the the same drift $b$, then $\P$-a.s.
    \begin{equation}\label{eq:intro_linear_stability}
        \| \zeta^1_t-\zeta^2_t\|_{L^1\cap L^2} \leq \| \zeta^1_0-\zeta^2_0\|_{L^1\cap L^2} + \int_0^t \| f^1_s-f^2_s\|_{L^1\cap L^2} \dd s \quad \forall\, t\geq 0.
    \end{equation}
\end{theorem}

A more detailed version of the result can be found in Theorem \ref{thm:linear_random_wellposed} and the proof is presented in Section~\ref{sec:linear_random}.
Notice that we may allow both the drift $b$ and forcing $f$ to be random; this is thanks to the linear nature of the SPDE \eqref{eq:new_transport_intro}, which allows to directly construct strong solutions by an alternative weak compactness argument (cf. Section \ref{subsec:existence.random}). The uniqueness strategy outlined above is implemented in Section \ref{subsec:uniqueness.random}.

The importance of this result is twofold. Firstly, we allow the irregular drift $b$ to be random. The restriction to deterministic drifts has been a long-standing limitation of many regularization by noise results: indeed, a result like Theorem \ref{thm:main_linear} is false when $W$ is replaced by a smooth-in-space noise, see the first counterexample in \cite[Section 6.2]{FGP2010}. Hence, the nonsmooth Kraichnan noise allows to overcome this limitation.

Secondly, the solvability of the stochastic transport equation \eqref{eq:new_transport_intro} gives a unique evolution of the associated passive scalar $\zeta$,
for a large class of initial data $\zeta_0$.
In analogy to the DiPerna--Lions--Ambrosio theory \cite{diperna1989ordinary,Ambrosio2004}
this suggests in some sense uniqueness of generalized flows for the underlying SDE
\begin{equation}\label{eq:intro_random_SDE}
    \dd X_t = b(t,X_t) \dd t + W(\circ \dd t, X_t).
\end{equation}
At the same time, in analogy with Le Jan--Raimond \cite{LeJRai2002,LeJRai2004} (rigorously treating the case $b\equiv 0$), we can only expect this SDE to admit a \emph{flow of kernels} instead of a classical \emph{flow of maps}; moreover, renormalizability of solutions to \eqref{eq:new_transport_intro} might fail, and anomalous dissipation of energy might take place.

Notice that the assumptions of Theorem \ref{thm:main_linear} are satisfied for $b=u=K_\beta\ast \theta$, for the solutions $\theta$ to gSQG coming from Theorem \ref{thm:main}; further interesting examples of random velocity fields $b$ can be generated as solutions to fluid dynamics SPDEs.
The above considerations apply to all these cases, and a general theory of solution flows for the underlying SDE \eqref{eq:intro_random_SDE} is lacking;
we believe the solvability of \eqref{eq:new_transport_intro} could play an important role in its development, but we leave this question for the future.

\begin{remark}
Theorem \ref{thm:main_linear} still holds in the regime $\gamma\in [1,2)$, see Theorem \ref{thm:linear_random_wellposed}; in this case the result is non-trivial, but in the absence of noise ($W\equiv 0$), well-posedness of the transport equation follows from DiPerna--Lions theory \cite{diperna1989ordinary}, contrary to the regime $\gamma<1$.

The assumptions of Theorem \ref{thm:main_linear} are likely not optimal: for example, we restricted to the case $d=2$, but it is reasonable to expect it to hold on $\R^d$ for any $d\ge 2$; similarly, a larger class of initial condition $\zeta_0 \in L^1_x\cap L^q_x$ (for suitable $q$ depending on $\alpha,\, \gamma,\, p$) might be allowed, and perhaps the restriction $\alpha<1/2$ can be dropped.
We leave the problem of finding sharper conditions for the solvability of \eqref{eq:new_transport_intro} for future investigations.
\end{remark}

\subsection{Review of existing literature}\label{subsec:literature}

We discuss here more in detail several results in the literature related to our (S)PDE of interest. To simplify the exposition, we divide them in different thematic blocks.

\emph{SQG equations.} The rigorous mathematical study of the SQG equations was initiated in \cite{CoMaTa1994}, which established local existence of smooth solutions and a blow-up criterion.
Global existence of weak $L^p$-valued solutions was shown in \cite{Resnick1995} for $\theta_0\in L^2$ and then extended in \cite{Marchand2008} to $\theta_0\in L^p$ with $p>4/3$; uniqueness of such solutions remains an open problem.

After these works, there have been several attempts at understanding whether loss of regularity and blow-up of strong norms might hold. Examples of self-similar solutions with infinite energy blowing up in finite time were provided in \cite{CasCor2010}; instead specific examples of global smooth solutions were constructed in \cite{CaCoGS2020}, with a rigorous computer-assisted proof.
The construction from \cite{HeKis2021} provides initial data $\theta_0$ such that either finite time blow-up or exponential growth of strong norms must happen.
More recently, the work \cite{CorMar2022} provides initial data $\theta_0$ for which strong norm inflation in $C^k$-spaces (as well in $H^s$, for $s\in (3/2,2)$) takes place, resulting also in non-existence of solutions $\theta\in C([0,\delta]; C^k)$ for any $\delta>0$; similar results were provided in \cite{JeoKim2024}, in the setting of the critical scales $H^2\cap W^{1,\infty}$.
In a different direction, the existence of invariant measures for SQG supported on $H^2\cap W^{1,4}$ has been shown in \cite{FolSy2021}.

Convex integration schemes have been successfully applied to SQG to deduce non-uniqueness of very weak solutions in \cite{BuShVi2019}; these solutions belong to $C^0_t C^{\gamma-1}$, for $\gamma\in (1/2,4/5)$.
It is important to mention here that the SQG (as well as gSQG and Euler) equations belong to a class of active scalar equations with \emph{odd Fourier multiplier}, which is especially hard to tackle by convex integration schemes, and for which the $h$-principle is not expected to hold, see the discussions in \cite{Shvydkoy2011,IseVic2015}. For this reason, the scheme is not implemented at the level of $\theta$, and does not produce $L^p$-valued solutions; rather in \cite{BuShVi2019} the SQG system is rewritten in its momentum form and applied to the potential velocity $v$ implicitly defined by $\theta = -\nabla^\perp \cdot v$.
After \cite{BuShVi2019}, alternative constructions have been presented in \cite{IseMa2021} and in the presence of forcing in \cite{bulut2023}; the last reference yields non-uniqueness for $\theta\in C_t C^{\gamma-1}$ in the presence of a forcing $f\in C_t C^{\gamma-1}$, for any $\gamma\in (0,1)$.

\emph{Generalised SQG equations.}
The inviscid gSQG equations belong to a general class of active scalar models first introduced in \cite{ChCoWu2011}; local well-posedness of smooth solutions and global existence of  weak $L^2$ solutions were established in \cite{CCCGW2012}. The latter result has been extended to bounded domains in \cite{Nguyen2018}.
It has been recently shown in \cite{IseMa2024} that such weak solutions preserve angular momentum.

In analogy to Euler, the vortex patch problem for gSGQ has received considerable attention. Finite time blow-up for the patch equation has been shown in \cite{KRYZ2016}, see also \cite{GanPat2021} for more refined results and \cite{CoGSIo2019} for some global existence results.
Existence of very weak solutions with white noise marginals has been shown in \cite{NPST2018,FlaSaa2019}, generalising a classical result for $2$D Euler by Albeverio and Cruzeiro \cite{AlbCru1990}.
The aforementioned results on strong norm inflation and non-existence of solutions have been extended to gSQG as well. Let us mention in this regard: i) the work \cite{cordoba2022}, for parameters $\beta\in (1,2)$ (a regime in which $u=-\nabla^\perp \Lambda^{2-\beta}$ is less regular than $\theta$); ii) the aforementioned \cite{JeoKim2024}, whose techniques readapt to show ill-posedness for gSQG case in $H^{1+\beta}$ (cf. Section 1C therein); iii) the recent \cite{choi2024well}, proving local well-posedness in $C^\gamma\cap L^1$ for $\gamma>\beta$, as well as constructing $\theta_0$ for which non-existence of solutions in $C([0,\delta];C^\beta)$ holds for any $\delta>0$. 
Convex integrations schemes have been applied to gSQG in \cite{ma2022some}; in particular, Theorem 3 from Section 3.3 therein implies non-uniqueness of very weak solutions such that $\Lambda^{-1+\beta/2}\theta\in C_t C^\gamma$, for suitable values $\gamma=\gamma(\beta)<3/10$.
We finally mention the recent work \cite{castro2024}, revisiting the arguments by Vishik \cite{vishik2018a,vishik2018b}, who first constructed non-unique $L^p$-valued solutions to the forced $2$D Euler equations in vorticity form, for a carefully chosen forcing $f\in L^1_t L^p$. In \cite{castro2024}, similar upcoming results for gSQG and SQG were announced, later proved in \cite{castro2025unstable}.

\emph{Kraichnan noise and turbulence.}
An idealized description of the effects of small,
possibly turbulent, fluid scales by means of a Brownian-in-time, coloured-in-space noise $W$ was
first proposed by Kraichnan in the context of passive scalar turbulence, see \cite{Kraichnan1968} and later \cite{kraichnan1994anomalous}.
This model has become very popular after \cite{BeGaKu1998,FaGaVe2001}, where it was shown that in the low regularity regime $\alpha\in (0,1)$ particle splitting and intrinsic stochasticity take place, resulting in anomalous dissipation for the passive scalar; the link between these properties holds in great generality in transport equations, see \cite{DriEyi2017}. At the same time, solutions to the inviscid equation are still unique and can be recovered by the vanishing viscosity limit. For more details on the physical features of the model, we refer to the lecture notes \cite{cardy2008}, and for rigorous mathematical results to \cite{LeJRai2002,LeJRai2004,GalLuo2023}.
Indeed, Le Jan and Raimond \cite{LeJRai2002,LeJRai2004} conducted a very deep mathematical study of a wider class of Brownian flows; according to their classification, the incompressible Kraichnan model is diffusive without hitting. We also mention the recent work \cite{Rowan2023}, quantifying anomalous $L^2$ dissipation for the rough Kraichnan model (on $\T^2$). 

More recently, there has been an increasing interest in incorporating transport noise in nonlinear fluid dynamics. We shortly mention some key references: i) the Lagrangian-based approach proposed in \cite{BrCaFl1991}, where the stochastic Navier-Stokes equations with transport noise were derived by computing the stochastic material derivative along the trajectories of the fluid particles; ii) the variational approach by Holm \cite{Holm2015}; iii) the Location Uncertainty method by Memin \cite{Memin2014}. In most approaches, the Stratonovich form of the noise arises naturally and is physically justified by the Wong--Zakai principle.

\emph{Regularization by transport noise in PDEs.}
The realization that transport noise can have a strongly regularizing effect on linear SPDEs is due to Flandoli, Gubinelli and Priola \cite{FGP2010}, who showed well-posedness of stochastic transport equations with H\"older continuous drifts.
The result has been subsequently refined in several ways, see \cite{FedFla2013} for the propagation of Sobolev regularity of initial data and \cite{FlMaNe2014} for the prevention of blow-up due to stretching.
Since then, the literature has flourished, see \cite{flandoli,gess} for an overview.
At the same time, it was already understood in \cite[Section 6.2]{FGP2010} that the situation can be much more complicated in the nonlinear setting, where too simple noise might not help.
This intuition has been recently confirmed by the stochastic convex integration results from \cite{HoLaPa2024,pappalettera2023}, which imply non-uniqueness for stochastic $3$D Euler and Navier--Stokes equations in the presence of regular transport noise.

On the other hand, a ``sufficiently active'' transport noise still still has the ability to suppress blow-up with high probability, for a large class of parabolic equations; this was shown e.g. for $3$D Navier--Stokes \cite{FlaLuo2021}, Keller--Segel and Kuramoto--Sivashinsky \cite{FGL2021blowup}, reaction-diffusion equations \cite{agresti2023} and Tao's averaged version of Navier--Stokes \cite{lange2023}. These results build on the scaling limit argument introduced in \cite{galeati2020}, which allows the noise to be smooth, but supported on very high Fourier modes.
The suppression mechanism is closely connected to the mixing properties of Kraichnan noise \cite{GesYar2021} and bears strong similarities with deterministic results like \cite{iyer2021}; it crucially relies on the parabolic nature of the problem and does not transfer easily to the analysis of inviscid PDEs.
In the inviscid setting, apart from the work \cite{coghi2023existence} which we build upon, let us mention an alternative approach based on Girsanov transform, applied to $2$D logEuler equations in \cite{GalLuo2023}, in part based on \cite{BBF2014} for $3$D Leray-$\alpha$ Euler.


\subsection{Structure of the paper}\label{subsec:structure_paper}

We conclude this introduction by recalling the most relevant notations and conventions we will adopt in the paper in Section \ref{subsec:notation}.
Then in Section \ref{sec:preliminaries} we describe in detail the properties of the Kraichnan noise $W$ we consider and define our notion of solutions to \eqref{eq:stochastic_gSQG_intro}.
Section \ref{sec:existence} is devoted to the proof of Theorem \ref{thm:weak_existence}, which ensures the existence of weak solutions, for a large class of $\theta_0$ and $f$. It is based on a standard probabilistic compactness argument, based on first studying a class of regularized SPDEs and deriving a priori estimates, and then passing to the limit removing the mollification.
Section \ref{sec:uniqueness} presents the proof of pathwise uniqueness and stability of solutions (Theorem \ref{thm:uniqueness}), which allows to complete of the proof of our main Theorem \ref{thm:main}. As a consequence, we deduce the convergence of vanishing viscosity and regular approximations to the unique solutions to the SPDE in Section \ref{subsec:vanishing_viscosity}.
Finally, in Section \ref{sec:linear_random} we prove the regularization result for 2D transport equations with nonsmooth random drifts, Theorem \ref{thm:main_linear}.
We collect some useful analysis facts that were used throughout the paper in Appendix \ref{app:analysis_lemmas}.

\subsection{Notation}\label{subsec:notation}
In this subsection we summarize the notations and conventions used throughout the manuscript.

Given $x,y \in \R^d$, we denote by $x\cdot y$ their scalar product, while $|x|$ stands for the Euclidean norm. For $d=2$, $x^\perp\coloneqq (-x_2, x_1)^T$ will denote the counterclockwise 90 degrees rotation of $x$ and similarly $\nabla^\perp\coloneqq(-\partial_2, \partial_1)$ will indicate the orthogonal gradient. When $A\in \R^{d\times d}$ is a matrix, $\trace A$ will stand for its trace. We will use the symbol $\otimes$ to indicate the tensor product $x\otimes y$, when $x,y \in \R^d$.

When dealing with inequalities, we will write $a\lesssim b$ if there exists a constant $c>0$ such that $a \le cb$; to stress the dependence of the hidden constant with respect to some parameter family $\lambda$, we will sometimes write $a \lesssim_\lambda b$.

We denote by $C^\infty_c(\R^d;\R^m)$ the space of smooth functions with compact support and by $\cS(\R^d;\R^m)$ the space of Schwartz functions; $\cS'$, the dual of $\cS(\R^d;\R^m)$, denotes the so-called tempered distributions. For $p\in [1,\infty]$, we write $L^p (\R^d;\R^m)$ for the standard Lebesgue spaces. For $k\in \N$ and $p\in[1,\infty]$, $W^{k,p}(\R^d;\R^m)$ will stand for the classical Sobolev spaces. In general, when clear from the context, we will drop from the notation possibly both the domain and codomain of such functions, writing for instance $L^p$ instead of $L^p (\R^d;\R^m)$. Further, we will use the subscript $\loc$ when dealing with functions that have some regularity only locally. For example, $L^p_\loc$ will denote the functions $f$ such that $f\varphi \in L^p$ for every $\varphi \in C^\infty_c(\R^d;\R)$ and similarly for other function spaces. Such localized spaces can be endowed with a countably family of seminorms, inducing a metric and a Fréchet topology.
We will use indifferently the bracket $\langle \cdot,\cdot\rangle$ to denote both the inner product in $L^2$ and the duality pairings between $C^\infty_c$ and $(C^\infty_c)'$, $\cS$ and $\cS'$, or more generally any function space $E$ and its dual.

We adopt the convection that the Fourier transform of an integrable function $f$ is defined by
\begin{equation}\label{eq:convention_fourier}
    \cF(f)(n) = \widehat f (n) \coloneqq (2\pi)^{-d/2} \int_{\R^d} f (x) e^{-ix\cdot n} \dd x,
\end{equation}
and suitably extended to tempered distributions. Under this convention, the antitransform is given by $(\cF^{-1}f)(x) = \check{f}(x) = \overline{\hat f(x)}$ and standard properties of products and convolutions read as
\begin{equation}\label{eq:properties_fourier_transform}
    \widehat{f\ast g} = (2\pi)^{d/2} \hat f \hat g, \quad \| \hat f\|_{L^2} = \| f\|_{L^2}, \quad \| f\ast g\|_{L^2}=(2\pi)^{d/2} \| \hat f\hat g\|_{L^2}.
\end{equation}
With this notation set, we introduce both homogeneous and inhomogeneous Sobolev spaces. First, for $s\in \R$, the inhomogenous Sobolev space $H^s(\R^d, \R^m)$ is the space of tempered distributions $f$ such that $\widehat f \in L^2_\loc$ and
\begin{equation*}
    \| f \|^2_{H^s} \coloneqq \int_{\R^d} \langle n \rangle^{2s}|\widehat f (n)|^2 \dd n < \infty, 
\end{equation*}
where $\langle n \rangle \coloneqq (1+|n|^2)^\frac{1}{2}$. On the other hand, the homogeneous Sobolev space $\dot H^s(\R^d, \R^m)$ is the space of tempered distributions $f$ such that $\widehat f \in L^1_\loc$ and
\begin{equation*}
    \| f \|^2_{\dot H^s} \coloneqq \int_{\R^d} |n |^{2s}|\widehat f (n)|^2 \dd n < \infty.
\end{equation*}
Let $\Lambda = |\nabla| \coloneqq (-\Delta)^\frac12$, notice that $\| f \|_{H^s} = \| (\id - \Delta)^\frac{s}{2} f \|_{L^2}$, while $\| f \|_{H^s} = \| \Lambda^s f \|_{L^2}$. For $s>0$ and $p\in (1,\infty)$, we introduce the homogeneous Bessel potential space $\dot L^{s,p}$ as the space of functions $f$ such that $\Lambda^s f\in L^p$, and we set $\| f\|_{\dot L^{s,p}}:= \| \Lambda^s f\|_{L^p}$.
For $s\in (0,1)$, the homogeneous spaces $\dot H^{-s}$ are Hilbert spaces, endowed with the inner product $\langle \varphi,\psi\rangle_{\dot H^{-s}} = \langle \Lambda^{-2s} \varphi,\psi\rangle$.
Given $s\in \R$, $p,r\in [1,\infty]$, $B^s_{p,r}$ will denote the nonhomogeneous Besov space as defined in \cite[Definition 2.68]{Bahouri2011}.

In this work, the time variable will always belong to the interval $[0,T]$, for an arbitrary large but finite $T>0$. We will use the subscripts $t\in [0,T]$ and $\omega\in \Omega$ to refer to functional spaces in the time and sample variables. For example, given a Banach space $E$, $L^q_t E$ will denote the space $L^q([0,T]; E)$ of $q$-integrable functions with values in $E$ (defined in the Bochner sense). Similarly, for $\gamma \in (0,1)$, $C^\gamma_t E$ will denote $\gamma$-H\"older functions with values in $E$, and $L^\infty_{\omega,t}L^p=L^\infty(\Omega\times [0,T];L^p(\R^d))$. Notice also that, given two Banach spaces $E$ and $F$, we endow their intersection $E\cap F$ with $\|\cdot\|_{E\cap F} = \|\cdot\|_E + \|\cdot \|_F$, which makes it Banach.

With $(\Omega,\mathcal{A},\mathcal{F}_t,\P)$ we will denote a filtered probability space satisfying the standard assumption. $\E$ or $\E_\P$ will stand for the expectation with respect to the probability measure $\P$. Similarly, $\text{Law}(X)$ or $\text{Law}_\P(X)$, will denote the law of the random variable $X$ under $\P$.

When $W$ is a Wiener process, $\circ \dd W_t$ refers formally to a stochastic integration in the Stratonovich sense. On the other hand, if $\circ$ is missing, stochastic integrals should be intended in the It\^o sense.

We say that a process $X$ with values in a Banach space $E \hookrightarrow \cS'$ has $\P$-a.s.\ weakly continuous trajectories if, for any $\varphi \in C^\infty_c$, the real valued process $t \mapsto \langle X_t, \varphi\rangle$ has $\P$-a.s.\ continous paths.

Finally, we endow $C_t^{\N}=C([0,T];\R)^\N$ with the product topology, which makes it a Polish space. In particular, if $F=\{f^i\}_i$, $G=\{g^i\}_i$ and $F,G \in C_t^{\N} $, it is the topology induced by the distance
\begin{equation*}
    d(F,G)= \sum_{i\in \N} 2^{-i} \frac{\| f^i-g^i\|_{C_t}}{1+\| f^i-g^i\|_{C_t}}.
\end{equation*}
In other words, it is the topology induced by the convergence $F_n \to F$ if and only if $f^i_n \to f^i$, as $n\to \infty$, for every $i\in \N$.

\section{Preliminaries}\label{sec:preliminaries}

In this section we collect all relevant facts concerning the structure of the Kraichnan noise $W$ and the nonlinearity of the gSQG equation that we will need in order to carry our analysis. This will culminate in providing rigorous meaning to what we mean by weak solutions, cf. Definition \ref{defn:weak_solution}. Most of the material contained in this section is fairly standard and recalled here for convenience; an experienced reader may consider skipping it.

\subsection{Structure of the noise}\label{subsec:structure_noise}

Throughout the paper we will always work with a centered vector-valued Gaussian noise $W=W_t(x)$ which is Brownian in time and coloured, divergence-free in space. $W$ is uniquely determined by its covariance function $Q:\R^2\to \R^{2\times 2}$, determined by the formula
\begin{align*}
    \E\big[ W_t(x)\otimes W_s(y)\big] = (s\wedge t) Q(x-y).
\end{align*}
The noise will be taken homogenous (as seen from $Q$ only depending on $x-y$) and isotropic; correspondingly, $Q$ can be described in Fourier space by
\begin{equation}\label{eq:Q_fourier}
    Q(z)= (2\pi)^{-1} \int_{\R^2} q(\xi)\, P^\perp_\xi\, e^{i\xi\cdot z} \dd \xi\quad \text{for}\quad P^\perp_\xi:= I - \frac{\xi}{|\xi|}\otimes \frac{\xi}{|\xi|}\quad \forall\, \xi\in\R^2\setminus\{0\}
\end{equation}
where we assume $q:\R^2\to \R_{\geq 0}$ to be a radially symmetric function, $q\in L^1\cap L^\infty$. $P^\perp_\xi$ is the projection matrix on $\xi^\perp$, ensuring that $Q$ and $W$ are divergence free (alternatively, it is the Fourier multiplier representation of the Leray--Helmholtz projector).

We will soon specify $q$ to be $q(\xi)=(1+|\xi|^2)^{-d/2-\alpha}$, which renders $W$ the incompressible Kraichnan noise of parameter $\alpha\in (0,1)$. It is however useful to recall a few facts which hold in higher generality for any $q$ as above, which we will exploit in Section \ref{sec:existence} when considering regular approximations of $W$.

We may identify the matrix-valued kernel $Q$ with the convolutional operator
\begin{align*}
    (\mathcal{Q} \psi)(x):= (Q\ast \psi)(x) = \int_{\R^2} Q(x-y)\psi(y) \dd y\quad\forall\, \psi\in L^2(\R^2;\R^2).
\end{align*}
This operator is semipositive definite and admits a square root $\mathcal{Q}^{1/2}$, so that $\langle Q\ast \psi,\psi\rangle=\| \mathcal{Q}^{1/2} \psi\|_{L^2}^2$.

\begin{lemma}\label{lem:properties_Q^1/2}
	The following hold:
	\begin{enumerate}
	\item[1)] $\mathcal{Q}^{1/2}$ is a bounded operator from $L^p$ to $L^2$ for any $p\in [1,2]$;
	\item[2)] $\mathcal{Q}^{1/2}$ is a bounded operator from $L^2$ to $L^q$, for any $q\in [2,\infty]$;
	\item[3)] $\mathcal{Q}^{1/2}$ is a bounded operator from $H^s$ to itself, for any $s\in\R$.
	\end{enumerate}
\end{lemma}

\begin{proof}
	Notice that $\mathcal{Q}^{1/2}$ corresponds in Fourier space to multiplication by the matrix-valued function $h(\xi)=q(\xi)^{1/2} P^\perp_\xi$, which belongs to $L^2\cap L^\infty$.
	Since $h$ belongs to $L^2$, it maps $L^1$ into $L^2$; since it belongs to $L^\infty$, it maps $L^2$ into itself; the general case of $p\in [1,2]$ follows by interpolation, proving 1).
	2) follows from 1) by duality, since $\mathcal{Q}^{1/2}$ is a self-adjoint operator.
	3) follows again from the fact that $h\in L^\infty$.
\end{proof}

For simplicity, in the following we will assume that $q$ does not vanish anywhere, so that $\mathcal{Q}$ is a non-degenerate operator.

\begin{lemma}\label{lem:chaos_expansions}
	Let $W$ be a Gaussian noise with covariance function $Q$ as above, defined on a probability space $(\Omega,\cA,\P)$.
    Let $\{e_k\}_{k\in\N}$ be any smooth orthonormal system in $L^2(\R^2;\R^2)$ and set $\sigma_k := \mathcal{Q}^{1/2}(e_k)$. Then $\sigma_k$ are smooth, divergence-free vector fields and it holds
    \begin{equation}\label{eq:noise_series_representation}
        W_t(x) = \sum_{k\in\N} \sigma_k(x) W^k_t \quad \text{ for } \quad W^k_t :=\langle W, \mathcal{Q}^{-1/2} e_k\rangle
    \end{equation}
    where the series is $\P$-a.s. convergent in $C_t L^2_\loc$ and $\{W^k\}_k$ is a collection of independent standard Brownian motions.
    Moreover it holds
    \begin{equation}\label{eq:covariance_series_representation}
        Q(x-y)=\sum_{k\in\N} \sigma_k(x)\otimes \sigma_k(y) \quad \forall\, x,y\in\R^2
    \end{equation}
    where the series is absolutely convergent, uniformly on compact sets.
    Conversely, given a collection of independent Brownian motions $\{W^k\}_k$, formula \eqref{eq:noise_series_representation} defines a noise $W$ with covariance $Q$.
\end{lemma}

\begin{proof}
	The fact that $\sigma_k$ are smooth follows from Lemma \ref{lem:properties_Q^1/2}-3). The rest of the statement is mostly a collection of statements from \cite[Section 2.1]{GalLuo2023}, mostly Lemma 2.3 and Proposition 2.5-a).
\end{proof}

In particular, representation \eqref{eq:noise_series_representation} allows to identify $W$ with the family $(W^k)_{k\in\N}$, so that it can be regarded as a random variable in $C_t^\N$.
On the other hand, given a probability space supporting a family of independent Brownian motions $(W^k)_{k\in\N}$, we can construct therein a noise $W$ for any given covariance operator $Q$. We refer to the collection $(W^k)_{k\in\N}$ as a cylindrical Brownian motion.

Given a filtered probability space $(\Omega,\cA,\cF_t,\P)$, we say that $W$ is a $\cF_t$-noise if it is adapted to $\cF_t$ and $W_t-W_s$ is independent of $\cF_s$ for any $t>s$.
We next recall some useful estimates concerning stochastic integration w.r.t. $W$; the result is taken from \cite[Lemma 2.8]{GalLuo2023}. 

\begin{lemma}\label{lem:stochastic_integrals}
	Let $(\Omega,\cA,\cF_t,\P)$ be a filtered probability space, $W$ be a $\cF_t$-noise with covariance function $Q$ satisfying the above conditions.
	Then for any $\cF_t$-adapted process $h:\Omega\times [0,T]\times \R^2\to \R^2$ such that $\P$-a.s. $\int_0^T \|\mathcal{Q}^{1/2} h_s\|_{L^2}^2 \dd s<\infty$, the stochastic integral $M_t:= \int_0^t \langle h_s, \dd W_s\rangle$ is a well-defined continuous real-valued local martingale. Moreover it holds
    \begin{align*}
        [M]_t=\int_0^t  \|\mathcal{Q}^{1/2} h_s\|_{L^2}^2 \dd s, \quad 
        \E\Big[ \sup_{t\in [0,T]} |M_t|^p\Big]\lesssim_p \E\bigg[\Big(\int_0^t  \|\mathcal{Q}^{1/2} h_s\|_{L^2}^2 \dd s\Big)^{\frac{p}{2}}\bigg] \quad \forall\, p\in [1,\infty)
    \end{align*}
    where $[M]$ denotes the quadratic variation process associated to $M$.
    Similarly, if $\P$-a.s. $\int_0^T \| h_s\|_{L^2}^2 \dd s <\infty$, then $N_t:= \int_0^t h_s\cdot\dd W_s$ is a is a well-defined continuous $L^2$-valued local martingale. Moreover it holds
    \begin{align*}
        [N]_t= \frac{{\rm Tr}(Q(0))}{2}\int_0^t  \|h_s\|_{L^2}^2 \dd s, \quad 
        \E\Big[ \sup_{t\in [0,T]} \|N_t\|_{L^2}^p\Big]\lesssim_p {\rm Tr}(Q(0))^{\frac{p}{2}}\E \bigg[\Big(\int_0^t  \|h_s\|_{L^2}^2 \dd s\Big)^{\frac{p}{2}}\bigg] \quad \forall\, p\in [1,\infty).
    \end{align*}
\end{lemma}

In light of the series representation \eqref{eq:noise_series_representation}, we may write the above stochastic integrals as
\begin{equation}\label{eq:convention_stoch_integrals}
	\int_0^t \langle h_s,\dd W_s\rangle = \sum_k \int_0^t \langle h_s,\sigma_k\rangle \dd W^k_s, \quad
	\int_0^t h_s(x)\cdot\dd W_s(x) = \sum_k \int_0^t h_s(x)\cdot \sigma_k(x) \dd W^k_s.
\end{equation}

With these preparations, we are finally ready to give an explicit description of the noise we will use. For any $\alpha\in (0,1)$, we say that $W$ is an incompressible Kraichnan noise of parameter $\alpha$ if its covariance function $Q$ is given by
\begin{equation}\label{eq:defn-Qalpha}
    Q(z)= (2\pi)^{-1} \int_{\R^2} (1+|\xi|^2)^{-1-\alpha} P^\perp_\xi\, e^{i\xi\cdot z} \dd \xi,
\end{equation}
in other words if $q(\xi)=(1+|\xi|^2)^{-1-\alpha}$. In this case, it can be shown that $Q$ is $C^{2\alpha}$-regular and $\P$-a.s.\ $W\in C^{\alpha-\eps}_\loc$ for all $\eps>0$, but nowhere $C^{\alpha+\eps}$-regular; in particular, $W$ is not Lipschitz continuous.
Standard arguments based on isotropy (and the fact that $d=2$) imply that
\begin{equation}\label{eq:Q(0)}
    Q(0) = 2 c_0 I\quad\text{for}\quad c_0 = \frac{{\rm Tr}(Q(0))}{4} = \frac{1}{8\pi} \int_{\R^2} (1+|\xi|^2)^{-1-\alpha} \dd \xi.
\end{equation}

\subsection{Properties of the nonlinearity and notion of weak solutions}\label{subsec:weak_sol}

Having described in detail the noise $W$ and stochastic integrals with respect to it, we can pass to define rigorously the SPDE \eqref{eq:stochastic_gSQG} and provide a notion of weak solutions.

To this end, we first need to recall some properties of the nonlinear term in the SPDE. By formula \eqref{eq:generalised_Biot_Savart}, we may write the quadratic nonlinearity in \eqref{eq:stochastic_gSQG_intro} as
\begin{equation}\label{eq:SQG_nonlinearity}
    \mathcal{N}(\theta) = \nabla\cdot [ (K_\beta\ast \theta)\, \theta]
\end{equation}
Notice that, being a Fourier multiplier with symbol $-\nabla^\perp |\nabla|^{\beta-2}$, $K_\beta$ maps $L^p$ into $\dot L^{1-\beta,p}$, for any $p\in (1,\infty)$. In particular, if $p<2/(1-\beta)$, then by Sobolev embeddings
\begin{align*}
    K_\beta\ast \theta\in L^q\quad \text{for} \quad \frac{1}{q}=\frac{1}{p}-\frac{1-\beta}{2}
\end{align*}
and so by H\"older's inequality
\begin{equation}\label{eq:basic_estim_nonlinearity}
    (K_\beta\ast \theta)\,\theta\in L^r\quad \text{for} \quad \frac{1}{r}=\frac{1}{p}+\frac{1}{q}=\frac{2}{p}-\frac{1-\beta}{2}, \quad
    \| (K_\beta\ast \theta) \theta\|_{L^r} \lesssim_{\beta,p} \| \theta \|_{L^p}^2.
\end{equation}
Combining these facts, we can conclude that $(K_\beta\ast \theta)\,\theta$ is a well-defined $L^1_\loc$ function (and thus $\mathcal{N}(\theta)$ is a well-defined distribution) as soon as
\begin{equation}\label{eq:condition_nonlinearity}
    \theta\in L^p\quad \text{with}\quad \frac{4}{3-\beta}\leq p < \frac{2}{1-\beta}.
\end{equation}
In particular, this condition is always satisfied if $\theta\in L^1\cap L^p$ with $p\geq 4/(3-\beta)$.

Next, we need to spend a few words about the Stratonovich formulation of the SPDE \eqref{eq:stochastic_gSQG_intro} and its It\^o reformulation.
Although the Stratonovich formalism is the more correct one from the physical point of view, being related to the Wong--Zakai principle and the Lagrangian representation of solutions, when dealing with low regularity solutions it is very convenient to rewrite the equation in It\^o form.
As standard in stochastic fluid dynamics equations, we will derive the It\^o version of the SPDE from the Stratonovich one as if all terms involved were regular enough, and then work systematically in It\^o form.
Rigorously defining the Stratonovich form of the equation can be a delicate issue, see \cite[Section 2.2]{GalLuo2023} and \cite[Section 2.3]{goodair2022} for a deeper discussion.

Therefore assume for the moment we are given a regular solution to \eqref{eq:stochastic_gSQG_intro}. With the notations \eqref{eq:SQG_nonlinearity}-\eqref{eq:convention_stoch_integrals} in mind, applying the standard rules of stochastic calculus to pass from Stratonovich to It\^o, we thus have
\begin{align*}
    \dd \theta_t(x) + \mathcal{N}(\theta_t)(x) \dd t
    & = - \sum_k \sigma_k(x)\cdot\nabla\theta_t(x) \circ \dd W^k_t\\
    & = - \sum_k \sigma_k(x)\cdot\nabla\theta_t(x) \dd W^k_t + \frac{1}{2} \sum_k \sigma_k(x)\cdot\nabla(\sigma_k(x)\cdot \nabla \theta_t(x)) \dd t\\
    & = - \nabla\theta_t(x) \cdot \dd W_t(x) + \frac{1}{2}\nabla\cdot \bigg( \sum_k \sigma_k(x)\otimes \sigma_k(x)\, \nabla\theta_t(x)\bigg) \dd t\\
    & = - \nabla\theta_t(x) \cdot \dd W_t(x) + \frac{1}{2} \nabla\cdot \big(Q(0) \nabla\theta_t(x)\big) \dd t\\
    & = - \nabla\theta_t(x) \cdot \dd W_t(x) + c_0 \Delta\theta_t(x) \dd t;
\end{align*}
in the above, we used the fact that $\sigma_k$ are divergence-free, $Q$ admits representation \eqref{eq:covariance_series_representation} and finally our definition of $c_0$ from \eqref{eq:Q(0)}.

In light of the above, the previous estimates on $\mathcal{N}$ and the content of Section \ref{subsec:structure_noise}, we can now rigorously define weak solutions (both in the analytical and probabilistic sense) to the stochastic gSQG equations as follows.

\begin{definition}\label{defn:weak_solution}
    Let $\alpha,\beta\in (0,1)$, $Q^\alpha$ be given by \eqref{eq:defn-Qalpha} and $\{\sigma_k\}_k$ be any collection of smooth divergence free velocity fields such that \eqref{eq:covariance_series_representation} holds.
    Let $p\geq 4/(3-\beta)$, $\theta_0\in L^1\cap L^p$ and $f\in L^1_t (L^1\cap L^p)$ be deterministic data.
    A weak $L^1\cap L^p$ solution to the stochastic gSQG equation 
    \begin{equation}\label{eq:stochastic_gSQG}
        \dd \theta_t + (K_\beta\ast\theta_t)\cdot\nabla \theta_t \dd t + \circ \dd W_t\cdot\nabla \theta_t = f_t \dd t, \quad \theta\vert_{t=0}=\theta_0
    \end{equation}
    is a tuple $(\Omega,\mathcal{A},\mathcal{F}_t,\P,(W^k)_{k\in\N},\theta)$ satisfying the following:
    \begin{enumerate}
        \item[i)] $(\Omega,\mathcal{A},\mathcal{F}_t,\P)$ is a filtered probability space satisfying the standard assumptions and $(W^k)_{k\in\N}$ is a sequence of independent real $\mathcal{F}_t$-Brownian motions;
        \item[ii)] $\theta:\Omega\times [0,T]\to L^1\cap L^p$ is a $\mathcal{F}_t$-progressively measurable process, $\theta\in L^2_{\omega,t} (L^1\cap L^p)$ and its trajectories $t\mapsto \theta_t$ are $\P$-a.s. weakly continuous in the sense of distributions;
        \item[iii)] For any $\psi\in C^\infty_c$, $\P$-a.s. for all $t\in [0,T]$ it holds
        \begin{equation}\label{eq:defn_weak_solution}
            \langle \theta_t,\psi\rangle
            = \langle\theta_0,\psi\rangle + \int_0^t \big[ \langle (K_\beta\ast \theta_r)\theta_r, \nabla\psi\rangle + c_0 \langle \theta_r, \Delta\psi\rangle + \langle f_r, \psi\rangle\big] \dd r  + \sum_k \int_0^t \langle \sigma_k\theta_r, \nabla\psi_r\rangle \dd W^k_r
    \end{equation}
    \end{enumerate}
\end{definition}

We sometimes say that $\theta$ is a weak solution on a tuple $(\Omega,\mathcal{A},\mathcal{F}_t,\P,(W^k)_k)$ if $(\Omega,\mathcal{A},\mathcal{F}_t,\P,(W^k)_k,\theta)$ is a weak solution.

\begin{remark}\label{rem:defn_weak_solution}
    Under the above assumptions, all integrals appearing on the r.h.s. of \eqref{eq:defn_weak_solution} are well-defined continuous stochastic processes. For instance, for the nonlinear term, thanks to \eqref{eq:basic_estim_nonlinearity} we have
    \begin{align*}
        \E\bigg[ \int_0^T |\langle (K_\beta\ast \theta_r)\theta_r, \nabla\psi\rangle| \dd r\bigg]
        & \leq \| \nabla\psi\|_{L^\infty} \E\bigg[ \int_0^T \| (K_\beta\ast \theta_r)\theta_r\|_{L^1} \dd r\bigg]\\
        & \lesssim \| \nabla\psi\|_{L^\infty} \E\bigg[ \int_0^T \|\theta_r\|_{L^{\frac{4}{3-\beta}}}^2 \dd r\bigg]
        =  \| \nabla\psi\|_{L^\infty}  \| \theta\|_{L^2_{\omega,t} L^{\frac{4}{3-\beta}}}^2
    \end{align*}
    where the last quantity is finite since $\theta\in L^2_{\omega,t} (L^1\cap L^p)$ with $p\geq 4/(3-\beta)$. The terms related to $c_0 \Delta$ and $f$ can be estimated similarly.
    For the stochastic integral in \eqref{eq:defn_weak_solution}, recalling our convention \eqref{eq:convention_stoch_integrals} and Lemma \ref{lem:stochastic_integrals}, it holds
    \begin{align*}
        \E\bigg[ \sup_{t\in [0,T]} \Big| \int_0^\cdot \langle \theta_r \nabla\varphi, \dd W_r\rangle\Big|^2 \bigg]
        & \lesssim \E\bigg[ \int_0^T \| \mathcal{Q}^{1/2} (\theta_r \nabla\varphi)\|_{L^2}^2 \dd r \bigg]\\
        & \lesssim \| \nabla\varphi\|_{L^\infty}^2 \E\bigg[ \int_0^T \| \theta_r \|_{L^1}^2 \dd r \bigg] = \| \nabla\varphi\|_{L^\infty}^2 \| \theta\|_{L^2_{\omega,t} L^1}^2
    \end{align*}
    where we applied Lemma \ref{lem:properties_Q^1/2}-i). As these computations suggest, Definition \ref{defn:weak_solution} admits several variants, for instance by considering solutions defined only up to a stopping time $\tau$, or by weakening the integrability on $\theta$ to just requiring that $\P$-a.s. $\int_0^T \| \theta_r\|_{L^{\frac{4}{3-\beta}}}^2 \dd r <\infty$.
    For simplicity we will not pursue this direction, since the transport structure of the equation and the upcoming Theorem \ref{thm:weak_existence} provide us with weak solutions satisfying much better a priori bounds.
\end{remark}

\begin{remark}\label{rem:weak_solution_bochner}
    In the setting of Definition \ref{defn:weak_solution}, if additionally $p\geq 2$, then the weak formulation \eqref{eq:defn_weak_solution} based on testing against $\psi\in C^\infty_c$ can be equivalently understood as an integral identity
    \begin{equation}\label{eq:weak_solution_bochner}
            \theta_\cdot
            = \theta_0 + \int_0^\cdot \nabla\cdot\big[ (K_\beta\ast \theta_r)\theta_r \big] \dd r + c_0 \int_0^\cdot \Delta \theta_r \dd r + \int_0^\cdot f_r \dd r + \int_0^\cdot \nabla\theta_r\cdot \dd W_r
    \end{equation}
    where all processes involved take values in $H^{-2}$. In particular, the first three integrals are meaningful in the Lebesgue--Bochner sense, while the last one as a stochastic integral in $H^{-1}$. Indeed, for the nonlinear term, by \eqref{eq:basic_estim_nonlinearity} and Sobolev embeddings it holds
    \begin{align*}
        \E\bigg[ \int_0^T \big\|\nabla\cdot\big[ (K_\beta\ast \theta_r)\theta_r\big]\big\|_{H^{-2}} \dd r\bigg]
        \lesssim \E\bigg[ \int_0^T \big\|(K_\beta\ast \theta_r)\theta_r\big\|_{L^{\frac{2}{1+\beta}}} \dd r\bigg]
        \lesssim \int_0^T  \E\big[ \| \theta_r\|_{L^2}^2\big] \dd r;
    \end{align*}
    the terms associated to $\Delta \theta$, $f$ can be treated similarly. For the stochastic integral, using the fact that $W$ is divergence-free and Lemma \ref{lem:stochastic_integrals}, it holds
    \begin{align*}
        \E\bigg[ \sup_{t\in [0,T]} \Big\| \nabla\cdot\Big( \int_0^\cdot \theta_r \dd W_r\Big) \Big\|_{H^{-1}}^2\bigg]
        \leq \E\bigg[ \sup_{t\in [0,T]} \Big\|  \int_0^\cdot \theta_r \dd W_r \Big\|_{L^2}^2\bigg]
        \lesssim \int_0^T  \E\big[ \| \theta_r\|_{L^2}^2\big] \dd r.
    \end{align*}
\end{remark}

\section{Weak existence}\label{sec:existence}

The aim of this section is to prove the following result.

\begin{theorem}\label{thm:weak_existence}
    Let $p\in [2,\infty]$, $\theta_0\in L^1\cap L^p$, $f\in L^1_t(L^1\cap L^p)$ be deterministic data. Then there exists a weak solution $\theta$ to \eqref{eq:stochastic_gSQG}, in the sense of Definition \ref{defn:weak_solution}.
    Moreover $\P$-a.s. $\theta$ has weakly continuous trajectories in $L^2$ and it satisfies the pathwise bound
    \begin{equation}\label{eq:apriori_bounds}
	    \sup_{t\in[0,T]} \| \theta_t \|_{L^q} \le \| \theta_0 \|_{L^q} + \int_0^T \| f_t\|_{L^q} \dd t \quad \forall\,q\in [1,p].
    \end{equation}
    Let $\eps>0$ be fixed. Then such a weak solution can be constructed in such a way that it satisfies the following decomposition: $\theta=\theta^<+\theta^>$, where $\theta^{\lessgtr}$ are $\cF_t$-adapted processes with trajectories in $C_t H^{-2}$, such that we have the $\P$-a.s. bounds
    \begin{equation}\label{eq:decomposition}
        \sup_{t\in [0,T]} \|\theta^>\|_{L^{p}} \leq \eps, \quad
        \sup_{t\in [0,T]} \|\theta^<\|_{L^q} \leq C_q \quad\forall\, q\in [\tilde p,\infty]
    \end{equation}
    where $C_q$ is a deterministic constant which depends on $q$, $\theta_0,\, f$ and $\varepsilon$, but not on the weak solution.
\end{theorem}

This result will be proved by a standard stochastic compactness argument. To this end, we will first introduce a family of regularized equations and derive a priori estimates in Section \ref{subsec:tightness}, then deduce tightness and pass to the limit in Section \ref{subsec:passage}.

\subsection{Regular approximations and tightness}\label{subsec:tightness}

We set ourselves on a filtered probability space $(\Omega,\mathcal{A},\mathcal{F}_t,\P)$ satisfying the standard assumptions, carrying a family $(W^k)_{k\ge1}$ of independent Brownian motions. Correspondingly, we can construct on such space a noise $W$ with covariance $Q$ by means of formula \eqref{eq:noise_series_representation}.
Let $\{\rho^\delta\}_{\delta>0}$ be a family of standard mollifiers associated to a radially symmetric probability density $\rho\in C^\infty_c$; we proceed to mollify all the terms of the SPDE \eqref{eq:stochastic_gSQG}, namely we set
\begin{align*}
    \theta^\delta_0:=\rho^\delta\ast \theta_0, \quad
    f^\delta_t:= \rho^\delta\ast f,\quad
    K^\delta_\beta := \rho^\delta\ast K_\beta,\quad
    W^\delta_t := \rho^\delta\ast W_t = \sum_k (\rho^\delta \ast \sigma_k) W^k_t =: \sum_k \sigma^\delta_k W^k_t.
\end{align*}
Notice that, by construction, the noise $W^\delta$ has covariance operator
\begin{align*}
    Q^\delta(x-y)
    & = \sum_k \sigma^\delta_k(x)\otimes \sigma^\delta_k(y)
    = \sum_k \int_{\R^2\times \R^2} \sigma_k(x-z)\otimes \sigma_k(y-z') \rho^\delta(z) \rho^\delta(z') \dd z \dd z'\\
    & = (Q\ast \rho^\delta\ast \rho^\delta)(x-y)
    = (Q\ast (\rho\ast \rho)^\delta)(x-y)
\end{align*}
where $(\rho\ast \rho)^\delta=\delta^{-2}(\rho\ast \rho)(\delta^{-1}\cdot)$. In particular, $Q^\delta$ is still a mollification of $Q$, associated to the probability density $\rho\ast\rho\in C^\infty_c$. Correspondigly, the noise $W^\delta$ is infinitely smooth in space, since $W$ belongs to $C_t L^2_\loc$.

With these preparations, for any $\delta>0$, we introduce the regularized stochastic gSQG model given by
\begin{equation}\label{eq:reg_gSQG}
		\dd \theta^\delta + (K^\delta_\beta\ast \theta^\delta)\cdot \nabla\theta^\delta \dd t +  \sum_{k} \sigma_k^\delta \cdot \nabla\theta^\delta \circ \dd W^k= f^\delta \dd t, \quad \theta^\delta\vert_{t=0}=0.
\end{equation}
We have the following lemma providing existence and a priori bounds for \eqref{eq:reg_gSQG}.
\begin{lemma}\label{lem:uniform_bounds}
	Let $p\in [1,\infty)$, $\theta_0\in L^1\cap L^p$, $f\in L^1_t (L^1\cap L^p)$. Then for any $\delta>0$, there exists a probabilistically strong, spatially smooth solution $\theta^\delta$ to \eqref{eq:reg_gSQG}; moreover $\P$-a.s. it holds
	\begin{equation}\label{eq:apriori_pathwise_bounds}
	    \sup_{t\in[0,T]} \| \theta_t^\delta \|_{L^q} \le \| \theta_0 \|_{L^q} + \int_0^T \| f_t\|_{L^q} \dd t \quad \forall\,q\in [1,p].
	\end{equation}
    Additionally, for any $\eps>0$, the solutions $\theta^\delta$ can be decomposed as $\theta^\delta=\theta^{\delta,>} + \theta^{\delta,<}$ in such a way that, for any $\delta>0$ fixed, $\P$-a.s. $\theta^{\delta,<}\in L^\infty_t (L^p\cap L^\infty)$ and we have the bounds
    \begin{align*}
        \sup_{t\geq 0} \|\theta^{\delta,>}\|_{L^p} \leq \eps, \quad \sup_{t\geq 0} \|\theta^{\delta,<}\|_{L^q} \leq C_q \quad\forall\, q\in [p,\infty]
    \end{align*}
    where $C_q$ is a deterministic constant which depends only on $\theta_0$, $f$ and $\varepsilon$, but not on $\delta$.
\end{lemma}

\begin{proof}
    By construction, $\theta^\delta_0$ and $f^\delta$ are spatially smooth, and moreover $K^\delta_\beta$ is a smooth kernel.
    For these reasons, strong existence of a smooth solution to \eqref{eq:reg_gSQG} holds by standard arguments, e.g.\ by adopting Kunita's approach \cite[Section 6]{Kunita1990} or by performing a similar contraction argument as in \cite{CogFla2016}, further generalised in \cite[Appendix C]{coghi2023existence} (and then boostrapping regularity).
    
    By employing either of the above references, one ends up obtaining a representation formula for $\theta^\delta$ through stochastic characteristics as
    \begin{equation}\label{eq:representation_characteristics}
        \theta^\delta_t(X_t(x))= \theta^\delta_0 (x) + \int_0^t f^\delta_s(X_s(x)) \dd s,
    \end{equation}
    where $X$ is the stochastic flow associated to the SDE
    \begin{align*}
        \dd X_t = u^\delta(X_t) \dd t + W^\delta (\circ \dd t, X_t)=: (K^\delta_\beta\ast \theta^\delta_t)(X_t) \dd t + \sum_k \sigma^\delta(X_t) \circ \dd W^k_t;
    \end{align*}
    notice that, since $u^\delta$ and $W^\delta$ are smooth and divergence-free, the stochastic flow $X$ leaves the Lebesgue measure invariant.
    As a consequence, taking the $L^q$-norm on both side of \eqref{eq:representation_characteristics} and applying Minkowski's inequality, we obtain the $\P$-a.s. estimate
    \begin{align*}
        \| \theta^\delta_t\|_{L^q}
        \leq \| \theta^\delta_0\|_{L^q} + \int_0^t \| f^\delta_s\|_{L^q} \dd s
        \leq \| \theta_0\|_{L^q} + \int_0^t \| f_s\|_{L^q} \dd s
    \end{align*}
    where in the second step we used the property of standard mollifiers that $\| \rho^\delta\ast g\|_{L^q} \leq \| g\|_{L^q}$ for any $g$. This proves \eqref{eq:apriori_pathwise_bounds}.

    It remains to show the decomposition property into $\theta^{\delta,>}+\theta^{\delta,<}$.
    For a parameter $R>0$ large enough to be chosen later, let us set
    \begin{align*}
        \theta_0^{\delta,> R}:=\theta_0^\delta  \mathbbm{1}_{|\theta^\delta|> R}, \quad
        \theta_0^{\delta,\leq R}:=\theta_0^\delta  \mathbbm{1}_{|\theta^\delta|\leq R}, \quad 
        f^{\delta,> R}:=f^\delta  \mathbbm{1}_{|f^\delta|> R}, \quad
        f^{\delta,\leq R}:=f^\delta  \mathbbm{1}_{|f^\delta|\leq R}; 
    \end{align*}
    and correspondingly let us define $\theta^{\delta,>R}$ and $\theta^{\delta,\leq R}$ implicitly by
    \begin{align*}
        &\theta^{\delta,> R}_t(X_t(x))= \theta^{\delta,>R}_0 (x) + \int_0^t f^{\delta,>R}_s(X_s(x)) \dd s,\\
        &\theta^{\delta,\leq R}_t(X_t(x))= \theta^{\delta,\leq R}_0 (x) + \int_0^t f^{\delta,\leq R}_s(X_s(x)) \dd s.
    \end{align*}
    In light of \eqref{eq:representation_characteristics}, it is clear that $\theta^\delta=\theta^{\delta,>R}+\theta^{\delta,\leq R}$.
    Notice that, by properties of mollifiers, it holds $\theta^\delta_0\to \theta_0$ in $L^p$ and $f^\delta\to f$ in $L^1_t L^p$ as $\delta\to 0$. As a consequence, we can invoke Lemma \ref{lem:decomposition} and repeat the argument used to derive \eqref{eq:apriori_pathwise_bounds}, to find the $\P$-a.s. estimates
    \begin{equation}\label{eq:estim_decomp1}
        \sup_{\delta>0} \sup_{t\in [0,T]} \| \theta^{\delta,>R}\|_{L^p}
        \leq \sup_{\delta>0} \| \theta_0^\delta  \mathbbm{1}_{|\theta^\delta|> R}\|_{L^p} + \sup_{\delta>0} \int_0^T \| f^\delta_s  \mathbbm{1}_{|f^\delta_s|> R}\|_{L^p} \dd s
    \end{equation}
    and
    \begin{equation}\label{eq:estim_decomp2}\begin{split}
         \sup_{t\in [0,T]} \| \theta^{\delta,\leq R}\|_{L^q}
         & \leq \| \theta_0^\delta  \mathbbm{1}_{|\theta^\delta|\leq R}\|_{L^q} + \int_0^t \| f^\delta_s  \mathbbm{1}_{|f^\delta_s|\leq R}\|_{L^q} \dd s\\
         & \lesssim_T R^{1-\frac{p}{q}} \Big(  \| \theta^{\delta}_0\|_{L^p}^{\frac{p}{q}} + \int_0^T \| f^\delta_s\|_{L^p} \dd s\Big)
         \leq R^{1-\frac{p}{q}} \Big(  \| \theta_0\|_{L^p}^{\frac{p}{q}} + \int_0^T \| f_s\|_{L^p} \dd s\Big).
    \end{split}\end{equation}
    By Lemma \ref{lem:decomposition}, we can choose $R$ such that the r.h.s. of \eqref{eq:estim_decomp1} becomes arbirarily small, in particular, smaller than $\eps$. Having fixed such $R$, the constant $C_q$ is then determined by the r.h.s. of \eqref{eq:estim_decomp2}.
\end{proof}

We now want to obtain uniform-in-$\delta$ eatimates for the time continuity of the solutions $\theta^\delta$, possibly in weak topologies. To this end, rather than working with \eqref{eq:reg_gSQG}, it is convenient to pass to consider its equivalent It\^o formulation; arguing as in Section \ref{subsec:weak_sol}, this is given by
\begin{equation}\label{eq:reg_gSQG_ito}
    \dd \theta^\delta + (K^\delta_\beta\ast \theta^\delta)\cdot \nabla\theta^\delta \dd t -c_\delta \Delta \theta^\delta  + \sum_{k} \sigma_k^\delta \cdot \nabla\theta^\delta \dd W^k= f^\delta \dd t, \quad \theta^\delta\vert_{t=0}=\theta^\delta_0.
\end{equation}
where $c_\delta = {\rm Tr} (Q^\delta(0))/4$. Recalling that $Q$ is a bounded, H\"older continuous function and that $Q^\delta$ is a mollified approximation, it follows from \eqref{eq:Q(0)} that
\begin{equation}\label{eq:convergence_c_delta}
    \sup_{\delta>0}\, c_\delta \lesssim \| Q\|_{L^\infty}, \quad \lim_{\delta\to 0} c_\delta = c_0 = \frac{1}{4} {\rm Tr} (Q(0)).
\end{equation}
In order to establish uniform bounds for \eqref{eq:reg_gSQG_ito}, it is useful to collect here a few properties of the (mollified) nonlinearity of the PDE; for simplicity, we already state here also a convergence property as $\delta\to 0$, which will be used later in Section \ref{subsec:passage} when passing to the limit.
Notice that, since $K_\beta$ and $\sigma$ are divergence-free (thus also $K^\delta_\beta$ and $\sigma^\delta$), we have
\begin{align*}
    (K_\beta^\delta\ast\theta^\delta)\cdot\nabla\theta^\delta = \div[(K_\beta^\delta \ast \varphi) \varphi], \quad 
    \sigma^\delta\cdot\nabla \theta^\delta = \div(\sigma^\delta \varphi).
\end{align*}

\begin{lemma}\label{lem:properties_nonlinearity}
    Let $\beta\in (0,1)$. For $\delta\geq 0$, define
    \begin{equation}\label{eq:notation_nonlinearity}
        \mathcal{N}^\delta( \varphi) := \div[(K_\beta^\delta \ast \varphi) \varphi], \quad
        \mathcal{N}(\varphi)=\mathcal{N}^0(\varphi):=\div[(K_\beta \ast \varphi) \varphi].
    \end{equation}
    Then the following hold:
    \begin{enumerate}
        \item[i)] There exists a constant $C=C_\beta>0$ such that $\sup_{\delta\geq 0}\| \mathcal{N}^\delta( \varphi)\|_{\dot H^{-1-\beta}} \leq C \| \varphi\|_{L^2}^2$.
        \item[ii)] If $(\varphi^\delta)_{\delta>0}\subset L^2$ is such that $\varphi^\delta\rightharpoonup \varphi$ weakly in $L^2$, then $\mathcal{N}^\delta(\varphi^\delta)\rightharpoonup \mathcal{N}(\varphi)$ weakly in $H^{-1-\beta}$.
    \end{enumerate} 
\end{lemma}

\begin{proof}
    i): By similar estimates as in \eqref{eq:condition_nonlinearity}, if $\varphi\in L^2$ then $K^\delta_\beta\ast \varphi\in L^{2/\beta}$ with
    \begin{align*}
        \sup_{\delta \geq 0} \| K^\delta_\beta\ast \varphi\|_{L^\frac{2}{\beta}}
        =  \sup_{\delta \geq 0} \| \rho^\delta\ast (K_\beta\ast \varphi)\|_{L^\frac{2}{\beta}}
        \leq \| K_\beta\ast \varphi\|_{L^\frac{2}{\beta}}
        \lesssim \| K_\beta\ast \varphi\|_{\dot H^{1-\beta}}
        \lesssim \| \varphi\|_{L^2}.
    \end{align*}
    By H\"older's inequality, we deduce that
    \begin{align*}
        \sup_{\delta \geq 0} \| (K^\delta_\beta\ast \varphi)\,\varphi\|_{L^\frac{2}{\beta+1}}
        \leq  \sup_{\delta \geq 0} \| (K^\delta_\beta\ast \varphi)\|_{L^\frac{2}{\beta}} \| \varphi\|_{L^2}
        \lesssim \| \varphi\|_{L^2}^2.
    \end{align*}
    The conclusion now follows from $\| \nabla\cdot g\|_{\dot H^{-1-\beta}}\leq \| g\|_{\dot H^{-\beta}}$ and the embedding $L^{\frac{2}{1+\beta}}\hookrightarrow \dot H^{-\beta}$.
    
    ii): By i), the sequence $\mathcal{N}^\delta(\varphi^\delta)$ is bounded in $H^{-1-\beta}$, so in order to verify weak convergence it suffices to test against smooth functions; namely, we need to show that
    \begin{align*}
        \langle (K_\beta^\delta\ast \varphi^\delta) \cdot\nabla\psi, \varphi^\delta\rangle \to \langle (K_\beta\ast \varphi) \cdot\nabla\psi, \varphi\rangle \quad\forall\,\varphi\in C^\infty_c.
    \end{align*}
    By the assumption $\varphi^\delta \rightharpoonup \varphi$ and weak-strong convergence, it suffices to show that $K_\beta^\delta\ast \varphi^\delta\to K_\beta\ast \varphi$ in $L^2_{\loc}$. We now split this term as
    \begin{align*}
        K_\beta^\delta\ast \varphi^\delta- K_\beta\ast \varphi
        = (\rho^\delta - \delta_0)\ast (K_\beta\ast \varphi) + \rho^\delta\ast (K_\beta\ast \varphi^\delta - K_\beta\ast\varphi).
    \end{align*}
    The first term converges to $0$ in $L^{2/\beta}$, thus also in $L^2_\loc$, by standard properties of mollifiers.
    For the second term, we can use the compactness of the operator $g\mapsto K_\beta\ast g$ from $L^2$ to $L^2_\loc$ (due to compact embedding $H^{1-\beta}(B_R)\hookrightarrow L^2(B_R)$ on bounded balls $B_R$) and our standing assumption $\varphi^\delta\rightharpoonup\varphi$, to deduce that $K_\beta\ast \varphi^\delta \to K_\beta\ast\varphi$ in $L^2_\loc$. This convergence is then preserved by the mollifier $\rho^\delta$.
\end{proof}

Armed with Lemma \ref{lem:properties_nonlinearity}, we can now establish H\"older continuity of the solutions to the regularized model.

\begin{lemma}\label{lem:time_continuity}
    Let $\theta^\delta$ be a solution to \eqref{eq:reg_gSQG} and set $v^\delta_t:= \theta^\delta_t-\int_0^t f^\delta_s \dd s$.
	Then, for any $\gamma \in (0,1/2)$ and $\eta\in [2,\infty)$, there exists a constant $C=C(\gamma,\eta,\| \theta_0\|_{L^2}, \| f\|_{L^1_t (L^1\cap L^2)})$ such that
	\begin{equation}\label{eq:holder_bound}
		\E\big[ \|v^\delta\|^\eta_{C^\gamma_t H^{-2}} \big]\le C \quad\forall\, \delta>0.
	\end{equation}
\end{lemma}

\begin{proof}
    Let us keep using the notation \eqref{eq:notation_nonlinearity}. By writing \eqref{eq:reg_gSQG_ito} in integral form, we have
    \begin{align*}
        v^\delta_t-\theta^\delta_0
        = \int_0^t \mathcal{N}^\delta(\theta_s) \dd s - c_\delta \int_0^t \Delta \theta^\delta_s \dd s + \int_0^t \nabla \theta^\delta_s\cdot \dd W^\delta_s
        \eqcolon S^1_t + S^2_t+ S^3_t.
    \end{align*}
    We estimate each term separately.
    For the nonlinear term, thanks to the uniform estimate \eqref{eq:apriori_pathwise_bounds} and Lemma \ref{lem:properties_nonlinearity}-i), we have the pathwise bound
    \begin{align*}
        \| S^1_t-S^1_s\|_{H^{-1-\beta}}
        \leq \int_s^t \| \mathcal{N}^\delta(\theta^\delta_r)\|_{H^{-1-\beta}} \dd r
        \leq |t-s| \sup_{t\in [0,T]} \| \theta^\delta_t\|_{L^2}^2
    \end{align*}
    so that
    \begin{equation}\label{eq:bound_S1}
        \E\Big[\| S^1\|_{W^{1,\infty}_t H^{-1-\beta}}^\eta\Big]
        \leq \E\Big[\sup_{t\in [0,T]} \| \theta^\delta_t\|_{L^2}^{2\eta}\Big]
        \lesssim \Big(\| \theta_0\|_{L^2} + \int_0^T \| f_t\|_{L^2} \dd t \Big)^{2\eta}.
    \end{equation}
    For $S_2$ we can follow a similar argument, using the fact that $c_\delta={\rm Tr}(Q^\delta(0))$ are uniformly bounded by \eqref{eq:convergence_c_delta} and that $\| \Delta g\|_{H^{-2}}\leq \| g\|_{L^2}$ for any $g$. As a consequence
    \begin{equation}\label{eq:bound_S2}
        \E\big[\| S^2\|_{W^{1,\infty}_t H^{-2}}^\eta\big]
        \leq \Big(\sup_{\delta >0} c_\delta^{\eta}\Big) \, \E\Big[\sup_{t\in [0,T]} \| \theta^\delta_t\|_{L^2}^{\eta}\Big]
        \lesssim \Big(\| \theta_0\|_{L^2} + \int_0^T \| f_t\|_{L^2} \dd t \Big)^{\eta}.
    \end{equation}
    It remains to handle the stochastic integral term. Let us preliminary observe that, since $W^\delta$ is divergence free, $S^3_t = \nabla\cdot( \int_0^t \theta^\delta_s \dd W^\delta_s)=: \nabla\cdot \tilde S^3_t$.    
    By virtue of Lemma \ref{lem:stochastic_integrals} and the uniform estimate \eqref{eq:apriori_pathwise_bounds}, for any $s\leq t$ it holds
    \begin{align*}
        \E\big[ \| S^3_t-S^3_s\|^\eta_{H^{-1}} \big]
        & \leq \E\big[ \| \tilde S^3_t-\tilde S^3_s\|^\eta_{L^2} \big]
        \lesssim (c_\delta)^{\eta/2} \E\bigg[ \Big(\int_s^t \| \theta^\delta_r\|_{L^2}^2 \dd r \Big)^{\frac{\eta}{2}} \bigg]\\
        & \lesssim |t-s|^{\eta/2} \Big(\sup_{\delta >0} c_\delta^{\eta/2}\Big) \, \E\Big[\sup_{t\in [0,T]} \| \theta^\delta_t\|_{L^2}^{\eta}\Big]\\
        & \lesssim |t-s|^{\eta/2} \Big(\| \theta_0\|_{L^2} + \int_0^T \| f_t\|_{L^2} \dd t \Big)^{\eta}.
    \end{align*}
    By the arbitrariness of $\eta\in [2,\infty)$ and Kolmogorov's continuity theorem in Banach spaces, we deduce that for any $\gamma<1/2$ and any $\eta\in [2,\infty)$ it holds that
    \begin{equation}\label{eq:bound_S3}
        \E\big[ \| S^3\|^\eta_{ C^\gamma_t H^{-1}} \big]\lesssim_{\gamma,\eta}  \Big(\| \theta_0\|_{L^2} + \int_0^T \| f_t\|_{L^2} \dd t \Big)^{\eta}.
    \end{equation}
    Combining the bounds \eqref{eq:bound_S1},\eqref{eq:bound_S2} and \eqref{eq:bound_S3}, which are all uniform in $\delta$, we obtain \eqref{eq:holder_bound}.
\end{proof}

For the sake of showing tightness of the laws of the processes $\theta^\delta$ and $u^\delta$, we introduce suitable weighted Sobolev spaces, in order to overcome the difficulties coming from working on the whole $\R^2$.
Let us define the weight $w(x)=(1+|x|^2)^{-2}$; correspondingly, for $s\in\R$, we define the weighted Sobolev space $H^{s}_w$ as the closure of smooth functions under the norm\begin{equation*}
    \| \varphi\|_{H^{s}_w} := \| \varphi\, w\|_{H^{s}},
\end{equation*}
for any $s\in\R$. $H^s_w$ defines an Hilbert space; properties of $H^s_w$ are recalled in Appendix \ref{app:analysis_lemmas}.

\begin{lemma}\label{lem:ascoli_arzela}
    For any $T\in (0,+\infty)$, $\gamma>0$ and $\eps>0$, we have the compact embedding
    \begin{align*}
		L^\infty([0,T];L^2) \cap C^\gamma([0,T];H^{-2}) \hookrightarrow C([0,T];H^{-\eps}_w).
	\end{align*}
\end{lemma}

\begin{proof}
    By Lemma \ref{lem:weighted_negative_sobolev}-ii), the embedding $H^{-2}\hookrightarrow H^{-2-\eps}_w$ is compact, which implies by Ascoli--Arzelà that the embedding $C^\gamma([0,T];H^{-2}) \hookrightarrow C([0,T];H^{-2-\eps}_w)$ is compact as well.
    On the other hand, since $L^2\hookrightarrow H^0_w$ by Lemma \ref{lem:weighted_negative_sobolev}-i), we have the interpolation estimate
    \begin{align*}
        \| \varphi\|_{C([0,T];H^{-\eps})}\lesssim \| \varphi\|_{L^\infty([0,T];L^2)}^{2/(2+\eps)}\, \| \varphi\|_{C([0,T];H^{-2-\eps})}^{\eps/(2+\eps)}
    \end{align*}
    as a consequence of Lemma \ref{lem:weighted_negative_sobolev}-iv). Combining these facts, we obtain the conclusion.
\end{proof}

As a consequence of Lemmas \ref{lem:time_continuity} and \ref{lem:ascoli_arzela}, we obtain the following.

\begin{corollary}\label{cor:tightness}
    For any $\eps>0$, the laws of $(\theta^\delta, (W^k)_{k\ge 1})_{\delta>0}$ are tight in $C_t H^{-\eps}_w \times C_t^{\N}$.
\end{corollary}

\begin{proof}
    Clearly, $(W^k)_{k\ge 1}$ is tight in $C_t^{\N}$ since it doesn't depend on $\delta$, so we only need to verify tightness of $\{\theta^\delta\}_{\delta>0}$.
    Recall that $\theta^\delta=v^\delta + \int_0^\cdot f^\delta_s \dd s$; by construction, $f^\delta\to f$ in $L^1_t L^2_x$, therefore $\int_0^\cdot f^\delta_s \dd s\to \int_0^\cdot f_s \dd s$ in $C_t L^2 \hookrightarrow C_t H^{-\eps}_w$. Since they converge therein, the functions $\{ \int_0^\cdot f^\delta_s \dd s\}_{\delta>0}$ are tight in $C_t H^{-\eps}_w$. On the other hand, a combination of Lemmas \ref{lem:uniform_bounds} and \ref{lem:time_continuity} yields
    \begin{align*}
        \sup_{\delta>0} \E\Big[ \| v^\delta\|_{L^\infty_t L^2}^2 + \| v^\delta \|_{C^\gamma_t H^{-2}}^2 \Big]<\infty
    \end{align*}
    for any $\gamma\in (0,1/2)$. Together with Markov's inequality and Lemma \ref{lem:ascoli_arzela}, this implies tightness of $\{v^\delta\}_{\delta>0}$ in $C_t H^{-\eps}_w$.
    Combining the results for $\int_0^\cdot f^\delta_s \dd s$ and $v^\delta$, we deduce tightness of $\{\theta^\delta\}_{\delta>0}$ in $C_t H^{-\eps}_w$ as well.
\end{proof}

\subsection{Passage to the limit and weak existence}\label{subsec:passage}

With the preparations from the previous section, we are now ready to complete the

\begin{proof}[Proof of Theorem \ref{thm:weak_existence}]
Given $\theta_0\in L^1\cap L^p$ and $f\in L^1_t (L^1\cap L^p)$, consider the sequence of solutions $\{\theta^n\}_n$ obtained by taking $\delta=1/n$ in the regularization scheme from Section \ref{subsec:tightness}. By Corollary \ref{cor:tightness}, $(\theta^n, (W^k)_{k})_{n\in\N}$ are tight in $C_t H^{-\eps}_w \times C_t^{\N}$.
Therefore we can invoke Prokhorov's and Skorokhod's theorems \cite[Theorems 5.1 and 6.7]{Billingsley1999} to find a new probability space $(\tilde\Omega,\tilde{\mathcal{F}}_t,\tilde\P)$ supporting a (not relabelled for simplicity) subsequence of random variables $\{\tilde\theta^n, (\tilde W^{k,n})_{k}\}_{n\in\N}$ such that $\text{Law}_{\P}(\theta^n, (W^k)_{k\ge 1})) = \text{Law}_{\tilde\P} (\tilde\theta^n, (\tilde W^{k,n})_{k\ge 1}))$ for every $n\in\N$ and with the property that there exists another pair $(\tilde \theta, (\tilde W^k)_{k\ge 1}))$ such that
\begin{equation}\label{eq:convergence_skorokhod}
    (\tilde\theta^n, (\tilde W^{k,n})_{k}) \to (\tilde\theta, (\tilde W^k)_{k}) \ \ \ \tilde\P\text{-a.s. in } C_t H^{-\eps}_w \times C_t^{\N}.
\end{equation}
We claim that $(\tilde\theta, (\tilde W^k)_{k})$ is the desired weak solution, w.r.t. the filtration $\tilde {\mathcal G}_t= \sigma(\tilde\theta_r, \tilde W^k_r: k\in\N, r\leq t )$.

We start by noticing that, since ${\rm Law}_{\P}(\theta^n, (W^k)_{k})) = {\rm Law}_{\tilde\P} (\tilde\theta^n, (\tilde W^{k,n})_{k}))$, by \eqref{eq:apriori_pathwise_bounds} we still have the $\tilde\P$-a.s. bound
\begin{equation}\label{eq:bounds_skorokhod}
	\sup_{t\in[0,T]} \| \tilde\theta_t^n \|_{L^q} \le \| \theta_0 \|_{L^q} + \int_0^T \| f_t\|_{L^q} \dd t \quad \forall\,q\in [1,p].
\end{equation}
Combining this with \eqref{eq:convergence_skorokhod}, lower-semicontinuity of $L^q$-norms, it's easy to deduce (cf. \cite[Lemma 3.5]{FGL2021}) that $\tilde\theta\in L^\infty_{\omega,t} (L^1\cap L^p)$, the trajectories $t\mapsto \tilde\theta_t$ are $\tilde\P$-a.s. weakly continuous in $L^q$ for any $q\in (1,p]$, and the $\tilde \P$-a.s. convergence
\begin{equation}\label{eq:convergence_skorokhod2}
    \tilde\theta^n_t \rightharpoonup \tilde\theta_t \text{ in } L^2 \quad \forall\, t\in [0,T].
\end{equation}
Again by lower-semicontinuity, this implies the validity of the bound \eqref{eq:bounds_skorokhod} with $\tilde\theta$ in place of $\tilde\theta^n$, proving \eqref{eq:apriori_bounds}.

Since $(\theta^n, (W^k)_{k})$ are strong solutions, the same holds for $(\tilde\theta^n, (\tilde W^{k,n})_{k})$; therefore for any fixed $n$, $(\tilde W^{k,n})_{k}$ are $\tilde{\mathcal{G}}^n_t$-Brownian motions, for $\tilde{\mathcal G}^n_t:= \sigma(\tilde\theta^n_r, \tilde W^{n,k}_r: k\in\N, r\leq t )$. In light of \eqref{eq:convergence_skorokhod}, a standard argument then implies that $(\tilde W^k)_{k}$ is a family of $\tilde {\mathcal G}_t$-Brownian motions, for $\tilde {\mathcal G}_t$ defined as above.

We now want to pass to the limit in the SPDE as $n\to\infty$. To this end, it is convenient to write \eqref{eq:reg_gSQG_ito} in integral form and tested against $\varphi\in C^\infty_c$, namely
\begin{equation} \label{eq:reg_gSQG_tested}
	\langle \tilde\theta^n_t, \varphi\rangle 
	= \langle \theta^n_0, \varphi \rangle 
	- \int_0^t \langle \mathcal{N}^n(\tilde\theta^n_s), \varphi \rangle \dd s 
	+ c_n \int_0^t  \langle \tilde\theta^n_s, \Delta \varphi\rangle \dd s
	+ \sum_{k}\int_0^t \langle \sigma_k^n\, \tilde\theta^n_s,  \nabla \varphi \rangle \dd \tilde W^{k,n}_s + \int_0^t f^n_s \dd s;
\end{equation}
in the above, for notational convenience we replaced all parameters $\delta$ with $n$, although technically $\delta=1/n$; the notation $\mathcal{N}^n$ for the nonlinearities must be interpreted as in Lemma \ref{lem:properties_nonlinearity}.

By construction, $\theta^n_0\to \theta_0\in L^1\cap L^p$ and $\int_0^\cdot f^n_s \dd s \to \int_0^\cdot f_s \dd s$ in $C_t(L^1\cap L^p)$.
For the nonlinearities, by \ref{lem:properties_nonlinearity}-ii) and \eqref{eq:convergence_skorokhod2} we have the $\tilde\P$-a.s. convergence $\mathcal{N}^n(\tilde\theta^n_s)\rightharpoonup \mathcal{N}(\tilde\theta_s)$ for all $s\in [0,T]$, which combined with the uniform bounds \eqref{eq:bounds_skorokhod} and dominated convergence immediately implies
\begin{align*}
    \int_0^\cdot \langle \mathcal{N}^n(\tilde\theta^n_s), \varphi \rangle \dd s \to \int_0^\cdot \langle \mathcal{N}(\tilde\theta_s), \varphi \rangle \dd s\ \ \ \ \tilde\P\text{-a.s. in } C_t
\end{align*}
A same argument holds for the term associated to $\Delta\varphi$, which together with convergence $c_n\to c_0$ (cf. \eqref{eq:convergence_c_delta}) implies
\begin{align*}
    c_n \int_0^\cdot  \langle \tilde\theta^n_s, \Delta \varphi\rangle \dd s \to c_0 \int_0^\cdot  \langle \tilde\theta_s, \Delta \varphi\rangle \dd s\ \ \ \ \tilde\P\text{-a.s. in } C_t.
\end{align*}
Convergence of the stochastic integral terms can also be treated standardly, cf. \cite[Lemma 4.3]{bagnara2023no}. Alternatively, one can check convergence by hand:  thanks to the convergences \eqref{eq:convergence_skorokhod}-\eqref{eq:convergence_skorokhod2}, as well as $\sigma^n_k\to\sigma_k$, one can employ \cite[Lemma 5.2]{GyoMat2001} to deduce that $\tilde\P$-a.s.
\begin{equation}\label{eq:convergence_stochastic_skorokhod}
    \int_0^\cdot \langle \sigma_k^n\, \tilde\theta^n_s,  \nabla \varphi \rangle \dd \tilde W^{k,n}_s \to \int_0^\cdot \langle \sigma_k\, \tilde\theta_s,  \nabla \varphi \rangle \dd \tilde W^k_s \ \ \ \tilde\P\text{-a.s. in } C_t
\end{equation}
for any fixed $k$. On the other hand, the tail of the series can be made arbitrarily small; indeed, if ${\rm supp} \varphi\subset B_R$, then by It\^o isometry and Doob's inequality we have
\begin{align*}
    \tilde\E\bigg[ \sup_{t\in [0,T]} \Big|  \sum_{k\geq N} \int_0^t \langle \sigma_k^n\, \tilde\theta^n_s,  \nabla \varphi \rangle \dd \tilde W^{k,n}_s \Big|^2\bigg]
    & \lesssim \sum_{k\geq N} \tilde\E\bigg[ \int_0^T |\langle \sigma_k^n\, \tilde\theta^n_s,  \nabla \varphi \rangle|^2 \dd s\bigg]\\
    & \lesssim \| \nabla\varphi\|_{L^\infty}^2 \sum_{k\geq N} \int_0^T \tilde\E\big[ \|\sigma_k^n\|_{L^2(B_R)}^2\, \|\tilde\theta^n_s\|_{L^2}^2\big] \dd s\bigg]\\
    & \lesssim \| \nabla\varphi\|_{L^\infty}^2\, T \, \Big( \|\theta_0\|_{L^2} + \int_0^T \| f_s\|_{L^2} \dd s\Big)^2 \sum_{k\geq N} \|\sigma_k\|_{L^2(B_{R+1})}^2
\end{align*}
where in the last step we used bound \eqref{eq:bounds_skorokhod} and the fact that, for each $k$, $\|\sigma_k^n\|_{L^2(B_R)}\leq \|\sigma_k\|_{L^2(B_{R+1})}$ by properties of mollifiers.
Recalling that the series \eqref{eq:covariance_series_representation} is absolutely convergent on compact sets by Lemma \ref{lem:chaos_expansions}, we conclude that
\begin{equation}\label{eq:tails_stochastic_skorokhod}
    \lim_{N\to\infty} \tilde\E\bigg[ \sup_{t\in [0,T]} \Big|  \sum_{k\geq N} \int_0^t \langle \sigma_k^n\, \tilde\theta^n_s,  \nabla \varphi \rangle \dd \tilde W^{k,n}_s \Big|^2\bigg] = 0;
\end{equation}
a similar estimates holds for the series of stochastic integrals taken w.r.t. $(\tilde\theta,\tilde W^k)$. Thanks to \eqref{eq:convergence_stochastic_skorokhod} and \eqref{eq:tails_stochastic_skorokhod}, overall one can conclude that
\begin{align*}
    \sum_{k\in\N}\int_0^\cdot \langle \sigma_k^n\, \tilde\theta^n_s,  \nabla \varphi \rangle \dd \tilde W^{k,n}_s
    \to \sum_{k\in\N}\int_0^\cdot \langle \sigma_k\, \tilde\theta_s,  \nabla \varphi \rangle \dd \tilde W^k_s \ \ \ \P\text{-a.s. in } C_t.
\end{align*}
Combining the convergence of each term appearing in \eqref{eq:reg_gSQG_tested}, passing to the limit as $n\to\infty$ in \eqref{eq:reg_gSQG_tested}, we conclude that for any fixed $\varphi\in C^\infty_c$, $\tilde\P$-a.s. it holds
\begin{equation*}
	\langle\tilde\theta_t, \varphi\rangle 
	= \langle\theta_0, \varphi \rangle 
	- \int_0^\cdot \langle \nabla\cdot (K_\beta\ast \tilde\theta_s) \tilde\theta_s, \varphi \rangle \dd s 
	+ c_0 \int_0^t  \langle \Delta \tilde\theta_s, \varphi\rangle \dd s
	- \sum_{k\in\N} \int_0^t \langle \sigma_k\cdot\nabla\tilde\theta_s, \varphi \rangle \dd \tilde W^k_s.
\end{equation*}
Choosing a countable collections $\{\varphi^j\}_j$ which is dense in $H^2$, by standard density arguments we can deduce that \eqref{eq:defn_weak_solution} holds, so that $\tilde \theta$ is a weak solution. Finally notice that, by virtue of Remark \ref{rem:weak_solution_bochner}, $\tilde\theta$ equivalently solves the SPDE in the integral form \eqref{eq:weak_solution_bochner} on $H^{-2}$, without testing against $\varphi$. 

It remains to show that, for any given $\varepsilon>0$, the above construction can be performed so to obtain the desired decomposition $\tilde\theta=\tilde\theta^< + \tilde\theta^>$ satisfying \eqref{eq:decomposition}.
Since $\text{Law}_{\P}(\theta^n, (W^k)_{k\ge 1})) = \text{Law}_{\tilde\P} (\tilde\theta^n, (\tilde W^{k,n})_{k\ge 1}))$, by Lemma \ref{lem:uniform_bounds} there is a decomposition $\tilde\theta^n=\tilde\theta^{n,<} + \tilde\theta^{n,>}$ such that $\tilde\P$-a.s. it holds
\begin{equation}\label{eq:decomposition.skorokhod}
    \sup_{n, t\in [0,T]}\, \|\tilde\theta^{n,>}\|_{L^p} \leq \eps, \quad \sup_{n, t\in [0,T]}\, \|\tilde\theta^{n,<}\|_{L^q} \leq C_q \quad\forall\, q\in [p,\infty].
\end{equation}
In fact, by the construction performed in Lemma \ref{lem:uniform_bounds}, $\tilde\theta^{n,\lessgtr}$ are actually solutions to the SPDEs
\begin{equation}\label{eq:decomposition_transport_approx}
    \dd \tilde\theta^{n,\lessgtr}_t + (K^\delta_\beta\ast \tilde\theta^n_t)\cdot \nabla\tilde\theta^{n,\lessgtr}_t \dd t -c_n \Delta \tilde\theta^{n,\lessgtr}_t  + \sum_{k\in\N} \sigma_k^n \cdot \nabla\tilde\theta^{n,\lessgtr}_t \dd \tilde W^{k,n}_t= f^{n,\lessgtr}_t \dd t, \quad  \tilde\theta^{n,\lessgtr}\vert_{t=0}= \tilde\theta^{n,\lessgtr}_0.
\end{equation}
Arguing as in Lemma \ref{lem:time_continuity} it's then easy to check that $\tilde v^{n,\lessgtr}=\tilde\theta^{n,\lessgtr}-\int_0^t f^{n,\lessgtr}_s \dd s$ still enjoy the bounds
\begin{equation}\label{eq:holder_bound_decomposition}
    \sup_n \tilde\E\big[ \|v^{n,\lessgtr}\|^\eta_{C^\gamma_t H^{-2}} \big]\lesssim_{\gamma,\eta} 1
\end{equation}
for any $\gamma<1/2$ and $\eta\in [1,\infty)$. By arguing as in the proof above, up to possibly further refining the subsequence or the probability space, one can pass to the limit in \eqref{eq:decomposition_transport_approx}
to conclude that $\tilde\theta^{\lessgtr}$ satisfy the SPDEs
\begin{equation}\label{eq:decomposition_transport}
    \dd \tilde\theta^\lessgtr_t + (K^\delta_\beta\ast \tilde\theta_t)\cdot \nabla\tilde\theta^\lessgtr_t \dd t -c_0 \Delta \tilde\theta^\lessgtr_t  + \sum_{k\in\N} \sigma_k \cdot \nabla\tilde\theta^\lessgtr_t \dd \tilde W^k_t= f^\lessgtr_t \dd t, \quad  \tilde\theta\lessgtr\vert_{t=0}= \tilde\theta^\lessgtr_0
\end{equation}
and moreover $\tilde v^\lessgtr:=\tilde\theta^\lessgtr-\int_0^t f^\lessgtr_s \dd s$ still belong to $C^\gamma_t H^{-2}$ with estimates of the form \eqref{eq:holder_bound_decomposition}. This implies as before that $\tilde\P$-a.s. $\tilde\theta^\lessgtr$ have weakly continuous trajectories in $L^2$ (in fact, by interpolation $\tilde v^\lessgtr\in C^\gamma_t H^{-\varepsilon}$) and that one can pass to the limit in \eqref{eq:decomposition.skorokhod} by weak-lower semicontinuity to find \eqref{eq:decomposition}.

Finally observe that, since $(\tilde\theta^n,\tilde\theta^{n,<},\tilde\theta^{n,>})$ is adapted to the filtration generated by $(\tilde W^{n,k})_{k}$, then by the usual standard arguments, upon passing to the limit in $n$, one can deduce that the tuple $(\tilde\theta,\tilde\theta^{<},\tilde\theta^{>},(\tilde W^k)_{k})$ is adapted to the filtration $\mathcal G'_t= \sigma(\tilde\theta_r, \tilde\theta^{<}_r,\tilde\theta^{>}_r,\tilde W^k_r: k\in\N, r\leq t )$ and $(\tilde W^k)_{k}$ are $\mathcal G'_t$-Brownian motions.
Overall this proves the existence of an adapted decomposition satisfying \eqref{eq:decomposition}.
\end{proof}

\begin{remark}\label{rem:extensions_existence}
    For simplicity, we only proved weak existence for $\theta_0\in L^1\cap L^2$.
    Let us also point out that in the proof, the rough Kraichnan structure of the noise didn't play any role, and the same result would be true for regular, divergence free $W$.
    It is clear that the result can be generalised for other classes of initial data $\theta_0$, up to technical details. For instance, one can easily readapt the argument to treat $\theta_0\in L^1\cap L^p$ with $p=4/(3-\beta)$, since by the estimates leading up to \eqref{eq:condition_nonlinearity} the nonlinearity $\mathcal{N}(\theta)=(K_\beta\ast \theta) \theta$ is a well-defined $L^1$ function in that case.
    In light of the result of Marchand \cite{Marchand2008} for deterministic SQG and those from \cite{jiao2024well}, we expect weak existence to hold for an even larger class of $\theta_0$, possibly depending on the roughness of $W$; we leave this question for future investigations.
\end{remark}

\section{Uniqueness and stability}\label{sec:uniqueness}

The main goal of this section is to prove the following uniqueness and stability result.

\begin{theorem}\label{thm:uniqueness}
    Let $\alpha$, $\beta$ and $p$ satisfy
    \begin{equation}\label{eq:parameters_uniqueness}
		0<\frac{\beta}{2}< \alpha< \frac{1}{2}, \quad
		\frac{\beta}{2}+\alpha \leq 1-\frac{1}{p}, \quad 2\le p<\infty.
    \end{equation}
    Let $\theta^1$, $\theta^2$ be two weak solutions on $[0,T]$, in the sense of Definition \ref{defn:weak_solution}, defined on the same tuple $(\Omega,\cA,\cF_t,\P,(W^k)_k)$, with initial conditions $\theta^1_0$, $\theta^2_0\in L^1\cap L^p$ and forcing terms $f^1_0$, $f^2_0\in L^1_t(L^1\cap L^p)$ respectively.
    Further assume that
    \begin{align*}
        \theta^i_t\in L^\infty_{\omega,t} (L^1\cap L^p).
    \end{align*}
    Then there exist a constant $C>0$, depending on $\alpha$, $\beta$, $p$, $\theta^1_0$ and $f^1$, but independent of the solutions $\theta^i$ in consideration, such that
    \begin{equation}\label{eq:stability_bd}
        \sup_{t\in [0,T]} \E\big[ \|\theta^1_t-\theta^2_t\|_{\dot{H}^{\beta/2-1}}^2\big]^{1/2}
        \leq e^{C T} \Big( \| \theta^1_0-\theta^2_0\|_{\dot H^{\beta/2-1}} + \| f^1-f^2\|_{L^1_t \dot H^{\beta/2-1}} \Big)
    \end{equation}
    Moreover, in the ``subcritical case'' $\beta/2+\alpha < 1-1/p$, we can choose $C$ to depend on $\|\theta^1_0\|_{L^1\cap L^p}$ and $\|f^1\|_{L^1_t(L^1\cap L^p)}$, rather than $(\theta^1_0,f^1)$, and to be monotone and locally bounded in these arguments.
\end{theorem}

Assuming for the moment the validity of Theorem \ref{thm:uniqueness}, we can complete the proof of our main result.

\begin{proof}[Proof of Theorem \ref{thm:main}]
    By applying Theorem \ref{thm:uniqueness} in the case $\theta^1_0=\theta^2_0$, $f^1=f^2$, we immediately deduce that pathwise uniqueness holds for \eqref{eq:stochastic_gSQG}, in the class of solutions $\theta_t\in L^\infty_{\omega,t} (L^1\cap L^p)$.
    Uniqueness in law follows from the Yamada--Watanabe theorem; combined with the weak existence result from Theorem \ref{thm:weak_existence}, strong existence on $[0,T]$ then holds as well.
    Since $T$ here can be taken finite but arbitrarily large, standard gluing arguments then imply that we can construct strong solutions $\theta$ globally defined on $[0,+\infty)$; by virtue of Theorem \ref{thm:weak_existence}, they satisfy the $\P$-a.s. pathwise estimate
    \begin{equation*}
	    \| \theta_t \|_{L^1\cap L^p} \le \| \theta_0 \|_{L^1\cap L^p} + \int_0^t \| f_s\|_{L^1\cap L^p} \dd s \quad \forall\,t\in [0,+\infty).
    \end{equation*}
    
    Now suppose we are given a sequence $(\theta^n_0, f^n)$ as in the second part of the statement, so that in particular for any finite $T$ it holds
    \begin{align*}
        \lim_{n\to\infty} \Big(\| \theta^n_0-\theta_0\|_{\dot H^{\beta/2-1}} + \int_0^T \| f^n_s-f_s\|_{\dot H^{\beta/2-1}} \dd s \Big)=0, \quad \sup_n \int_0^T \| f^n_s\|_{L^1\cap L^p} \dd s <\infty.
    \end{align*}
    Then by Theorem \ref{thm:uniqueness} we have
    \begin{equation}\label{eq:intermediate_stability}
        \lim_{n\to\infty} \sup_{t\in [0,T]} \E\big[ \|\theta^n_t-\theta_t\|_{\dot{H}^{\beta/2-1}}^2\big] = 0;
    \end{equation}
    on the other hand, since $\theta^n$ and $\theta$ satisfy uniform pathwise $L^2$-bounds (cf. \eqref{eq:apriori_bounds}), by interpolation we can upgrade \eqref{eq:intermediate_stability} to
    \begin{equation}\label{eq:intermediate_stability2}
        \lim_{n\to\infty} \sup_{t\in [0,T]} \E\big[ \|\theta^n_t-\theta_t\|_{\dot{H}^{-\delta}}^m\big] = 0
    \end{equation}
    for any $m\in [1,\infty)$ and $\delta\in (0,1-\beta/2)$.
    By the estimates considered in Lemma \ref{lem:time_continuity}, it's easy to check that the solutions $\theta^n$ are uniformly bounded in $L^\eta_\omega C^\gamma_t H^{-2}$; combined with the $L^2$-bound and again interpolation estimates, we can deduce that for any $m\in [1,\infty)$ and $\delta>0$ there exists $\gamma=\gamma(m,\delta)>0$ such that
    \begin{equation*}
        \sup_n \E\big[ \|\theta^n\|_{C^\gamma_t H^{-\delta}}^m] + \E\big[\|\theta\|_{C^\gamma_t H^{-\delta}}^m\big] <\infty.
    \end{equation*}
    Combining this uniform bound with estimate \eqref{eq:intermediate_stability2} (and the fact that $\dot H^{-s}\hookrightarrow H^{-s}$) finally allows to bring the supremum in time inside expectation and conclude the validity of \eqref{eq:intro_convergence}.
\end{proof}

\subsection{Preparations}\label{subsec:YamWat}

We present here three tools which will be fundamental for the proof of Theorem \ref{thm:uniqueness}, which will be presented in the upcoming Section \ref{subsec:proof_uniqueness}: i) Lemma \ref{lem:YamWat} a Yamada--Watanabe type result, needed to construct coupling between different weak solutions; ii) Lemma \ref{lem:product_sobolev_norm}, an analytical result on products in fractional spaces which will come handy to control nonlinear terms; iii) Lemma \ref{lem:general_balance}, an abstract result on the evolution of negative Sobolev norms for Kraichnan type SPDEs. 

Let us shortly motivate the upcoming Lemma \ref{lem:YamWat}. Especially in the ``critical'' case $\alpha+\beta/2=1-1/p$, in order to establish uniqueness, it would be useful to work not just with any weak solution $\theta^0$, but rather with those satisfying some additional properties, like the decomposition \eqref{eq:decomposition} coming from Theorem \ref{thm:weak_existence}. However, we do not know a priori that such solution $\theta^0$ exists on every probability space, so we cannot compare directly any other weak solution with $\theta^0$. We can overcome this issue by enlarging the underlying probability space by a coupling procedure, in the style of a Yamada-Watanabe argument.

\begin{lemma}\label{lem:YamWat}
    Let $p\geq 2$ and let $\theta^1$, $\theta^2$ be two weak $L^1\cap L^p$ solutions to \eqref{eq:stochastic_gSQG} on the same tuple $(\Omega,\mathcal{A},\mathcal{F}_t,\P,(W^k)_k)$, in the sense of Definition \ref{defn:weak_solution}.
    Let $\theta^0= \theta^{0,<}+\theta^{0,>}$ be a solution on another tuple $(\Omega^0,\mathcal{A}^0,\mathcal{F}^0_t,\P^0,(W^{0,k})_k)$, satisfying the bound \eqref{eq:apriori_bounds} and the decomposition property in \eqref{eq:decomposition} for some $\varepsilon>0$, whose existence is guaranteed by Theorem \ref{thm:weak_existence}.
    
    Then there exists a filtered probability space $(\tilde{\Omega},\tilde{\mathcal{A}},\tilde{\mathcal{F}}_t,\tilde{\P})$, a cylindrical $\tilde{\mathcal{F}}_t$-Brownian motion $(\tilde{W}^k)_k$ and $\tilde{\mathcal{F}}_t$-progressively measurable processes $\tilde{\theta}^{0,<}$, $\tilde{\theta}^{0,>}$, $\tilde{\theta}^1$, $\tilde{\theta}^2$ such that $((\tilde{W}^k)_k,\tilde{\theta}^{0,<},\tilde{\theta}^{0,>})$, resp. $((\tilde{W}^k)_k,\tilde{\theta}^1,\tilde\theta^2)$, has the same law of $((W^{0,k})_k,\theta^{0,<},\theta^{0,>})$, resp. of $((W^k)_k,\theta^1,\theta^2)$ (on the space $C_t^{\mathbb{N}}\times C_t(H^{-2})^2$). In particular, taking $\tilde{\theta}^0=\tilde{\theta}^{0,<}+\tilde{\theta}^{0,>}$, $(\tilde{\Omega},\tilde{\mathcal{A}},\tilde{\mathcal{F}}_t,\tilde{\P},(\tilde{W}^k)_k,\tilde{\theta}^i)$ is a weak solution to \eqref{eq:stochastic_gSQG}, for $i=0,1,2$.
\end{lemma}

The argument is classical and there are many variants, see for example \cite{rockner2008yamada}. In the context of weak-strong uniqueness for stochastic PDEs, a similar use has been done, for example, in \cite[Theorem 4.3 and Theorem 4.4]{BreFeiHof2017}.
Let us point out that, by the assumption $p\geq 2$, the statement of Lemma \ref{lem:YamWat} is rigorous: by Remark \ref{rem:weak_solution_bochner}, the stochastic gSQG equation \eqref{eq:stochastic_gSQG} is equivalent to the integral equation \eqref{eq:weak_solution_bochner} on $H^{-2}$, so the solutions $\theta^i$ can be regarded as random variables with values in $C_t(H^{-2})$.

\begin{proof}
    In the proof, we will use the notation $W=(W^k)_k$, $W^0=(W^{0,k})_k$, $\tilde{W}=(\tilde{W}^k)_k$. We call $\text{Law}(\theta^{0,<},\theta^{0,>})(\cdot\mid W^0=w^0)$ a regular version of the conditional law of $(\theta^{0,<},\theta^{0,>})$ given $W^0$, and analogously $\text{Law}(\theta^1,\theta^2)(\cdot\mid W=w)$ a regular version of the conditional law of $(\theta^1,\theta^2)$ given $W$ (existence and a.e.\ uniqueness of the regular versions follow from \cite[Theorem 5.3.1]{AmGiSa2008}). We define
    \begin{align*}
        \tilde{\Omega} = C_t^{\mathbb{N}} \times C_t(H^{-2})^2\times C_t(H^{-2})^2,
    \end{align*}
    and we take $\tilde{W}$ and $\tilde{\theta}^{0,<},\tilde{\theta}^{0,>},\tilde{\theta}^1,\tilde{\theta}^2$ as the canonical projections on $C_t^{\mathbb{N}}$ and on $C_t(H^{-2})^2\times C_t(H^{-2})^2$, respectively. We take $\mathcal{A}$ as the Borel $\sigma$-algebra on $\tilde{\Omega}$, enlarged with the $\tilde{\P}$-null sets, and
    \begin{align*}
        &\tilde{\P}(\dd(\tilde{w},\tilde{\omega}^{0,<},\tilde{\omega}^{0,>},\tilde{\omega}^1,\tilde{\omega}^2))\\
        &\ \ := \P^W(\dd\tilde{w}) \otimes \text{Law}(\theta^{0,<},\theta^{0,>})(\dd\tilde{\omega}^{0,<},\dd\tilde{\omega}^{0,>}\mid W^0=\tilde{w}) \otimes \text{Law}(\theta^1,\theta^2)(\dd\tilde{\omega}^1,\dd\tilde{\omega}^2\mid W=\tilde{w}),
    \end{align*}
    where $\P^W$ is the cylindrical Wiener measure. Finally we take $\tilde{\mathcal{F}}^-_t$ as the filtration generated by $\tilde W$, $\tilde \theta^{0,<}$, $\tilde \theta^{0,>}$, $\tilde \theta^1$, $\tilde \theta^2$ and by the $\tilde{\P}$-null sets and we set $\tilde{\mathcal{F}}_t = \cap_{s<t}\tilde{\mathcal{F}}^-_s$, for each $t\ge 0$. One can show as in the proof of \cite[Proposition 2.5, point (1)]{Bas2011} that $(\tilde{\Omega},\tilde{\mathcal{A}},(\tilde{\mathcal{F}}_t)_t,\tilde{\P})$ satisfies the standard assumption and that $\tilde{W}$ is a cylindrical $(\tilde{\mathcal{F}}_t)_t$-Brownian motion. The fact that $(\tilde{W},\tilde{\theta}^{0,<},\tilde{\theta}^{0,>})$, resp.\ $(\tilde{W},\tilde{\theta}^1,\tilde{\theta}^2)$, has the same law of $(W^0,\theta^{0,<},\theta^{0,>})$, resp.\ of $(W,\theta^1,\theta^2)$, follows by construction. In particular, since $(W,\theta^1)$, $(W,\theta^2)$ and $(W^0,\theta^0)$ are weak $L^1\cap L^p$ solutions, also $(\tilde{W},\tilde{\theta}^i)$, $i=0,1,2$, are weak $L^1\cap L^p$ solutions to \eqref{eq:stochastic_gSQG}.
\end{proof}

The next analytic lemma plays a key role in controlling the nonlinear terms.

\begin{lemma} \label{lem:product_sobolev_norm}
    Let $\alpha\in (0,1)$, $\beta\in (0,1)$ and $p\in (1,\infty)$ satisfy
    \begin{equation}\label{eq:ass_product_sobolev_norms}
        \frac{\beta}{2}< \alpha<\frac{1}{2}, \quad \alpha+\frac{\beta}{2} = 1 -\frac{1}{p}.
    \end{equation}
    Then for any $f\in \dot L^{1-\beta,p}$ and $g\in \dot H^{1-\alpha-\beta/2}$ it holds $f\, g\in \dot H^{\alpha-\beta/2}$ with    \begin{equation}\label{eq:product_fractional_inequality}
	    \| f\, g\|_{\dot H^{\alpha-\beta/2}} \lesssim \| f\|_{\dot L^{1-\beta,p}} \| g\|_{\dot H^{1-\alpha-\beta/2}}.
\end{equation}
\end{lemma}

\begin{proof}
The statement follows from Lemma \ref{lem:gen_product_sobolev_norm} in Appendix \ref{app:analysis_lemmas}, for the choice of parameters
\begin{align*}
    r_1=p, \quad r_2=2,\quad
    s_1=1-\beta, \quad s_2=1-\alpha-\frac{\beta}{2},\quad s=\alpha-\frac{\beta}{2}.
\end{align*}
Indeed, under condition \eqref{eq:ass_product_sobolev_norms}, one can readily check that $s<s_1\wedge s_2<2$ and that
\begin{align*}
    s_1 r_1=(1-\beta)p<2, \quad s_2 r_2=\Big(1-\alpha-\frac{\beta}{2}\Big)2<2, \quad \frac{1}{p}=\frac{1}{r_1}+\frac{1}{r_2}-\frac{1}{2}=\frac{s_1+s_2-s}{2}. \qquad \qedhere
\end{align*}
\end{proof}

We conclude this section with a lemma on the evolution of negative Sobolev norms for It\^o processes driven by Kraichnan noise. We present it in a slightly more abstract fashion, since the argument is quite general in nature and similar arguments could be useful in other applications.

In the following, for fixed $\beta\in (0,1)$, we denote by $G$ the Green kernel of the fractional Laplacian $(-\Delta)^{1-\beta/2}$, i.e. $G=|\nabla|^{\beta-2}= c_\beta |x|^{-\beta}$, which acts as a Fourier multiplier by $G\ast \varphi = \cF^{-1}(|k|^{\beta-2} \cF \varphi)$ (due to our convention \eqref{eq:properties_fourier_transform}, this means that $\hat G(n)=(2\pi)^{-1} |n|^{\beta-2}$). Given a family of mollifiers $\{\chi^\delta\}_{\delta>0}$ associated to a Schwarz probability density $\chi$, we set
\begin{equation}\label{eq:defn_Gdelta}
    G^\delta:=(\chi^\delta\ast\chi^\delta)\ast G.
\end{equation}

\begin{lemma}\label{lem:general_balance}
    Let $p\geq 2$, $\xi$ be a weakly continuous process in $L^\infty_{\omega,t}(L^1\cap L^p)$, with trajectories $\P$-a.s. in $C_t H^{-2}$ satisfying the $H^{-2}$-valued integral equation
    \begin{equation}\label{eq:ito_process}
        \xi_t = \xi_0 + \int_0^t [h_s + c_0 \Delta \xi_s] \dd s - \int_0^t \nabla \xi_s\cdot \dd W_s,
    \end{equation}
    for some stochastic process $h$ such that $h\in L^1_{\omega,t} \dot H^s$ for some $s\leq \beta/2-1$ and such that $|\nabla|^{\beta-2} h\in L^1_{\omega,t} (L^1\cap L^p)'$. Let $\{\chi^\delta\}_{\delta>0}$ be standard mollifiers, $G^\delta$ as defined above. Then it holds
    \begin{equation}\label{eq:general_balance}\begin{split}
        \E[\| 
        &\xi_t\|_{\dot H^{\beta/2-1}}^2] -\E\| \xi_0\|_{\dot H^{\beta/2-1}}^2] - 2\int_0^t \E [\langle |\nabla|^{\beta-2} h_s, \xi_s\rangle ] \dd s\\
        & = \lim_{\delta\to 0} \int_0^t \int_{\R^2\times \R^2} {\rm Tr} \big( (Q(0)-Q(x-y)) D^2 G^\delta(x-y)\big) \E[\xi_s(x)\xi_s(y)] \dd x \dd y \dd s
    \end{split}\end{equation}
    where the limit on the r.h.s.\ of \eqref{eq:general_balance} is well-defined.
\end{lemma}

\begin{proof}
    By Lemma \ref{lem:stochastic_integrals}, the divergence-free property of $W$ and our assumption on $\xi$, the stochastic integral appearing in \eqref{eq:ito_process} is a well-defined $\dot H^{-1}$-valued martingale.    
    Set $\xi^\delta=\chi^\delta\ast \xi$, similarly for $h^\delta$, $(\sigma_k\cdot\nabla \xi)^\delta$; by virtue of \eqref{eq:ito_process}, it holds
    \begin{align*}
        \xi^\delta_t = \xi_0 + \int_0^t [h^\delta_s + c_0 \Delta \xi^\delta_s] \dd s - \sum_k \int_0^t (\sigma_k\cdot\nabla \xi_s)^\delta \dd W^k_s =: \xi_0 + \int_0^t [h^\delta_s + c_0 \Delta \xi^\delta_s] \dd s + M^\delta_t;
    \end{align*}
    by the above observation, our assumptions and properties of mollifiers, $h^\delta_s + c_0 \Delta \xi^\delta_s \in L^1_{\omega,t} \dot H^{\beta/2-1}$, while $M^\delta_t$ a well-defined continuous martingale in $\dot H^{\beta/2-1}$.
    Therefore we can apply It\^o formula on the Hilbert space $\dot H^{\beta/2-1}$ to find
    \begin{align*}
        \dd \| \xi^\delta_t\|_{\dot H^{\beta/2-1}}^2
        & = 2 \langle \xi^\delta_t, \dd \xi^\delta_t\rangle_{\dot H^{\beta/2-1}} + [M^\delta]_{\dot H^{\beta/2-1}}(t)\\
        & = \Big[ 2 \langle \xi^\delta_t, |\nabla|^{\beta-2} h^\delta_t\rangle + 2 c_0 \langle \xi^\delta_t, |\nabla|^{\beta-2} \Delta \xi^\delta\rangle + \sum_k \| (\sigma_k\cdot\nabla \xi_t)^\delta\|_{\dot H^{\beta/2-1}}^2 \Big] \dd t + 2\langle \xi^\delta_t, \dd M^\delta_t \rangle_{\dot H^{\beta/2-1}}.
    \end{align*}
    As the last term is a martingale, it vanishes upon taking expectation, thus after integrating in time and rearranging the terms we find
    \begin{equation}\label{eq:intermediate_ito}\begin{split}
        \E[\| \xi^\delta_t\|_{\dot H^{\beta/2-1}}^2] & -\E\| \xi_0^\delta\|_{\dot H^{\beta/2-1}}^2] - 2\int_0^t \E [\langle |\nabla|^{\beta-2} h_s^\delta, \xi^\delta_s\rangle ] \dd s\\
        & = \E \int_0^t \Big[ 2 c_0 \langle \xi^\delta_s, |\nabla|^{\beta-2} \Delta \xi^\delta_s\rangle + \sum_k \| (\sigma_k\cdot\nabla \xi_s)^\delta\|_{\dot H^{\beta/2-1}}^2 \Big] \dd s.
    \end{split}\end{equation}
    Thanks to our assumptions on $\xi$ and $h$, by dominated convergence we can pass to the limit on the l.h.s. of \eqref{eq:intermediate_ito}, since
    \begin{align*}
        \lim_{\delta\to 0} \E[\| \xi^\delta_t\|_{\dot H^{\beta/2-1}}^2] = \E[\| \xi_t\|_{\dot H^{\beta/2-1}}^2],
        \ \ 
        \lim_{\delta\to 0} \int_0^t \E [\langle |\nabla|^{\beta-2} h_s^\delta, \xi^\delta_s\rangle ] \dd s = \int_0^t \E [\langle |\nabla|^{\beta-2} h_s, \xi_s\rangle ] \dd s.
    \end{align*}
    As the limit on the l.h.s.\ of \eqref{eq:intermediate_ito} is well-defined, the same must hold for the r.h.s.. To deduce \eqref{eq:general_balance}, it remains to analyse the r.h.s of \eqref{eq:intermediate_ito}.
    Using the properties of convolutions, integration by parts and the definition of $G$, it holds
    \begin{align*}
        2c_0 \langle \xi^\delta_t, |\nabla|^{\beta-2} \Delta \xi^\delta_t\rangle
        = 2c_0 \langle \xi_t, (\chi^\delta\ast G\ast \chi^\delta) \ast\Delta\xi_t\rangle
        = 2c_0 \langle \xi_t, (\Delta G^\delta)\ast \xi_t\rangle
        = \langle \xi_t, [{\rm Tr} (Q(0)D ^2 G^\delta)]\ast \xi_t\rangle
    \end{align*}
    where in the last passage we used the fact that $Q(0)=2c_0 I$.
    Similarly, it holds
    \begin{align*}
        \sum_k \| (\sigma_k\cdot\nabla \xi_t)^\delta\|_{\dot H^{\beta/2-1}}^2
        & = \sum_k \langle G\ast \nabla\cdot (\sigma_k \xi_t)^\delta, \nabla\cdot (\sigma_k \xi_t)^\delta\rangle
        = - \sum_k \langle D^2G^\delta\ast (\sigma_k \xi_t), \sigma_k \xi_t\rangle\\
        & = - \sum_k \int_{\R^2\times \R^2} {\rm Tr} \big(D^2 G^\delta(x-y) \sigma_k(x)\otimes \sigma_k(y)\big) \xi_t(x)\xi_t(y) \dd x \dd y\\
        & = - \int_{\R^2\times \R^2} {\rm Tr} \big(D^2 G^\delta(x-y) Q(x-y)\big) \xi_t(x)\xi_t(y) \dd x \dd y
    \end{align*}
    where in the last passage we used \eqref{eq:covariance_series_representation}. Combining these identities with \eqref{eq:intermediate_ito}, we get the conclusion.
\end{proof}

\subsection{Uniqueness}\label{subsec:proof_uniqueness}

In view of Lemma \ref{lem:general_balance}, for $\chi^\delta$ and $G^\delta$ as in \eqref{eq:defn_Gdelta}, it is convenient to define the kernels
\begin{equation}\label{eq:defn_kappa_delta}
    \kappa^\delta(z):= {\rm Tr} \big( (Q(0)-Q(z)) D^2 G^\delta(z)\big) 
\end{equation}
so that we may rewrite the term inside the integral on the r.h.s.\ of \eqref{eq:general_balance} as $\E[\langle \kappa^\delta\ast\xi,\xi\rangle]$.
We can now apply Lemma \ref{lem:general_balance} to the specific case where $\xi$ is the difference of two solutions to gSQG. To this end, for fixed $\alpha,\beta$ satisfying \eqref{eq:parameters_uniqueness}, it is convenient to denote by $\fkp\in (1,\infty)$ the associated  ``critical'' Lebesgue exponent defined by the relation
\begin{equation}\label{eq:defn_fkp}
    \frac{1}{\fkp} = 1-\alpha-\frac{\beta}{2}.
\end{equation}
In this way, we can rewrite the condition on $p$ coming from \eqref{eq:parameters_uniqueness} compactly as $p\geq \fkp\vee 2$.

\begin{lemma}\label{lem:energy_balance}
    Let $p\geq \fkp\vee 2$, $\theta^i$ be two $L^1\cap L^p$ solutions to \eqref{eq:stochastic_gSQG}, defined on the same tuple $(\Omega,\mathcal{A},\mathcal{F}_t,\P,(W^k)_{k\in\N})$, such that $\theta^i\in L^\infty_{\omega,t} (L^1\cap L^p)$; set $\xi=\theta^1-\theta^2$.
    Then it holds    \begin{equation}\label{eq:balance_difference_solutions}\begin{split}
        \E[\| \xi_t\|_{\dot H^{\beta/2-1}}^2] & - \E[\| \xi_0\|_{\dot H^{\beta/2-1}}^2] - 2 \int_0^t \E\Big[ \langle (K_\beta\ast \theta^1_s) \cdot\nabla |\nabla|^{\beta-2} \xi_s, \xi_s\rangle + \langle |\nabla|^{\beta-2} (f^1_s-f^2_s),\xi_s\rangle \Big] \dd s\\
        & = \lim_{\delta\to 0} \int_0^t \E[\langle \kappa^\delta\ast \xi_s, \xi_s\rangle] \dd s.
    \end{split}\end{equation}
\end{lemma}

\begin{proof}
    Given two solutions $\theta^i$, their difference $\xi$ satisfies
    \begin{align*}
        \dd \xi_t = [-\mathcal{N}(\theta^1_t)+\mathcal{N}(\theta^2_t) + f^1_t-f^2_t] \dd t + c_0 \Delta \xi_t \dd t - \nabla\xi_t\cdot \dd W_t
    \end{align*}
    We aim to verify the conditions of Lemma \ref{lem:general_balance} for $h:= \mathcal{N}(\theta^1_t)-\mathcal{N}(\theta^2_t) + f^1_t-f^2_t$.
    For the deterministic forcing terms, we have $f^i\in L^1_t (L^1\cap L^2)\hookrightarrow L^1_t \dot H^{-s}$ for any $s\in (0,1)$ by Lemma \ref{lem:embedding_in_neg_Sobolev}; moreover $|\nabla|^{\beta-2} f^i\in L^1_t \dot L^{2-\beta,p}$ for any $p\in (1,2]$ and so by Sobolev embeddings $|\nabla|^{\beta-2} f^i\in L^1_t L^r$ for any $r\in (2/\beta,\infty)$, so that in particular $|\nabla|^{\beta-2} f^i\in L^1_t (L^1\cap L^2)'$.
    Next we verify the regularity of $\mathcal{N}(\theta^1)$, $\theta^2$ being identical. By Lemma \ref{lem:properties_nonlinearity}, $\mathcal{N}(\theta^1)\in L^\infty_{\omega,t} \dot H^{-1-\beta}$; by its definition and integration by parts, for any $\varphi\in C^\infty_c$ it holds
    \begin{align*}
        |\langle |\nabla|^{\beta-2}\mathcal{N}(\theta^1_t), \varphi\rangle|
        = |\langle (K_\beta\ast \theta^1_t)\cdot\nabla|\nabla|^{\beta-2}\varphi, \theta^1_t \rangle|
        \leq \| (K_\beta\ast \theta^1_t)\cdot\nabla|\nabla|^{\beta-2}\varphi \|_{\dot H^{\alpha-\beta/2}} \| \theta^1_t\|_{\dot H^{\beta/2-\alpha}}
    \end{align*}
    By Lemma \ref{lem:embedding_in_neg_Sobolev}, $\| \theta^1_t\|_{\dot H^{\beta/2-\alpha}}\in L^\infty_{\omega,t}$. For the first term, we can invoke Lemma \ref{lem:product_sobolev_norm} to find
    \begin{equation*}
        \| (K_\beta\ast \theta^1_t)\cdot\nabla|\nabla|^{\beta-2}\varphi \|_{\dot H^{\alpha-\beta/2}}
        \lesssim \| \nabla|\nabla|^{\beta-2}\varphi \|_{\dot L^{1-\beta,\fkp}} \| K_\beta\ast \theta^1_t\|_{\dot H^{1-\alpha-\beta/2}}
        \lesssim \| \varphi\|_{L^{\fkp}} \| \theta^1\|_{\dot H^{\beta/2-\alpha}}.
    \end{equation*}
    Overall, this proves that
    \begin{equation}\label{eq:application_sobolev_product}
        |\langle |\nabla|^{\beta-2}[(K_\beta\ast\theta^1_t)\cdot\nabla \theta^1], \varphi\rangle| \lesssim \| \varphi\|_{L^{\fkp}} \| \theta^1\|_{\dot H^{\beta/2-\alpha}}^2
    \end{equation}
    which in particular shows that $\mathcal{N}(\theta^1)\in L^\infty_{\omega,t} L^{\fkp'}$. 

    Therefore we are in the position to apply Lemma \ref{lem:general_balance} to deduce that
    \begin{align*}
        \E[\| \xi_t\|_{\dot H^{\beta/2-1}}^2] - \E[\| \xi_0\|_{\dot H^{\beta/2-1}}^2] - 2\int_0^t \E[|\nabla|^{\beta-2} h_s,\xi_s] \dd s
        = \lim_{\delta\to 0} \int_0^t \E[\langle \kappa^\delta\ast \xi_s, \xi_s\rangle] \dd s.
    \end{align*}
    In order to obtain \eqref{eq:balance_difference_solutions} from the definition of $h$, it remains exploit the cancellation properties of the nonlinearity already mentioned in \eqref{eq:cancellation_nonlinearity}: it holds
    \begin{align*}
        \langle \mathcal{N}(\theta^1)-\mathcal{N}(\theta^2), \xi \rangle_{\dot H^{\beta/2-1}}
        & = \langle (K_\beta\ast \theta^1)\cdot\nabla \xi, |\nabla|^{\beta-2} \xi\rangle + \langle (K_\beta\ast\xi)\cdot\nabla \theta^2, |\nabla|^{\beta-2} \xi\rangle\\
        & = -\langle (K_\beta\ast \theta^1)\cdot \nabla|\nabla|^{\beta-2} \xi, \xi\rangle - \langle (K_\beta\ast \xi)\cdot \nabla|\nabla|^{\beta-2} \xi, \theta^2\rangle
    \end{align*}
    where the last term on the r.h.s.\ is $0$.
\end{proof}

In order to analyse the limiting behaviour of the kernels $\kappa^\delta$ appearing in \eqref{eq:balance_difference_solutions}, it is convenient to rewrite them in Fourier variables. Under our convention \eqref{eq:convention_fourier}, it holds
\begin{align*}
    \langle \kappa^\delta \ast \xi, \xi\rangle
    = 2\pi \int_{\R^2} \hat\kappa^\delta (n) |\hat\xi(n)|^2 \dd n.
\end{align*}
Let us introduce the convenient notation $A:B={\rm Tr} (A^T B)$; by the properties \eqref{eq:properties_fourier_transform} we have
\begin{align*}
    \hat\kappa^\delta(n)
    & = \cF\big( Q(0): D^2 G^\delta\big)(n)- \cF\big( Q:D^2 G^\delta\big)(n)\\
    & = Q(0) : (n\otimes n) \hat G^\delta(n) - (2\pi)^{-1} \int_{\R^2} \hat Q(k-n) : (k\otimes k) \, \hat G^\delta(k) \dd k\\
    & = (2\pi)^{-1} \int_{\R^2} \big[\hat Q(n-k) : (n\otimes n)\, \hat G^\delta(n) - \hat Q(n-k) : (k\otimes k)\, \hat G^\delta(k)\big] \dd k  
\end{align*}
where in the second passage we used the antitransform of $Q(0)$. We now use the explicit formula for $\hat Q$ coming from \eqref{eq:defn-Qalpha}; noticing that the projection matrices $P^\perp_z$ satisfy
\begin{align*}
    P^\perp_{n-k} : n\otimes n = |P^\perp_{n-k} n|^2 = |P^{\perp}_{n-k} k|^2 = P^\perp_{n-k} : k\otimes k,
\end{align*}
overall we end up finding
\begin{equation}\label{eq:fourier_coercoice_term}
      \langle \kappa^\delta \ast \xi, \xi\rangle
      = \int_{\R^2}  F^\delta(n) |\widehat{\xi}(n)|^2\dd n.
\end{equation}
for the deterministic function $F^\delta$ given by
\begin{equation}\label{eq:defn_Fdelta}
      F^\delta(n):=\int_{\R^2} \langle n- k \rangle  ^{-(2+2\alpha)} |P^\perp_{n-k} n|^2 \big(\widehat{G}^\delta(n)  -\widehat{G}^\delta(k)\big) \dd k
\end{equation}

At this point, we can pass to the limit as $\delta \to 0$ by leveraging on the following key facts, which are contained in the statement and proof of \cite[Proposition 3.2]{GGM2024} and in \cite[Proposition 4.2]{GGM2024}:

\begin{theorem}\label{thm:GGM}
    Let $\xi$ be a process in $L^\infty_{t,\omega}(L^1\cap L^2)$; let $\chi$ be a Schwarz function whose transform $\hat\chi$ satisfies $\hat\chi\equiv 1$ for $|n|\leq 1$, $\hat\chi\equiv 0$ for $|n|\geq 2$, and consider the associated mollifiers $\{\chi^\delta\}_\delta$. Then for any $t\in [0,T]$ it holds
    \begin{align*}
        \lim_{\delta \to 0} \E\int_0^t \langle \kappa^\delta\ast \xi_s,\xi_s\rangle \dd s 
         = \lim_{\delta \to 0} \E\int_0^t \int_{\R^2} F^\delta(n) |\widehat{\xi}_s(n)|^2\dd n \dd s = \int_0^t  \int_{\R^2} F(n)\, \E|\hat{\xi}_s(n)|^2 \dd n \dd s,
    \end{align*}
    for a function $F\in L^1+ L^\infty$ which is the pointwise limit of $F^\delta$ as defined in \eqref{eq:defn_Fdelta}. Moreover there exists positive constants $K=K(\alpha,\beta)$ and $C=C(\alpha,\beta)$ such that
    \begin{align*}
        F(n) \le -K|n|^{\beta-2\alpha} + C |n|^{\beta-2}\quad \forall\, n\in\R^2.
    \end{align*}
\end{theorem}

Combining identity \eqref{eq:balance_difference_solutions} with Theorem \ref{thm:GGM}, we deduce that
\begin{equation}\label{eq:intermediate_uniqueness_estimate}\begin{split}
    \frac{\dd}{\dd t} \E[\| \xi_t &\|_{\dot H^{\beta/2-1}}^2] + K \E[\| \xi_t\|_{\dot H^{\beta/2-\alpha}}^2]\\
    &\leq 2 \E[\langle (K_\beta\ast \theta^1_t) \cdot\nabla |\nabla|^{\beta-2} \xi_t, \xi_t\rangle] + 2\E[\langle |\nabla|^{\beta-2} (f^1_t-f^2_t),\xi_t\rangle] + C \E[\| \xi_t\|_{\dot H^{\beta/2-1}}^2]
\end{split}\end{equation}

With estimate \eqref{eq:intermediate_uniqueness_estimate} at hand, we can now complete the

\begin{proof}[Proof of Theorem \ref{thm:uniqueness}]
    Let $\theta^i$ be the solutions given in the statement and let $\eps>0$ be a small parameter to be determined later in the proof.
    
    Without loss of generality, we can assume that $\theta^1$ admits a decomposition $\theta^1=\theta^>+\theta^<$, where $\P$-a.s. $\sup_{t\in [0,T]} \|\theta^>_t\|_{L^\fkp}\leq \eps$ and there exist $\gamma=\gamma(p)$, $C=C(\eps,\gamma)>0$ such that $\P$-a.s., $\sup_{t\in [0,T]} \|\theta^>_t\|_{L^\fkp\cap L^{\fkp+\gamma}}\leq C$.
    Indeed if $p>\fkp$ this is obvious, since we can just take $\theta^>\equiv 0$, $\theta^<=\theta^1$ and $\gamma=p-\fkp$ (recall the definition of $\fkp$ in \eqref{eq:defn_fkp}). Instead if $\fkp = p$, we can invoke Lemma \ref{lem:YamWat} to copy $(\theta^1,\theta^2)$ on a new probability space and construct there another solution $\theta^0$ associated to the data $(\theta^1_0,f^1)$ and satisfying the desired decomposition, in view of Theorem \ref{thm:weak_existence}.
    We can then compare $\theta^0$ with $\theta^i$; indeed, once estimate \eqref{eq:stability_bd} is established for $(\theta^0,\theta^1)$ in place of $(\theta^1,\theta^2)$ it follows that $\theta^0=\theta^1$.
    We can then compare $\theta^0$ and $\theta^2$ and use $\theta^0=\theta^1$ to obtain the estimate \eqref{eq:stability_bd} for $(\theta^1,\theta^2)$.
    
    To deduce \eqref{eq:stability_bd}, we want to implement a Gr\"onwall type argument; to this end, we need to find suitable estimates for the terms appearing on the r.h.s.\ of \eqref{eq:intermediate_uniqueness_estimate}, leveraging on the presence of the ``good'' term $K \E[\| \xi_t\|_{\dot H^{\beta/2-\alpha}}^2]$ on the l.h.s..
    The term coming from the forcings $f^i$ is easy to handle, since by Cauchy
    \begin{equation}\label{eq:estim_forcing}
        \E[\langle |\nabla|^{\beta-2} (f^1_t-f^2_t),\xi_t\rangle]
        \leq \| f^1-f^2\|_{\dot H^{\beta/2-1}} \E[\| \xi_t\|_{\dot H^{\beta/2-1}}]
        \leq \| f^1-f^2\|_{\dot H^{\beta/2-1}} \E[\| \xi_t\|_{\dot H^{\beta/2-1}}^2]^{1/2}
    \end{equation}
    For the nonlinear term, recall that we may assume $\theta^1$ to satisfy the aforementioned decomposition $\theta^1=\theta^>+\theta^<$. Correspondingly, we split $I_t:=\langle (K_\beta\ast \theta^1_t) \cdot\nabla |\nabla|^{\beta-2} \xi_t, \xi_t\rangle$ as
    \begin{align*}
        I_t=I^{>}_t+I^{<}_t:= \langle (K_\beta\ast \theta^>_t) \cdot\nabla |\nabla|^{\beta-2} \xi_t, \xi_t\rangle + \langle (K_\beta\ast \theta^<_t) \cdot\nabla |\nabla|^{\beta-2} \xi_t, \xi_t\rangle;
    \end{align*}
    we treat the two terms separately.
    For the ``critical term'', arguing  similarly to the computations yielding \eqref{eq:application_sobolev_product}, we find
    \begin{equation}\label{eq:estim_>}
        |I^>_t| = |\langle (K_\beta\ast \theta^>_t) \cdot\nabla |\nabla|^{\beta-2} \xi_t, \xi_t\rangle|
        \leq C_1 \| \theta^>_t\|_{L^{\fkp}} \| \xi_t\|_{\dot H^{\beta/2-\alpha}}^2
        \leq C_1 \eps \| \xi_t \|_{\dot H^{\beta/2-\alpha}}^2
    \end{equation}    
    for a constant $C_1=C_1(\alpha,\beta)$ coming from the application of Lemma \ref{lem:product_sobolev_norm} and Sobolev embeddings. At this point, we choose and fix the parameter $\eps>0$ to be small enough such that $C_1 \eps \leq K/4$.

    For the ``subcritical'' term $I^<_t$, we need to set up the computation slightly differently, so not to produce any large constant. Recalling that there exists $\gamma=\gamma(p)>0$ such that $\theta^>$ belongs to $L^q$ for all $q\in [\fkp,\fkp+\gamma]$, we now choose $\tilde\alpha\in (\alpha,1/2)$ small enough such that
    \begin{align*}
        \tilde p:=\Big( 1-\tilde \alpha-\frac{\beta}{2}\Big)^{-1}
    \end{align*}
    belongs to the interval $(\fkp, \fkp+\gamma]$. We can then apply duality estimates and Lemma \ref{lem:product_sobolev_norm}, with $(\alpha,p)$ replaced by $(\tilde\alpha,\tilde p)$, to find
    \begin{align*}
        |I_1^{<}|
        & =|\langle (K_\beta\ast \theta^<_t) \cdot\nabla |\nabla|^{\beta-2} \xi_t, \xi_t \rangle|\\
        & \le \| (K_\beta\ast \theta^<_t) \cdot\nabla |\nabla|^{\beta-2} \xi_t\|_{\dot H^{\tilde\alpha-\beta/2}} \|\xi \|_{\dot H^{\beta/2-\tilde\alpha}}\\
        & \lesssim \| K_\beta\ast \theta^<_t\|_{\dot L^{1-\beta,\tilde p}} \|  \nabla |\nabla|^{\beta-2} \xi_t\|_{\dot H^{1-\tilde\alpha-\beta/2}} \|\xi \|_{\dot H^{\beta/2-\tilde\alpha}}\\
        &\lesssim \| \theta^<_t\|_{L^{\tilde p}} \|\xi \|^2_{\dot H^{\beta/2-\tilde\alpha}} 
    \end{align*}
    where the hidden constant now depends on $\alpha, \beta$ and $p$. Note that, by the structure of our decomposition and possibly Theorem \ref{thm:weak_existence}, $\P$-a.s. we have
    \begin{align*}
        \sup_{t\in [0,T]} \| \theta^<_t \|_{L^{\tilde{p}}} \le C_2,
    \end{align*}
    for some constant $C_2=C_2(\alpha,\beta,p, \tilde{p},\theta_0^1,f^1)$.
    On the other hand, since by construction $\beta/2-\tilde\alpha\in (\beta/2-1,\beta/2-\alpha)$, by interpolation (\cite[Proposition 1.32]{Bahouri2011}) and Young inequalities, there exists $\lambda\in (0,1)$ such that 
    \begin{align*}
        \|\xi\|^2_{\dot H^{\beta/2-\tilde\alpha}}
        \le \|\xi\|_{\dot H^{-1+\beta/2}}^{2\lambda} \|\xi\|_{\dot H^{-\alpha+\beta/2}}^{2(1-\lambda)} 
        \le \eta\|\xi\|^2_{\dot H^{-\alpha+\beta/2}}  + C_3 \|\xi\|^2_{\dot H^{-1+\beta/2}}
    \end{align*}
    for any $\eta>0$ and a suitable constant $C_3=C_3(\alpha,\beta,\eta)>0$. Overall, upon choosing $\eta>0$ small enough, we conclude that
    \begin{equation}\label{eq:estim_<}
        |I^<_t| \leq \frac{K}{4} \|\xi\|^2_{\dot H^{-\alpha+\beta/2}} + C_4 \|\xi\|^2_{\dot H^{-1+\beta/2}}
    \end{equation}   
    for another constant $C_4$ depending on $C_2$ and $C_3$.
    Inserting estimates \eqref{eq:estim_forcing}, \eqref{eq:estim_>} and \eqref{eq:estim_<} in \eqref{eq:intermediate_uniqueness_estimate} we arrive at
    \begin{align*}
	\E \|\xi_t\|^2_{\dot H^{-1+\beta/2}} & - \|\xi_0\|^2_{\dot H^{-1+\beta/2}}
	+ \frac{K}{2}  \E\int_0^t \|\xi_s\|^2_{\dot{H}^{-\alpha+\beta/2}} \dd s\\
    &\le C_5\, \E\int_0^t \|\xi_s\|^2_{\dot H^{-1+\beta/2}} \dd s + 2\int_0^t \E[\|\xi_s\|^2_{\dot{H}^{-1+\beta/2}}]^{1/2} \|f^1_s-f^2_s\|_{\dot{H}^{-1+\beta/2}} \dd s
    \end{align*}
    for $C_5=C_4+C$. Calling
    \begin{align*}
    h(t):=\E \|\xi_t\|^2_{\dot H^{-1+\beta/2}}
	   + \frac{K}{2}  \E\int_0^t \|\xi_s\|^2_{\dot{H}^{-\alpha+\beta/2}} \dd s,
    \end{align*}
    we have
    \begin{align*}
        h(t)\le h(0) +C_5 \int_0^t h(s) \dd s +2\int_0^t \|f^1-f^2\|_{\dot{H}^{-1+\beta/2}} h(s)^{1/2} \dd s.
    \end{align*}
    Applying Gr\"onwall inequality, we get
    \begin{align*}
        h(t)\le \e^{C_5 t} \Big( h(0) + 2 \int_0^t \|f^1_s-f^2_s\|_{\dot{H}^{-1+\beta/2}} h(s)^{1/2} \dd s\Big).
    \end{align*}
    Applying now Bihari inequality, we arrive at
    \begin{align*}
        h(t)^{1/2} \le  \e^{C_5 t/2} h(0)^{1/2} + \e^{C_5 t}\int_0^t \|f^1_s-f^2_s\|_{\dot{H}^{-1+\beta/2}} \dd s,
    \end{align*}
    which implies \eqref{eq:stability_bd}.
    Notice that this implies in particular pathwise uniqueness and, by Yamada-Watanabe, uniqueness in law.
    
    Finally, note that in the subcritical case our decomposition is just $\theta^>\equiv 0$, $\theta^<=\theta^1$ and the constant $C_2$ (thus also $C_5$) can be controlled more directly by $\| \theta^1_0\|_{L^1\cap L^p}$, $\| f^1\|_{L^1_t(L^1\cap L^p)}$ by virtue of \eqref{eq:apriori_bounds}.
\end{proof}

\subsection{Vanishing viscosity limit and smooth approximations}\label{subsec:vanishing_viscosity}

Having established well-posedness of solutions, we can now pass to understand whether suitable approximations schemes will converge to our solutions. We will do it for vanishing viscosity and smoothing approximations.

We start with a lemma establishing strong existence and uniqueness of solutions for the dissipative stochastic gSQG equations. The notion of solution we consider is analogous to Definition \ref{defn:weak_solution}, thus for simplicity we omit the full details.

\begin{lemma}\label{lem:WP_dissipative_gSQG}
    Let $p\in [2,\infty]$, $\alpha\in (0,1)$, $\beta\in (0,1)$ and $\nu>0$;
    let $(\Omega,\mathcal{A},\mathcal{F}_t,\P)$ be a filtered probability space carrying a family $(W^k)_{k\ge1}$ of independent Brownian motions.
    Then there exists a probabilistically strong solution to the viscous stochastic gSQG equation
    \begin{equation}\label{eq:dissipative_gSQG}
         \dd \theta_t^\nu + (K_\beta\ast\theta^\nu_t)\cdot\nabla \theta^\nu_t \dd t + \circ \dd W_t\cdot\nabla \theta^\nu_t = \big[\nu \Delta \theta^\nu + f_t\big] \dd t, \quad \theta\vert_{t=0}=\theta_0,
    \end{equation}
    which satisfies the $\P$-a.s. bounds
    \begin{align}
        & \| \theta^\nu_t \|_{L^q} \le \| \theta_0 \|_{L^q} + \int_0^t \| f_s\|_{L^q} \dd s \quad \forall\,t\geq 0,\, q\in [1,p]\label{eq:apriori_viscous1},\\
        & \nu \int_0^t \| \nabla\theta^\nu_s\|_{L^2}^2 \dd s
        \leq \frac{1}{2} \Big( \|\theta_0\|_{L^2} + \int_0^t \| f_s\|_{L^2} \dd s\Big)^2\quad \forall\,t\geq 0.\label{eq:apriori_viscous2}
    \end{align}
    Moreover, pathwise uniqueness holds in the class of solutions to \eqref{eq:dissipative_gSQG} satifying \eqref{eq:apriori_viscous1}-\eqref{eq:apriori_viscous2}.    
\end{lemma}

\begin{proof}
    The existence of weak solutions can be established by compactness arguments, exactly as in Section \ref{sec:existence}; for this reason, let us only explain a bit loosely how to achieve the bounds \eqref{eq:apriori_viscous1}-\eqref{eq:apriori_viscous2}, manipulating everything as if it were smooth.
    Estimate \eqref{eq:apriori_viscous1} can be derived similarly to \eqref{eq:apriori_pathwise_bounds}, as the additional presence of $\nu \Delta\theta^\nu$ can only help dissipating $L^p$-norms faster; for an alternative argument, based on computing the evolution of $\dd \int_{\R^2} g(\theta^\nu_t(x)) \dd x$ for convex functions $g$ (like $a\mapsto |a|^p$), we refer to \cite[Proposition 3.1]{GalLuo2023}.
    Estimate \eqref{eq:apriori_viscous2} follows by energy estimates: testing the SPDE against $\theta^\nu$ itself, using the fact that the noise is divergence free and in Stratonovich form, one would formally find the $\P$-a.s. identity
    \begin{align*}
        \frac{\dd}{\dd t} \| \theta^\nu_t\|_{L^2}^2 + 2 \nu \| \nabla \theta^\nu_t\|_{L^2} = 2\langle \theta^\nu_t, f_t\rangle.
    \end{align*}
    Defining $h(t):= \| \theta^\nu_t\|_{L^2}^2 + 2 \nu \int_0^t \| \nabla \theta^\nu_s\|_{L^2} \dd s$, it then holds
    \begin{align*}
        \frac{\dd}{\dd t} h(t)^{1/2}
        = \frac{1}{2 h(t)^{1/2}} \frac{\dd}{\dd t} h(t)
        = \frac{1}{h(t)^{1/2}} \langle \theta^\nu_t, f_t\rangle
        \leq \frac{\| \theta^\nu_t\|_{L^2}}{h(t)^{1/2}}\, \| f_t\|_{L^2}
        \leq \| f_t\|_{L^2};
    \end{align*}
    integrating in time, using the definition of $h(t)$ and the fact that $h(0)=\| \theta_0\|_{L^2}$, after some rearrangements one arrives at \eqref{eq:apriori_viscous2}.

    Assume now that we are given two solutions $\theta^\nu$, $\tilde\theta^\nu$ solving \eqref{eq:dissipative_gSQG}, then the difference $\xi=\theta^\nu-\tilde\theta^\nu$ satisfies
    \begin{equation}\label{eq:toward_monotonicity}
        \dd \xi + [\mathcal{N}(\theta^\nu)-\mathcal{N}(\tilde\theta^\nu)] \dd t = \nu \Delta\xi \dd t - \circ \dd W\cdot\nabla \xi = (\nu + c_0) \Delta\xi \dd t - \dd W\cdot\nabla \xi
    \end{equation}
    where we used $\mathcal{N}$ to denote the nonlinearity as in Section \ref{subsec:weak_sol}.
    Noticing that $\nabla\xi\in L^\infty_\omega L^2_t L^2_x$, so that the It\^o part of the stochastic integral appearing in \eqref{eq:toward_monotonicity} is a well-defined $L^2$-valued martingale, while the term $(\nu+c_0)\Delta\xi$ belongs to $L^2_t H^{-1}_x$, we can apply \cite[Theorem 2.13]{RozLot2018} and exploit the cancellations coming from the It\^o--Stratonovich corrector to deduce that
    $\P$-a.s. it holds
    \begin{align*}
        \frac{\dd}{\dd t} \frac{\| \xi_t\|_{L^2}^2}{2} + \nu\| \nabla \xi_t\|_{L^2}^2 = - \langle \mathcal{N}(\theta^\nu_t)-\mathcal{N}(\tilde\theta^\nu_t), \xi_t \rangle.
    \end{align*}
    From here, by classical arguments for monotone (S)PDEs (cf. \cite{FGL2021blowup,RoShZh2024} and the references therein) allow to reduce the problem of uniqueness to devising good estimates for the nonlinear function $\mathcal{N}$.
    In our case, since $\beta\in (0,1)$, using the divergence-free property of $K_\beta$ it holds
    \begin{align*}
	|\langle \mathcal{N}(\theta^\nu_t)-\mathcal{N}(\tilde\theta^\nu_t), \xi_t\rangle|
	& = |\langle (K_\beta\ast \xi_t)\cdot\nabla \theta^\nu_t, \xi_t\rangle|
	\leq \| K_\beta\ast\xi_t\|_{L^\infty} \| \nabla\theta^\nu_t\|_{L^2} \| \xi_t\|_{L^2}\\
	& \lesssim \| K_\beta\ast\xi_t\|_{\dot H^{2-\beta}} \| \nabla\theta^\nu_t\|_{L^2} \| \xi_t\|_{L^2}
	\lesssim \| \xi_t\|_{H^1} \| \nabla\theta^\nu\|_{L^2}  \| \xi_t\|_{L^2}.
\end{align*}
    By Young's inequality one can then arrive at
   \begin{align*}
        \frac{\dd}{\dd t} \frac{\| \xi_t\|_{L^2}^2}{2} + \frac{\nu}{2}\| \nabla \xi_t\|_{L^2}^2 \lesssim  (1 + \| \nabla\theta^\nu\|_{L^2}^2) \| \xi_t\|_{L^2}^2
    \end{align*}
    and thus deduce that $\xi\equiv 0$ by Gr\"onwall lemma, since $1+\|\nabla\theta^\nu\|_{L^2}^2$ is locally integrable.
    Having established weak existence and pathwise uniqueness, strong existence then follows by Yamada--Watanabe.
\end{proof}

Armed with Lemma \ref{lem:WP_dissipative_gSQG}, we can now establish convergence of the viscous solutions $\theta^\nu$ to $\theta$, with quantitative rates as $\nu\to 0$.

\begin{proposition}\label{prop:vanishing_viscosity}
    Let $\alpha$, $\beta$, $p$ be as in Theorem \ref{thm:main}, $\theta_0\in L^1\cap L^p$, $f\in L^1_\loc([0,+\infty); L^1\cap L^p)$.
    Let $(\Omega,\mathcal{A},\mathcal{F}_t,\P)$ be a filtered probability space carrying a family $(W^k)_{k\ge1}$ of independent Brownian motions.
    For any $\nu>0$ let $\theta^\nu$ the unique strong solution to \eqref{eq:dissipative_gSQG} associated to $(\theta_0,f)$, similarly $\theta$ be the unique strong solution to \eqref{eq:stochastic_gSQG}. 
    Then there exists a constant $C>0$, depending on $\alpha$, $\beta$, $p$, $\theta_0$ and $f$ such that
    \begin{equation}\label{eq:vanishing_viscosity_estimate}
        \sup_{t\in [0,T]} \E[ \| \theta^\nu_t-\theta_t\|_{\dot H^{\beta/2-1}}^2] \leq C e^{TC}  \nu^{2-\beta/2-\alpha} \quad \forall\, \nu>0,\ T>0.
    \end{equation}
\end{proposition}

\begin{remark}
    Under condition \eqref{eq:parameters_uniqueness}, $\beta/2+\alpha<1$, thus the exponent $2-\beta/2-\alpha$ in the viscosity parameter $\nu$ satisfies
    \begin{align*}
        2- \frac{\beta}{2}-\alpha\geq 1+\frac{1}{p}>1;
    \end{align*}
    in particular, the r.h.s. of \eqref{eq:vanishing_viscosity_estimate} vanishes as $\nu\to 0$. Furthermore, the smaller $\alpha$ and $\beta$ are, the faster the rate of convergence.
\end{remark}

\begin{proof}
    Define $\xi^\nu=\theta^\nu-\theta$, so that it solves the SPDE in It\^o form
    \begin{align*}
        \dd \xi^\nu + \big[\mathcal{N}(\theta^\nu)-\mathcal{N}(\theta)\big] \dd t + \nabla\xi^\nu u\cdot \dd W = \big[c_0 \Delta\xi^\nu + \nu\Delta \theta^\nu\big] \dd t
    \end{align*}
    We can now proceed exactly in the same way as in the proof of Theorem \ref{thm:uniqueness}, up to replacing the term $f^1-f^2$ with $\nu \Delta \theta^\nu$. The resulting term $I_f^\delta$, which previously was an approximation of $2\langle \xi, f^1-f^2\rangle_{\dot H^{\beta/2-1}}$, will now converge to
    \begin{align*}
        2\nu \int_0^t \langle \xi^\nu_s, \nu \Delta \theta^\nu_s\rangle_{\dot H^{\beta/2-1}} \dd s =
        2\nu \int_0^t \langle |\nabla|^{\beta-2} \xi^\nu_s, \nu \Delta\theta^\nu_s\rangle \dd s.
    \end{align*}
    Performing the same estimates as in the proof of Theorem \ref{thm:uniqueness}, based on the validity of estimates \eqref{eq:apriori_viscous1} for $\theta^\nu$ (which takes the role of $\theta^2$) and the decomposition $\theta=\theta^>+\theta^<$ (which takes the role of $\theta^1$), we therefore end up finding constants $K_1$, $K_2>0$ such that
    \begin{equation}\label{eq:balance_viscosity}
        \E \|\xi^\nu_t\|^2_{\dot H^{-1+\beta/2}}
	+ K_1  \E\int_0^t \|\xi^\nu_s\|^2_{\dot H^{-\alpha+\beta/2}} \dd s
        \le K_2  \E\int_0^t \|\xi^\nu_s\|^2_{\dot H^{-1+\beta/2}} \dd s + 2\nu \int_0^t \E\big[ \langle |\nabla|^{\beta-2} \xi^\nu_s, \Delta\theta^\nu_s\rangle \big]\dd s.
    \end{equation}
    We now want to estimate the last term on the r.h.s. of \eqref{eq:balance_viscosity} in a way that it can be controlled by the a priori estimate \eqref{eq:apriori_viscous2} and the coercive term $\| \xi^\nu\|_{\dot H^{-\alpha+\beta/2}}$ appearing on the r.h.s.
    To this end, we perform the estimate
    \begin{align*}
	\langle |\nabla|^{\beta-2} (\theta^\nu_t-\theta_t), \Delta\theta^\nu_t\rangle
	& = \langle |\nabla|^{\beta/2-\alpha} (\theta^\nu_t-\theta_t), |\nabla|^{\beta/2+\alpha-2}\Delta\theta^\nu_t\rangle\\
	& \leq \| \theta^\nu_t-\theta_t\|_{\dot H^{\beta/2-\alpha}} \| |\nabla|^{\beta/2+\alpha-2}\Delta\theta^\nu_t\|_{L^2}\\
	& \leq \| \theta^\nu_t-\theta_t\|_{\dot H^{\beta/2-\alpha}} \| \theta^\nu_t\|_{\dot H^{\beta/2+\alpha}}
    \end{align*}
    Plugging this into \eqref{eq:balance_viscosity} and applying Young's inequality, we find
    \begin{align*}
	\frac{\dd}{\dd t} \E[ \| \theta^\nu_t-\theta_t\|_{\dot H^{\beta/2-1}}^2] + \frac{K_1}{2} \E[ \| \theta^\nu_t -\theta_t\|_{\dot H^{\beta/2-\alpha}}^2] 
	\leq \tilde K_2\bigg( \E[ \| \theta^\nu_t-\theta_t\|_{\dot H^{\beta/2-1}}^2] + \nu^2 \int_0^t \E \| \theta^\nu_r\|_{\dot H^{\beta/2+\alpha}}^2 \dd r\bigg).
    \end{align*}
    for some new constant $\tilde K_2$, depending on $K_1$ and $K_2$.
    Notice that under our assumptions $\beta/2+\alpha<1$; thus by interpolation, for any fixed $T>0$ we have the $\P$-a.s. estimate
    \begin{align*}
	\nu^2 \int_0^T \| \theta^\nu_r\|_{\dot H^{\beta/2+\alpha}}^2 \dd r
	& \leq \nu^2 \Big(\int_0^T \| \theta^\nu_r\|_{\dot H^1}^2 \dd r\Big)^{\beta/2+\alpha} \Big(\int_0^T \| \theta^\nu_r\|_{L^2}^2 \dd r\Big)^{1-\beta/2-\alpha}\\
	& \lesssim \nu^{2-\beta/2-\alpha}\, T^{1-\beta/2-\alpha}\, \Big( \|\theta_0\|_{L^2}+\| f\|_{L^1([0,T];L^2)} \Big)^2
    \end{align*}
    where in the last passage we applied \eqref{eq:apriori_viscous1}-\eqref{eq:apriori_viscous2}.
    Therefore we are in the position to apply Gr\"onwall to conclude that
    \begin{align*}
    \sup_{t\in [0,T]} \E\big[ \| \theta^\nu_t-\theta_t\|_{\dot H^{\beta/2-1}}^2\big]
    & \leq e^{\tilde K_2 T} \nu^2 \int_0^T \E \| \theta^\nu_r\|_{\dot H^{\beta/2+\alpha}}^2 \dd r\\
    & \leq e^{\tilde K_2 T} T^{1-\beta/2-\alpha}\, \Big( \|\theta_0\|_{L^2}+\| f\|_{L^1([0,T];L^2)} \Big)^2 \nu^{2-\beta/2-\alpha}
    \end{align*}
    Up to relabelling constants, using the fact that $T^{1-\beta/2-\alpha}$ can always be reabsorbed in the exponential upon modifying $\tilde K_2$, we deduce the validity of \eqref{eq:vanishing_viscosity_estimate}.    
\end{proof}

We now pass to consider smoothened approximations. Here, we take the same setting as in the beginning of Section \ref{subsec:tightness}: given filtered probability space $(\Omega,\mathcal{A},\mathcal{F}_t,\P)$
carrying a family $(W^k)_{k\ge1}$ of independent Brownian motions, and a family of radial mollifiers $\{\rho^\delta\}_{\delta>0}$, we look at solutions $\theta^\delta$ to \eqref{eq:reg_gSQG}. In light of Theorem \ref{thm:main}, we can now assume the solution $\theta$ associated to $(\theta_0,f)$ to be defined on the same probability space as well.

\begin{proposition}\label{prop:smooth_approximations}
Let $\alpha$, $\beta$, $p$ be as in Theorem \ref{thm:main}, $\theta_0\in L^1\cap L^p$, $f\in L^1_\loc([0,+\infty); L^1\cap L^p)$. Let $\theta^\delta$ be solutions to \eqref{eq:reg_gSQG} associated to $(\theta^\delta_0, f^\delta)$, $\theta$ the unique solution to \eqref{eq:stochastic_gSQG} associated to $(\theta_0,f)$. 
Then for any $T<\infty$, $m\in [1,\infty)$ and $\eps\in (0,1/2)$ it holds that
\begin{equation}\label{eq:convergence_mollifications}
    \lim_{\delta\to 0} \E\Big[\sup_{t\in [0,T]} \| \theta^\delta_t-\theta_t\|_{H^{-\eps}_w}^m\Big]=0.
\end{equation}
\end{proposition}

\begin{proof}
    The argument is basically an application of the Gyongy--Krylov Lemma, cf. \cite[Lemma 1.1]{GyoKry1996}.
    Indeed, consider the family of random variables $\{(\theta^\delta,\theta,W)\}_{\delta>0}$; arguing as in Section \ref{sec:existence}, in particular applying Lemma \ref{lem:ascoli_arzela}, this family is tight in $C([0,T];H^{-\eps}_w)^2\times C_t^\N$.
    One can then argue as in the proof of Theorem \ref{thm:weak_existence} to consider a subsequence $\{(\tilde\theta^{\delta_n},\tilde\theta^n,\tilde W^n)\}_n$, on a new probability space, which in convergence in that topology. However, by passing to the limit, one then produces two solutions $(\theta^1,\theta^2)=\lim_{n\to\infty}(\tilde\theta^{\delta_n},\tilde\theta^n)$ which both solve the SPDE \eqref{eq:stochastic_gSQG} for same noise $\tilde W=\lim_n \tilde W^n$, and same data $(\theta_0,f)$.
    In light of Theorem \ref{thm:main}, this implies that $\theta^1=\theta^2\eqqcolon \tilde \theta$ and so that
    \begin{align*}
        \lim_{n\to \infty} \E_\P\Big[\sup_{t\in [0,T]} \| \theta^{\delta_n}-\theta\|_{H^{-\eps}_w}\Big]
        = \lim_{n\to \infty} \E_{\tilde \P}\Big[\sup_{t\in [0,T]} \| \tilde\theta^{\delta_n}-\tilde \theta\|_{H^{-\eps}_w}\Big] = 0.
    \end{align*}
    As the argument holds for any subsequence $\{\theta^{\delta_n}\}_n$ one can extract, as well as for any $\eps>0$, we conclude that
    \begin{equation}\label{eq:conv_mollifications_intermediate}
        \lim_{\delta\to 0} \E_\P\Big[\sup_{t\in [0,T]} \| \theta^{\delta}-\theta\|_{H^{-\eps}_w}\Big] =0 \quad \forall\, \eps>0.
    \end{equation}
    From here, using the uniform bounds \eqref{eq:apriori_pathwise_bounds} in $L^1\cap L^p$, it is easy to upgrade the convergence in $H^{-\eps}_w$ to remove the weight and obtain \eqref{eq:convergence_mollifications}.
\end{proof}

\section{Well-posedness for linear transport equation with random drift}\label{sec:linear_random}

In this section we consider the linear SPDE
\begin{equation}\label{eq:new_transport}
    \dd \zeta + b\cdot\nabla\zeta \dd t + \circ \dd W\cdot\nabla \zeta= f \dd t
\end{equation}
where we assume we are given a filtered probability space $(\Omega,\mathcal{A},\mathcal{F}_t,\P)$ satisfying the standard assumptions and a $\mathcal{F}_t$-incompressible Kraichnan noise $W$ of parameter $\alpha\in (0,1)$ as defined in Subsection \ref{subsec:structure_noise} (hence with spatial covariance matrix from \eqref{eq:defn-Qalpha} and series representation \eqref{eq:noise_series_representation}).
We further assume that we are given some \emph{random} drift and forcing coefficients, denoted respectively by $b:\Omega\times [0,T]\times\R^2 \to \R^2$ and $f:\Omega\times [0,T]\times\R^2 \to \R$, which are possibly very rough;
we assume $b$ to be spatially divergence free and $b$, $f$ to be $\cF_t$-progressive.
In the following, given a map $\chi:\Omega\times [0,T]\times\R^2 \to \R^m$, we will refer to it as a $\cF_t$-\emph{progressive} process if it is measurable from $\cP\otimes \cB(\R^2)$ to $\cB(\R^m)$, where $\cP$ is the progressively measurable $\sigma$-algebra on $\Omega\times [0,T]$ induced by the filtration $\cF_t$.

\begin{definition}\label{def:transport_solution}
    Let $\zeta_0\in L^1\cap L^2$ be deterministic, $b$ and $f$ be $\cF_t$-progressive processes with $b$ divergence free and $b,\, f\in L^1_t L^1_\loc$ $\P$-a.s.; let $W$ be a $\mathcal{F}_t$-incompressible Kraichnan noise $W$ of parameter $\alpha\in (0,1)$.
    An $(L^1\cap L^2)$-valued solution to the stochastic linear transport equation \eqref{eq:new_transport} is a process $\zeta:\Omega\times [0,T]\to L^1\cap L^2$ such that:
    \begin{enumerate}
        \item[a)] $\zeta:\Omega\times [0,T]\to L^1\cap L^2$ is $\mathcal{F}_t$-progressive, $\P$-a.s.\ $\zeta\in L^2_t (L^1\cap L^2)$ and its trajectories $t\mapsto \zeta_t$ are $\P$-a.s.\ weakly continuous in the sense of distributions;
        \item[b)] $\P$-a.s. $b\,\zeta\in L^1_t L^1_\loc$;
        \item[c)] For any $\varphi\in C^\infty_c$, $\P$-a.s. for all $t\in [0,T]$ it holds
        \begin{equation}\label{eq:transport_solution}
            \langle \zeta_t,\varphi\rangle
            = \langle\zeta_0,\varphi\rangle + \int_0^t \big[ \langle b_r\zeta_r, \nabla\varphi\rangle + c_0 \langle \zeta_r, \Delta\varphi\rangle] \dd r + \int_0^t \langle \zeta_r \nabla\varphi,\dd W_r\rangle + \int_0^t \langle f_r, \varphi\rangle \dd r.
        \end{equation}
    \end{enumerate}
\end{definition}

Thanks to the above assumptions, all integrals appearing in \eqref{eq:transport_solution} are well-defined as either Lebesgue integrals or stochastic integrals via Lemma \ref{lem:stochastic_integrals}; cf. Remark \ref{rem:defn_weak_solution}.

At a technical level notice that, unlike the nonlinear case, here we need to fix a probability space once for all, as the randomness of $(b,f)$ prevents from employing naively tightness arguments. At the same time, we can leverage on the linearity of the SPDE to construct probabilistically strong solutions directly, see the proof of Proposition \ref{prop:linear_existence} below.

The main result of this section is the following:

\begin{theorem}\label{thm:linear_random_wellposed}
    Let $\gamma \in (0,2)$, $p\in(1,\infty)$ and $\alpha \in (0,1/2)$ be parameters satisfying 
	\begin{equation}\label{eq:linear_wellposed_coefficients}
		\gamma p<2,\qquad \alpha < \frac{\gamma +1}{2}-\frac{1}{p}.
    \end{equation}
    Let $(\Omega,\cA,\cF_t,\P)$ be a given standard filtered probability space carrying a $\mathcal{F}_t$-incompressible Kraichnan noise $W$ of parameter $\alpha\in (0,1)$.
    Assume that $b:\Omega\times[0,T]\times\R^2\to \R^2$ and $f:\Omega\times[0,T]\times\R^2\to \R$ are $\cF_t$-progressive processes such that
    \begin{align*}
        \nabla\cdot b\equiv 0, \quad
        b\in L^\infty_{\omega,t} \dot L^{\gamma,p}, \quad
        f\in L^\infty_\omega L^1_t (L^1\cap L^2).
    \end{align*}
    Then, for any deterministic initial condition $\zeta_0\in L^1\cap L^2$, there exists a strong solution $\zeta$ to \eqref{eq:new_transport}, with the property that
    \begin{equation}\label{eq:uniqueness_class_2}
        \zeta\in L^\infty_{\omega,t} L^\infty_t (L^1\cap L^2) \quad \forall\, T<\infty.
    \end{equation}
    Furthermore, pathwise uniqueness holds in the class of solutions $\zeta$ satisfying \eqref{eq:uniqueness_class_2}.
    Finally, given different initial data $\zeta_0^i$ and forcings $f^i$ satisfying the above assumptions, for $i=1,2$, the solutions $\zeta^i$ associated to $(\zeta^i_0,b,f^i)$ satisfy the pathwise $\P$-a.s. stability estimate
    \begin{equation}\label{eq:linear_random_stability}
        \sup_{t\in [0,T]} \| \zeta^1_t-\zeta^2_t\|_{L^1\cap L^2}
        \leq \| \zeta^1_0-\zeta^2_0\|_{L^1\cap L^2} + \int_0^T \| f^1_s-f^2_s\|_{L^1\cap L^2} \dd s\quad \forall\, t\in [0,T].
    \end{equation}
\end{theorem}

\subsection{Existence}\label{subsec:existence.random}

In the following statement, we write $b\in L^2_{\omega,t} L^2_\loc$ to mean that
\begin{align*}
    \E\bigg[\int_0^T \| b_t \,\mathbbm{1}_{|x|\leq R}\|_{L^2}^2 \dd t\bigg]<\infty \quad \forall\, R<\infty.
\end{align*}

\begin{proposition}\label{prop:linear_existence}
    Let $b$ be a divergence free progressive vector field such that $b\in L^2_{\omega,t} L^2_\loc$ and let $f$ be a progressive scalar-valued process such that $f\in L^2_\omega L^1_t (L^1\cap L^2)$.
    Then for any $\zeta_0\in L^1\cap L^2$, there exists a solution $\zeta$ to the SPDE \eqref{eq:new_transport}, in the sense of Definition \ref{def:transport_solution}, satisfying $\P$-a.s. the pathwise bound
    \begin{equation}\label{eq:linear_SPDE_bound}
        \sup_{t\in [0,T]} \| \zeta_t\|_{L^1\cap L^2}
        \leq \| \zeta_0\|_{L^1\cap L^2} + \int_0^T \| f_s\|_{L^1\cap L^2} \dd s.
    \end{equation}
    Moreover $\zeta$ is a probabilistically strong solution, in the sense that it is adapted to the filtration generated by $(W,b,f)$.
\end{proposition}

\begin{proof}
    The proof relies on classical arguments based on a priori estimates, weak compactness and linearity of the equation, see for instance \cite[Propositions 3.1 and 3.3]{GalLuo2023} and the references therein. For completeness, we give a self-contained proof.

    We first mollify everything as in Section \ref{subsec:tightness}: given $\{\rho^\delta\}_{\delta>0}$ standard mollifiers associated to a radially symmetric probability density $\rho\in C^\infty_c$, we set
    \begin{align*}
        \zeta^\delta_0:=\rho^\delta\ast \zeta_0, \quad
        f^\delta_t:= \rho^\delta\ast f_t,\quad
        W^\delta_t := \rho^\delta\ast W_t
        = \sum_k (\rho^\delta \ast \sigma_k) W^k_t =: \sum_k \sigma^\delta_k W^k_t.
    \end{align*}
    Correspondingly, we set $Q^\delta=Q\ast (\rho\ast \rho)^\delta$ and $c_\delta = {\rm Tr}(Q^\delta(0))/4$.
    For $b$, we need an additional cutoff since we only have local integrability; letting $\psi\in C^\infty_c$ be a nonnegative function such that $\psi(x)=1$ for $|x|\leq 1$ and $\psi(x)=0$ for $|x|\geq 2$, we set
    \begin{equation*}
        b^\delta_t:= \rho^\delta\ast(\psi(\delta\, \cdot\, ) b_t).
    \end{equation*}
    In this way $b^\delta\in L^2_{\omega,t} C^\infty_c$ and $b^\delta\to b$ in $L^2_{\omega,t} L^2_\loc$.    
    
    At fixed $\delta>0$, since all the coefficients are spatially smooth, we can invoke \cite[Theorem 6.1.9]{Kunita1990} for $(F_1,\ldots,F_d)=\int_0^\cdot b^\delta_s \dd s + W^\delta_t$, $F^{d+1}=0$ and $F^{d+2}=\int_0^\cdot f^\delta_s \dd s$ to deduce that (up to going back and forth between It\^o and Stratonovich formulations as usual), the transport SPDE
    \begin{align*}
        \dd \zeta^\delta + b^\delta\cdot \zeta^\delta \dd t + \nabla \zeta^\delta \cdot \dd W^\delta = [f^\delta + c_\delta \Delta\zeta^\delta] \dd t
    \end{align*}
    has a probabilistically strong solution $\zeta^\delta$, which satisfies the characteristics representation    
    \begin{equation}\label{eq:linear_SPDE_approx}
        \zeta^\delta_t(\Phi^\delta_t(x))=\zeta^\delta_0(x)+\int_0^t f^\delta_s(\Phi^\delta_s(x)) \dd s;
    \end{equation}
    here $\Phi^\delta_t(x)$ is the stochastic flow associated to the Stratonovich SDE with progressive coefficients
    \begin{align*}
        \dd \Phi^\delta_t(x) = b^\delta_t(\Phi^\delta_t(x)) \dd t + \sum_k \sigma^\delta_k(\Phi^\delta_t(x)) \circ \dd W^k_t, \quad \Phi^\delta_0(x)=x.
    \end{align*}
    Since $b^\delta$ and $W^\delta$ are divergence-free, by \cite[Lemma 4.3.1]{Kunita1990} the stochastic flow $\Phi^\delta$ preserves the Lebesgue measure;
    therefore taking $L^q$-norms on both sides of \eqref{eq:linear_SPDE_approx} for any given $q\in [1,\infty]$, by Minkowski's inequality we find the $\P$-a.s. pathwise estimate
    \begin{equation}\label{eq:linear_SPDE_apriori}
        \sup_{t\in [0,T]} \| \zeta^\delta_t\|_{L^q}
        \leq \| \zeta^\delta_0\|_{L^q} + \int_0^T \| f^\delta_s\|_{L^q} \dd s
        \leq \| \zeta_0\|_{L^q} + \int_0^T \| f_s\|_{L^q} \dd s.
    \end{equation}
    By applying \eqref{eq:linear_SPDE_apriori} for $q=1,2$ and taking expectation, by the assumption on $f$ we conclude that
    \begin{equation}\label{eq:linear_SPDE_apriori2}
        \sup_{\delta>0} \E\bigg[ \sup_{t\in [0,T]} \| \zeta^\delta_t\|_{L^1\cap L^2}^2\bigg]
        \leq \| \zeta_0\|_{L^1\cap L^2} + \E\bigg[\bigg(\int_0^T \| f_s\|_{L^1\cap L^2} \dd s\bigg)^2\bigg].
    \end{equation}

    Thanks to the uniform bound \eqref{eq:linear_SPDE_apriori2}, by Banach--Alaoglu we can extract a (not relabelled) subsequence which converges weakly in $L^2_\omega L^2_t L^2_x$ to some limit $\zeta$, which is still progressive since it is the limit of progressive processes.
    Notice that the collection of processes $\chi$ such that
    \begin{equation}\label{eq:linear_SPDE_exproof_eq1}
        \big\| \|\chi_t\|_{L^1\cap L^2} \big\|_{L^\infty([0,T])} \leq \| \zeta_0\|_{L^p} + \int_0^T \| f_s\|_{L^p} \dd s\quad \P\text{-a.s.}
    \end{equation}
    is a convex and closed set in $L^2_\omega L^2_t L^2_x$ (by Fatou's lemma), thus it is also closed in the weak topology; therefore the extracted limit $\zeta$ satisfies \eqref{eq:linear_SPDE_exproof_eq1} as well.
    
    We now want to pass to the limit  as $\delta\to 0$ in the SPDE for $\zeta^\delta$; recall that for any fixed $\varphi\in C^\infty_c$, it holds
    \begin{equation}\label{eq:linear_SPDE_approx_weak}
        \langle \zeta^\delta_\cdot, \varphi\rangle
        = \langle \zeta^\delta_0, \varphi\rangle + \int_0^\cdot \langle \zeta^\delta_s, \nabla \varphi\cdot b^\delta_s\rangle \dd s
        + \int_0^\cdot \langle \zeta^\delta_s \nabla \varphi, \dd W^\delta_s \rangle + \int_0^\cdot \langle  f^\delta_s,\varphi\rangle \dd s + c_\delta \int_0^\cdot \langle \zeta^\delta_s, \Delta \varphi \rangle \dd s.
    \end{equation}
    We plan to analyze each term separately and pass to the limit by weak-strong convergence arguments.
    
    It is clear that by construction $f^\delta\to f$ in $L^2_\omega L^1_t L^1_x$ and $\zeta^\delta_0\to \zeta_0$ in $L^1$, therefore the respective terms converge strongly. Since the operation of time integration is linear and strongly continuous from $L^2_\omega L^2_t$ to $L^2_\omega C^0_t$, it is also weakly continuous; combined with the fact that $c_\delta\to c$, this implies that
    \begin{align*}
        c_\delta \int_0^\cdot \langle \Delta \varphi, \zeta^\delta_s\rangle \dd s \rightharpoonup c_0 \int_0^\cdot \langle \Delta \varphi, \zeta_s\rangle \dd s\quad \text{weakly in } L^2_\omega C^0_t.
    \end{align*}
    Similarly, since $\nabla\varphi$ is compactly supported, by construction $\nabla\varphi\cdot b^\delta\to \nabla\varphi\cdot b$ strongly in $L^2_\omega L^2_t L^2_x$ and therefore by weak-strong convergence we can conclude that
    \begin{align*}
        \int_0^\cdot \langle \zeta^\delta_s, \nabla \varphi\cdot b^\delta_s\rangle \dd s \rightharpoonup \int_0^\cdot \langle \zeta_s, \nabla \varphi\cdot b_s\rangle \dd s\quad \text{weakly in } L^1_\omega C^0_t.
    \end{align*}
    It remains to handle the stochastic integral. To this end, by the same arguments as before, the linear map $\chi\mapsto \int_0^\cdot \langle \chi_s \nabla\varphi, \dd W_s\rangle$ is strongly continuous in the right topologies, and so preserves weak continuity, implying that
    \begin{align*}
        \int_0^\cdot \langle \zeta^\delta_s \nabla\varphi, \dd W_s\rangle \rightharpoonup \int_0^\cdot \langle \zeta_s \nabla\varphi, \dd W_s\rangle\quad \text{weakly in } L^2_\omega C^0_t.
    \end{align*}
    On the other hand, by adding and subtracting $\int_0^\cdot \langle \zeta^\delta_s \nabla\varphi, \dd W_s\rangle$, noticing that $W-W^\delta=(\delta_0-\rho^\delta)\ast W$ has covariance $\tilde Q^\delta := Q\ast (\delta_0-\rho^\delta)\ast (\delta_0-\rho^\delta)$, by the It\^o isometry of Lemma \ref{lem:stochastic_integrals} and Young's inequality we have
    \begin{align*}
        \E\bigg[\sup_{t\in [0,T]} \bigg| \int_0^t \langle\zeta^\delta_s \nabla\varphi, \dd (W_s-W^\delta_s)\rangle\bigg|^2\bigg]
        & \lesssim
        \E\bigg[\int_0^T \langle \tilde Q^\delta\ast (\zeta^\delta_s \nabla\varphi),\zeta^\delta_s \nabla\varphi\rangle^2 \dd s\bigg]\\
        & \lesssim \| \tilde Q^\delta\|_{L^\infty}^2\, 
        \E\bigg[\int_0^T \|\zeta^\delta_s \nabla\varphi\|_{L^1}^2 \dd s\bigg]\\
        & \lesssim \| \tilde Q^\delta\|_{L^\infty}^2\, \| \nabla\varphi\|_{L^\infty}^2 \E\bigg[\sup_{t\in [0,T]} \|\zeta^\delta_s \|_{L^1}^2 \bigg]\\
        & \lesssim \| \tilde Q^\delta\|_{L^\infty}^2\, \| \nabla\varphi\|_{L^\infty}^2 \Big( \|\zeta_0\|_{L^1}^2 + \| f\|_{L^2_\omega L^1_t L^1_x}^2\Big)
    \end{align*}
    thanks to the bound \eqref{eq:linear_SPDE_apriori}.
    Using the representation \eqref{eq:Q_fourier}, it's easy to check  that
    \begin{align*}
        \| \tilde Q^\delta\|_{L^\infty} \lesssim {\rm Tr}(\tilde Q^\delta(0)) = {\rm Tr} (Q(0)-2(\rho^\delta\ast Q)(0)+Q^\delta(0))\to 0 \text{ as } \delta\to 0
    \end{align*}
    by properties of mollifiers, since $Q$ is continuous.
    Combining the above facts we can therefore conclude that
    \begin{align*}
        \int_0^\cdot \langle \zeta^\delta_s \nabla\varphi, \dd W^\delta_s\rangle \rightharpoonup \int_0^\cdot \langle \zeta_s \nabla\varphi, \dd W_s\rangle\quad \text{weakly in } L^2_\omega C^0_t.
    \end{align*}
    
    We can now pass to the limit on the r.h.s. of \eqref{eq:linear_SPDE_approx_weak} to conclude that it converges weakly in $L^1_\omega C^0_t$ to a continuous process, which coincides in $L^2_\omega L^2_t$ with $\langle \zeta_\cdot,\varphi\rangle$, and is given by
    \begin{align*}
        \langle \zeta_0, \varphi\rangle + \int_0^\cdot \langle \zeta_s, \nabla \varphi\cdot b^\delta_s\rangle \dd s
        + \int_0^\cdot \langle \zeta_s \nabla \varphi, \dd W_s \rangle + \int_0^\cdot \langle  f_s,\varphi\rangle \dd s + c_0 \int_0^\cdot \langle \zeta_s, \Delta \varphi \rangle \dd s.
    \end{align*}
    In particular, $\langle \zeta_\cdot,\varphi\rangle$ has a continuous modification, exactly given by the r.h.s. of \eqref{eq:transport_solution}. As the argument hold for any $\varphi\in C^\infty_c$, we can now take a countable dense collection of them and apply to argument to construct a weakly continuous modification of $\zeta$, which is therefore the desired solution to \eqref{eq:new_transport}. Furthermore, by weak lower semicontinuity of $L^p$-norms, the essential supremum in the bound \eqref{eq:linear_SPDE_exproof_eq1} can be upgraded to a full supremum, yielding estimate \eqref{eq:linear_SPDE_bound}.
\end{proof}

\begin{remark}
    The assumptions on $b$ and $f$ prescribed in Proposition \ref{prop:linear_existence} to guarantee existence of solutions are by no means optimal, although sufficient for our purposes; at the price of making the proof more and more technical, the result admits several variants. For instance, given any $q\in [1,\infty]$, if $\zeta_0\in L^q$, $f\in L^2_\omega L^1_t L^q$ and $b$ is divergence free with $b\in L^2_\omega L^1_t L^{q'}_\loc$, then it's likely still possible to construct a solution $\zeta\in L^2_\omega L^\infty_t L^q$, as the a priori estimate \eqref{eq:linear_SPDE_apriori} is still available and all relevant terms have the correct complementary integrability.
    It is likely even possible to relax the $L^2_\omega$-integrability on $f$ and $b$, since the bound \eqref{eq:linear_SPDE_apriori} is pathwise in nature, up to invoking weak compactness in more refined spaces and using ucp convergence arguments for the stochastic integrals.
\end{remark}

\subsection{Uniqueness and stability}\label{subsec:uniqueness.random}
We can now complete the

\begin{proof}[Proof of Theorem \ref{thm:linear_random_wellposed}]
    Under our assumptions on $b$ and $f$, we can invoke Proposition \ref{prop:linear_existence} to infer the existence of a strong solution $\zeta$, satisfying the pathwise bound \eqref{eq:linear_SPDE_bound}; in particular, since $f\in L^\infty_t L^1_t (L^1\cap L^2)$, it follows from \eqref{eq:linear_SPDE_bound} that $\zeta\in L^\infty_{\omega, t} (L^1\cap L^2)$.

    Next we focus on pathwise uniqueness. As the argument strongly resembles the ones from Section \ref{subsec:proof_uniqueness}, we simply sketch some of its passages to avoid repetitions.
    
    Suppose we are given two distinct solutions $\zeta^1$, $\zeta^2$ satisfying \eqref{eq:uniqueness_class_2}; then $\zeta:=\zeta^1-\zeta^2\in L^\infty_{\omega,t}(L^1\cap L^2)$ solves the linear SPDE (in integral form, in an analytically weak sense)
    \begin{align*}
        \zeta_t = \int_0^t [-\nabla\cdot (b_r\, \zeta_r) + c_0 \Delta \zeta_r] \dd r - \int_0^t \nabla \zeta_r\cdot \dd W_r.
    \end{align*}
    In particular, going line-by-line through the proofs of Lemmas  \ref{lem:general_balance}-\ref{lem:energy_balance} for $h:=-\nabla\cdot(b\, \zeta)$, one can see that if there exist $\bar s\in (1/2,1)$, $s'\geq \bar s$ such that
    \begin{equation}\label{eq:linear_uniqueness_proof1}
        h\in L^\infty_{\omega,t} \dot H^{-s'}, \quad |\nabla|^{-2\bar s} h\in L^\infty_{\omega,t} L^2,
    \end{equation}
    then necessarily it must hold
    \begin{align*}
        \E[ \| \zeta_t\|_{\dot H^{-\bar s}}^2] - 2 \int_0^t \E[\langle |\nabla|^{-2\bar s} h_r, \zeta_r \rangle] \dd r
        = \lim_{\delta\to 0} \int_0^t \E[\langle \tilde\kappa^\delta \ast \zeta_r, \zeta_r \rangle] \dd r;
    \end{align*}
    here the kernel $\tilde\kappa^\delta$ is defined as in \eqref{eq:fourier_coercoice_term}-\eqref{eq:defn_Fdelta}, up to replacing $G=|\nabla|^{\beta-2}$ with $\tilde G=|\nabla|^{-2\bar s}$. From here, an analogue of Theorem \ref{thm:GGM} (cf. \cite{GGM2024}) implies the estimate
	\begin{equation}\label{eq:intermediate_uniqueness_estimate_lin}
	    \E[\| \zeta_t \|_{\dot H^{-\bar s}}^2] + K \int_0^t \E[\| \zeta_r\|_{\dot H^{-\bar s+1-\alpha}}^2] \dd r
		\leq 2 \int_0^t \E[|\langle b_r \cdot\nabla |\nabla|^{-2\bar s} \zeta_r, \zeta_r\rangle|] \dd r
        + C \int_0^t \E[\| \zeta_r\|_{\dot H^{-\bar s}}^2] \dd r
    \end{equation}
    which will ultimately yield the conclusion $\zeta\equiv 0$ by a Gr\"onwall argument (see later below).

    We start by verifying the claim \eqref{eq:linear_uniqueness_proof1} and identifying the relevant parameters $\bar s,s'$.
    By Lemma \ref{lem:bessel_embeddings}, $b_t \in \ L^{p'}$ for $\frac{1}{p'}=\frac1p -\frac\gamma2$ and, being $0<\alpha<\frac{\gamma+1}{2}-\frac1p$, it follows that $p'>2$.
    Then, due to $\zeta \in L^\infty_{\omega,t} L^2$, H\"older inequality implies that $b\,\zeta \in L^\infty_{\omega,t} L^{\hat p}$ for some $\hat p \in (1,2)$.
    In addition, by \cite[Corollary 1.39]{Bahouri2011} we have
	\begin{equation*}
		\|\div(b\,\zeta)\|_{\dot H^{-s'}}
            \le \|b\zeta\|_{\dot H^{-s'+1}}
            \lesssim \|b\,\zeta \|_{L^{\hat p}} \quad \text{for } s':=\frac{2}{\hat p}\in (1,2)
	\end{equation*}
    so that $h=-\div(b\,\zeta)\in L^\infty_{\omega,t} \dot H^{-s'}$.
    Next we define $\tilde{\alpha} \coloneqq \frac{\gamma+1}{2} - \frac1p$, which satisfies $0<\alpha<\tilde{\alpha}<\frac12$ thanks to assumption \eqref{eq:linear_wellposed_coefficients};  correspondingly, we choose $\bar s\in (0,1)$ satisfying  
	\begin{equation*}
		\frac12+\frac 1p -\frac \gamma2 <\bar s <\min \left(1,\frac 12+ \frac1p\right)\iff 1-\tilde\alpha <\bar s <\min \left(1,1-\tilde \alpha + \frac\gamma2\right).
	\end{equation*}
	We show that $|\nabla|^{-2\bar s}\div(b\,\zeta)\in L^\infty_{\omega,t}L^2$ by duality. Let $\varphi\in C^\infty_c$ and define $r \coloneqq \gamma + 2\bar s -1-\frac 2p>0$; we have
	\begin{align*}
		|\langle |\nabla|^{-2\bar s}\div(b_t\zeta_t), \varphi\rangle|
		= |\langle b_t\cdot\nabla|\nabla|^{-2\bar s}\varphi, \zeta_t \rangle|
		\leq \| b_t\cdot\nabla|\nabla|^{-2\bar s}\varphi \|_{\dot H^{r}} \| \zeta_t\|_{\dot H^{-r}}
	\end{align*}
	Being $0<r<\min(\gamma,2\bar s-1)<1$, we can apply Lemma \ref{lem:gen_product_sobolev_norm} (for $f_1= b_t$, $f_2= \nabla|\nabla|^{-2\bar s}\varphi$) and then Lemma \ref{lem:embedding_in_neg_Sobolev} (to $\zeta_t$) to find
	\begin{align*}
	 |\langle |\nabla|^{-2\bar s}\div(b_t\zeta_t), \varphi\rangle|
		&\lesssim \| b_t\|_{\dot L^{\gamma,p}} \| \nabla|\nabla|^{-2\bar s}\varphi \|_{\dot H^{2\bar s-1}}   \| \zeta_t\|_{\dot H^{-r}}\\
		&\lesssim \| b_t\|_{\dot L^{\gamma,p}} \|\varphi \|_{L^2}   \| \zeta_t\|_{L^1\cap L^2};
	\end{align*}
    overall, this proves the claim \eqref{eq:linear_uniqueness_proof1}, yielding \eqref{eq:intermediate_uniqueness_estimate_lin} for $\bar s$ as given above.

    To apply Gr\"onwall, we need a different estimate on $\nabla\cdot (b_t\, \zeta_t)$.
    Applying again duality similarly to above and Lemma \ref{lem:gen_product_sobolev_norm}, this time with
    \begin{align*}
        f_1=b_t, \quad f_2= \nabla|\nabla|^{-2\bar s} \zeta_t, \quad s_1=\gamma, \quad s_2=\bar s-\tilde \alpha, \quad s= \bar s-1+\tilde\alpha, \quad r_1=p,\quad r_2=2,
    \end{align*}
    one finds
	\begin{align*}
		|\langle b_t \cdot\nabla |\nabla|^{-2\bar s} \zeta_t, \zeta_t \rangle|
		& \le \| b_t \cdot\nabla |\nabla|^{-2\bar s} \zeta_t\|_{\dot H^{\bar s-1+\tilde\alpha}} \|\zeta \|_{\dot H^{-\bar s+1-\tilde\alpha}}\\
		& \lesssim \| b_t\|_{\dot L^{\gamma, p}} \|  \nabla |\nabla|^{-2\bar s} \zeta_t\|_{\dot H^{\bar s-\tilde \alpha}} \|\zeta \|_{\dot H^{-\bar s+1-\tilde\alpha}}\\
		&\lesssim \| b_t\|_{\dot L^{\gamma, p}} \|\zeta \|^2_{\dot H^{-\bar s+1-\tilde\alpha}};
	\end{align*}
	by interpolation (cf. \cite[Proposition 1.32]{Bahouri2011}) and Young inequalities, there exists $\lambda\in (0,1)$ such that 
	\begin{align*}
		\|\zeta\|^2_{\dot H^{-\bar s+1-\tilde\alpha}}
		\le \|\zeta\|_{\dot H^{-\bar s}}^{2\lambda} \|\zeta\|_{\dot H^{-\bar s+1-\alpha}}^{2(1-\lambda)} 
		\le \eta\|\zeta\|^2_{\dot H^{-\bar s+1-\alpha}}  + C_\eta \|\zeta\|^2_{\dot H^{-\bar s}}
	\end{align*}
	for any $\eta>0$ and a suitable constant $C_\eta>0$. Overall, upon choosing $\eta>0$ small enough, we conclude that
	\begin{equation*}
		|\langle b_t \cdot\nabla |\nabla|^{-2\bar s} \zeta_t, \zeta_t \rangle| \leq \frac{K}{2} \|\zeta\|^2_{\dot H^{-\bar s+1-\alpha}} + \tilde C_\eta \|\zeta\|^2_{\dot H^{-\bar s}};
	\end{equation*}
    inserting this estimate in \eqref{eq:intermediate_uniqueness_estimate_lin} we obtain
    \begin{equation*}
	    \E[\| \zeta_t \|_{\dot H^{-\bar s}}^2] + \frac{K}{2} \int_0^t \E[\| \zeta_r\|_{\dot H^{-\bar s+1-\alpha}}^2] \dd r
		\leq
        C' \int_0^t \E[\| \zeta_r\|_{\dot H^{-\bar s}}^2] \dd r
    \end{equation*}
    and so by Gr\"onwall lemma, we conclude that $\E[\| \zeta_t\|_{\dot H^{-\bar s}}^2]=0$ for all $t\geq 0$. Therefore $\P$-a.s. $\zeta=\zeta^1-\zeta^2\equiv 0$ and pathwise uniqueness holds.

    Now assume that $\zeta^1$, $\zeta^2$ are two solutions associated to $(\zeta^i_0,b,f^i)$, both belonging to $L^\infty_{\omega,t} (L^1\cap L^2)$; then by linearity $\zeta:=\zeta^1-\zeta^2\in L^\infty_{\omega,t} (L^1\cap L^2)$ is a solution to the linear SPDE associated to $(\zeta_0, b, f)$, for
    \begin{align*}
        \zeta_0\in L^1\cap L^2, \quad f=f^1- f^2\in L^\infty_\omega L^1_t (L^1\cap L^2).
    \end{align*}
    On the other hand, by Proposition \ref{prop:linear_existence}, for the same data we can construct another strong solution $\tilde \zeta$ to the same SPDE satisfying the pathwise estimate \eqref{eq:linear_SPDE_bound}; therefore $\tilde \zeta\in L^\infty_{\omega,t} (L^1\cap L^2)$, and by pathwise uniqueness $\zeta=\tilde \zeta$, from which the $\P$-a.s. pathwise bound \eqref{eq:linear_random_stability} follows.
\end{proof}

\begin{proof}[Proof of Theorem \ref{thm:main_linear}]
    Applying Theorem \ref{thm:linear_random_wellposed} on any finite interval $[0,n]$ provides a strong solution $\zeta^n$ therein; by pathwise uniqueness, $\zeta^{n+1}\vert_{t\in [0,n]} =\zeta^n$, therefore we can consistently define a unique process $\zeta$ on $[0,+\infty)$, which must coincide with $\zeta^n$ on $[0,n]$ for all $n\in\N$. Similarly, the global estimate \eqref{eq:intro_linear_stability} follows from the finite-time estimate \eqref{eq:linear_random_stability}, applied to any rational $T$, which is then extended to all continuous times thanks to the weak continuity of $t\mapsto \zeta_t$ and the  lower-semicontinuity of $L^p$-norms.
\end{proof}

\appendix
\section{Some useful lemmas}\label{app:analysis_lemmas}

This appendix comprises a collection of standard analytic results we used throughout the paper. Although our setting is on $\mathbb{R}^2$, for simplicity here we allow $\R^d$ for any $d\geq 2$.

Recall the notation $\Lambda=|\nabla|$ and the Bessel spaces $\dot L^{s,p}$ as defined in Section \ref{subsec:notation}.

\begin{lemma}\label{lem:bessel_embeddings}
    Let $s\in\R$, $1<p<\infty$. Then:
    \begin{enumerate}
        \item For any $r\in (0,s)$, $\Lambda^r$ is a bijection from $\dot L^{s,p}$ to $\dot L^{s-r,p}$, and $\| \Lambda^r f\|_{\dot L^{s-r,p}}= \| f\|_{\dot L^{s,p}}$.
        \item If $s<d$ and $q\in (p,\infty)$ satisfy
        \begin{align*}
            \frac{1}{p}-\frac{1}{q}=\frac{s}{d},
        \end{align*}
        then $\dot L^{s,p}\hookrightarrow L^q$ and $\| f\|_{L^q} \lesssim \| f\|_{\dot L^{s,p}}$  
    \end{enumerate}
\end{lemma}

\begin{proof}
    Point i) is a consequence of the semigroup property $\Lambda^s = \Lambda^{s-r} \Lambda^r$; point ii) follows from \cite[Theorem 1.2.3]{Grafakos2014}.
\end{proof}

\begin{lemma}\label{lem:embedding_in_neg_Sobolev}
$L^1(\R^2)\cap L^2(\R^2)$ continuously embeds in $\dot H^{-s}$ for any $s\in [0,1)$.
\end{lemma}
\begin{proof}
    Let $p\in (1,2]$, then by \cite[Corollary 1.39]{Bahouri2011} $L^p$ continuously embeds in $\dot H^\gamma$ for $\gamma= 1- \frac 2p$.
    By our assumption and interpolation, $f\in L^p$ for any $p\in [1,2]$, which concludes the proof.
\end{proof}

The next elementary statement, roughly informing us that we can split any $L^p$ function into a small $L^p$ part and a large $L^p\cap L^\infty$ one, is central to our main strategy to achieve uniqueness for ``critical'' $p$.

\begin{lemma}\label{lem:decomposition}
    Let $p\in [1,\infty)$, $\varphi\in L^p$; for any $R\in (0,+\infty)$, set $\varphi^{>R}(x):=\varphi(x) \mathbbm{1}_{|\varphi(x)|>R}$, $\varphi^{\leq R}(x):=\varphi(x)-\varphi^{>R}(x)$. Then it holds
    \begin{equation}\label{eq:decomposition_estim1}
        \| \varphi^{\leq R}\|_{L^q} \leq R^{1-\frac{p}{q}} \| \varphi\|_{L^p}^{\frac{p}{q}} \quad \forall\, q\in [p,\infty], R\in (0,\infty), \quad  \lim_{R\to\infty} \| \varphi^{>R} \|_{L^p}=0.
    \end{equation}
    Similarly, given a sequence $\{\varphi^n\}_n$ such that $\varphi^n\to \varphi$ in $L^p$, it holds
    \begin{equation}\label{eq:decomposition_estim2}
        \| \varphi^{n,\leq R}\|_{L^q} \leq R^{1-\frac{p}{q}} \| \varphi^n\|_{L^p}^{\frac{p}{q}} \quad \forall\, q\in [p,\infty], R\in (0,\infty), \quad  \lim_{R\to\infty} \sup_{n\in\N} \| \varphi^{n,>R} \|_{L^p}=0.
    \end{equation}
    For $p,R$ as above, $f\in L^1_t L^p$, similarly defining  $f^{>R}_t(x):=f_t(x) \mathbbm{1}_{|f_t(x)|>R}$, $f_t^{\leq R}(x):=f_t(x)-f^{>R}_t(x)$, it holds
    \begin{equation}\label{eq:decomposition_estim3}
        \| f_t^{\leq R}\|_{L^{\frac{q}{p}}_t L^q} \leq R^{1-\frac{p}{q}} \| f\|_{L^1_t L^p}^{\frac{p}{q}} \quad \forall\, q\in [p,\infty], R\in (0,\infty), \quad  \lim_{R\to\infty} \| f^{>R} \|_{L^1_t L^p}=0.
    \end{equation}
    Similarly, given a sequence $\{f^n\}_n$ such that $f^n\to f$ in $L^1_tL^p$, it holds
     \begin{equation}\label{eq:decomposition_estim4}
        \| f_t^{n,\leq R}\|_{L^{\frac{q}{p}}_t L^q} \leq R^{1-\frac{p}{q}} \| f^n\|_{L^1_t L^p}^{\frac{p}{q}} \quad \forall\, q\in [p,\infty], R\in (0,\infty), \quad  \lim_{R\to\infty} \sup_{n\in\N} \| f^{n,>R} \|_{L^1_t L^p}=0.
    \end{equation}
\end{lemma}

\begin{proof}
    It suffices to show \eqref{eq:decomposition_estim3}-\eqref{eq:decomposition_estim4}, since the analogues \eqref{eq:decomposition_estim1}-\eqref{eq:decomposition_estim2} follows by regarding $\varphi^n$ as time-dependent functions. The first bound in \eqref{eq:decomposition_estim3} comes from explicit computation, the second is a consequence of dominated convergence. The second statement in \eqref{eq:decomposition_estim4} is a consequence of Vitali's convergence theorem.
\end{proof}

Similarly to Section \ref{sec:existence}, for $d\geq 2$, let us consider the weight $w(x)=(1+|x|^2)^{-d/2-1}$; correspondingly, for $s\in\R$, we define the weighted Sobolev space $H^{s}_w$ as the closure of smooth functions under the norm\begin{equation*}
    \| \varphi\|_{H^{s}_w} := \| \varphi\, w\|_{H^{s}},
\end{equation*}
for any $s\in\R$.
It is immediate to verify that $H^{s}_w$ as defined is a Hilbert space.

\begin{lemma}\label{lem:weighted_negative_sobolev}
    Let $s\in\R$. The following hold:
    \begin{enumerate}
        \item[i)] The embedding $H^{s}\hookrightarrow H^{s}_w$ is bounded.
        \item[ii)] The embedding $H^{s}\hookrightarrow H^{s-\eps}_w$ is compact, for any $\eps>0$.
        \item[iii)] If $\varphi^n\to \varphi$ in $H^{s}_w$, then $\varphi^n\to \varphi$ in $H^{s}_\loc$.
        \item[iv)] For any $s_0<s_1$ and $\theta\in (0,1)$, setting $s_\theta = \theta s_0 + (1-\theta) s_1$, we have the interpolation estimate
        \begin{equation}\label{eq:interpolation_weighted_sobolev}
            \| \varphi\|_{H^{s_\theta}_w} \leq \| \varphi \|_{H^{s_0}_w}^\theta \, \| \varphi\|_{H^{s_1}_w}^{1-\theta}
        \end{equation}
    \end{enumerate}
\end{lemma}

\begin{proof}
    i) Notice that $w$ is smooth, $w\in W^{k,1}\cap W^{k,\infty}$ for all $k\in \N$. In particular, $w$ belongs to the Besov-H\"older space $B^s_{\infty,\infty}$ for any $s\in\R$. By standard paraproducts in Besov spaces \cite{Bahouri2011}, it follows that
    \begin{align*}
        \| \varphi w\|_{H^{s}} \lesssim \| \varphi\|_{H^{-s}} \| w\|_{B^{|s|+1}_{\infty,\infty}} \lesssim \| \varphi\|_{H^{s}}.
    \end{align*}

    ii) Let $\{\varphi^n\}_n$ be a bounded sequence in $H^{s}$; by weak compactness, without loss of generality we may assume that $\varphi^n\rightharpoonup \varphi$ for some $\varphi\in H^{s}$. We claim that $\varphi^n\to \varphi$ in $H^{s-\eps}_w$.

    To this end, let $h\in C^\infty_c$ be a smooth radial function such that $h\equiv 1$ on $B_1$ and $h\equiv 0$ on $B_2^c$; for any $R>0$, set $h^R:=h(\cdot/R)$, $w^{\leq R} := w\, h^R$, $w^{>R}:= w (1-h^R)$.
    Using Leibniz's formula, it's easy to verify that for any $k\in\N$ and $R>1$ it holds
    \begin{equation}\label{eq:cutoff_weoghts}
        \big\| D^k [w(1-h^R)]\big\|_{L^\infty} \lesssim R^{-d-2-k}.
    \end{equation}
    Now let us decompose
    \begin{align*}
        \varphi^n=\varphi^{n,\leq R}+\varphi^{n,>R}, \quad
        \varphi^{n,\leq R}:=\varphi^n h^R, \quad
        \varphi^{n,> R}:=\varphi^n (1-h^R),
    \end{align*}
    similarly for $\varphi=\varphi^{\leq R} + \varphi^{>R}$.
    For any fixed $R>1$, since $h^R$ is smooth, $\varphi^{n,\leq R}\rightharpoonup \varphi^{\leq R}$ in $H^{s}$; since they are also uniformly compactly supported, it follows that they converge strongly in $H^{s-\eps}$. By Point 1), strong convergence in $H^{s-\eps}_w$ holds as well.
    On the other hand, in light of \eqref{eq:cutoff_weoghts} and paraproducts, the tails $\varphi^{n,> R}$ can be made arbitrarily small in $H^{s}_w$ by taking $R$ large enough, since
    \begin{align*}
        \| \varphi^{n,> R}\|_{H^{s}_w}
        \lesssim \| \varphi^n\|_{H^{s}} \| w(1-h^R)\|_{B^{|s|+1}_{\infty,\infty}} \lesssim R^{-d-2}.
    \end{align*}
    Combining these facts, the desired convergence $\varphi^n\to \varphi$ in $H^{s-\eps}_w$ follows

    iii) Let $\varphi^n\to \varphi$ in $H^{s}_w$ and $\varphi\in C^\infty_c$. Since $\varphi$ is compactly supported, $\varphi w^{-1}$ is a smooth function; therefore again by paraproducts
    \begin{align*}
        \| (\varphi^n-\varphi)\,\psi\|_{H^{s}}
        = \| (\varphi^n-\varphi)\, w\, w^{-1}\,\psi\|_{H^{s}}
        \lesssim \| (\varphi^n-\varphi)\, w\|_{H^{s}} \| w^{-1}\,\psi\|_{B^{|s|+1}_\infty,\infty}
        \lesssim_\psi \| \varphi^n-\varphi\|_{H^{s}_w}
    \end{align*}
    which implies that $\varphi^n\, \psi\to \varphi\, \psi$ in $H^{s}$.
    
    iv) This follows immediately from the same interpolation inequality in $H^{s}$-spaces, applied to $\tilde \varphi = \varphi w$.
\end{proof}

\begin{lemma} \label{lem:gen_product_sobolev_norm}
	Let $i=1,2$, $0<s<s_i<2$ and $r_i\in(0,\infty)$ be such that
	\begin{align*}
		r_i s_i <2, \quad \frac{1}{r_1}+\frac{1}{r_2}-\frac12=\frac{s_1+s_2-s}{2}.
	\end{align*}
	Then, for any $f_1\in \dot L^{s_1,r_1}(\R^2)$ and $f_2\in \dot L^{s_2,r_2}(\R^2)$ it holds $f_1\, f_2\in \dot H^{s}(\R^2)$ with    \begin{equation}\label{eq:gen_product_fractional_inequality}
		\| f_1\, f_2\|_{\dot H^{s}} \lesssim \| f_1\|_{\dot L^{s_1,r_1}} \| f_2\|_{\dot L^{s_2,r_2}}.
	\end{equation}
\end{lemma}
	\begin{proof}
		Being $0<s<s_i<2$, $r_i\in(0,\infty)$ and $r_i s_i <2$, Theorem \ref{lem:bessel_embeddings} (with $d=2$) implies that $p_i, q_i$ defined by the relation
		\begin{equation*}
			\frac{1}{r_i}-\frac{1}{p_i}= \frac{s_i}{2}, \quad \frac{1}{r_i}-\frac{1}{q_i}= \frac{s-s_i}{2}
		\end{equation*}
		satisfy $p_i,q_i\in (r_i,\infty)$, and the following embeddings hold:
		\begin{align*}
			\|f_i\|_{L^{p_i}} \lesssim \|f_i\|_{\dot L^{s_i,r_i}}, \quad \| \Lambda^s f_i\|_{L^{q_i}} = \|f_i\|_{\dot L^{s,q_i}} \lesssim \|f_i\|_{\dot L^{s_i,r_i}}.
		\end{align*}
		The conclusion then follows by combining these embeddings with the fractional Leibniz rule, as stated in \cite[Theorem 1]{GraOh2014}, to find
		\begin{align*}
			\| f_1\, f_2\|_{\dot H^{s}}
            & = \| \Lambda^s( f_1 f_2)\|_{L^2}\\
			& \lesssim \| f_1\|_{L^{p_1}} \| \Lambda^s f_2\|_{L^{q_2}} + \| \Lambda^s f_1\|_{ L^{q_1}} \| f_2\|_{\dot L^{p_2}}\\
			& \lesssim \| f_1\|_{\dot L^{s_1,r_1}} \| f_2\|_{\dot L^{s_2,r_2}},
		\end{align*}
		where the condition $\frac 12 = \frac{1}{p_1} + \frac{1}{q_2}= \frac{1}{q_1} + \frac{1}{p_2}$ is satisfied thanks to our assumptions on $(r_i,s_i)$.
	\end{proof}

\textbf{Acknowledgements.}
    LG was supported by the SNSF Grant 182565 and by the Swiss State Secretariat for Education, Research and lnnovation (SERI) under contract number MB22.00034 through the project TENSE.  MM acknowledges support from the Italian Ministry of Research through the project PRIN 2022 ``Noise in fluid dynamics and related models'', project number I53D23002270006.
    MB, LG and MM are members of the Istituto Nazionale di Alta Matematica (INdAM), group GNAMPA.
    MB, LG and MM acknowledge support from the Istituto Nazionale di Alta Matematica (INdAM) through the project GNAMPA 2025 “Modelli stocastici in Fluidodinamica e Turbolenza”.
    \\

\textbf{Conflict of interest statement.} On behalf of all authors, the corresponding author states that there is no conflict of interest.

\textbf{Data availability statement.} This manuscript has no associated data.

\bibliography{biblio.bib}{}
\bibliographystyle{plain}

\end{document}